\documentclass{base_class}

\title[]{From harmonic analysis of translation-invariant valuations to geometric inequalities\\ for convex bodies
}

\author{Jan Kotrbat{\'y}}\thanks{JK was supported by FWF grant P31448-N35 and by DFG grant BE 2484/5-2.}
\author{Thomas Wannerer}\thanks{TW was supported by  DFG grant WA 3510/3-1.}
 
\email{kotrbaty@math.uni-frankfurt.de}
\email{thomas.wannerer@uni-jena.de}
\address{Institut für Mathematik, Goethe-Universit\"at Frankfurt, Robert-Mayer-Str. 10, 60629 Frankfurt, Germany}
\address{Fakult\"at f\"ur Mathematik und Informatik, Friedrich-Schiller-Universit\"at Jena, Ernst-Abbe-Platz 2, 07743 Jena, Germany}
\subjclass[2020]{}
\date{\today}

\begin{document}

\begin{abstract}
	The Alesker--Bernig--Schuster theorem asserts that each irreducible representation of the special orthogonal group appears with multiplicity at most one as a subrepresentation of the space of continuous translation-invariant valuations with fixed degree of homogeneity. Moreover, the theorem describes in terms of highest weights which irreducible representations appear with multiplicity one.
	In this paper, we present a refinement of this result, namely the explicit construction of a highest weight vector in each irreducible subrepresentation. We then describe how important natural operations on valuations (pullback, pushforward, Fourier transform, Lefschetz operator, Alesker--Poincaré pairing) act on these highest weight vectors. We use  this information to   prove the Hodge--Riemann relations for valuations in the case of Euclidean balls as reference bodies.
	Since special cases of the  Hodge--Riemann relations have recently been used to prove new geometric inequalities for convex bodies,  our work immediately extends  the scope of these inequalities. 
\end{abstract}

\maketitle
\tableofcontents

\section{Introduction}

A valuation is a finitely additive function defined on a fixed class  $\calS$ of geometric sets, i.e., a function $\phi \colon \mathcal S\to \CC$ satisfying
$$ \phi(A\cup B)= \phi(A)+\phi(B)-\phi(A\cap B)$$
whenever $A,B,A\cap B, A \cup B\in\calS$.  With its origins in the solution of Hilbert's third problem by Dehn and in Pick's theorem from plane geometry, the concept has---as a generalization of measure---long played an important role in convex geometry, see, e.g., the monographs by Gruber \cite{Gruber:CDG}, Hadwiger \cite{Hadwiger:Vorlesungen}, and Schneider \cite{Schneider:BM}. The (mixed) volume, the lattice point enumerator, and the Euler characteristic are fundamental examples of valuations, defined, respectively, on the classes of convex bodies, lattice polytopes, and finite unions of convex bodies.
 
With very few exceptions, the majority of the classical results on valuations are closely tied to the dissection theory of polytopes. Some of the most striking results in this direction are the proof of the $g$-theorem, which characterizes the possible numbers of faces of simple polytopes, and the proof of the Alexandrov--Fenchel inequality by McMullen \cite{McMullen:PolytopeAlgebra,McMullen:SimplePolytopes}.  In contrast, results on continuous valuations on convex bodies that are not directly implied by the corresponding discrete statements were scarce.  

This situation  changed completely with the groundbreaking work of Alesker  from the early 2000s \cite{Alesker:Irreducibility,Alesker:Product,Alesker:HLComplex,Alesker:Fourier}. The most important technical innovation introduced by Alesker was the discovery of a natural dense subspace $\Val^\infty(\RR^n)\subset \Val(\RR^n)$ of smooth valuations, see Section~\ref{s:val}  for precise definitions. Apart from much better analytical properties, the salient feature of the subspace of smooth valuations is the existence of a natural multiplicative structure, namely the Alesker product, which turns this space into a graded commutative algebra. The Alesker product and related algebraic  structures had a profound impact on  integral geometry, see, e.g., \cite{BernigFu:Hig,Fu:Unitary,AleskerFu:Barcelona,BFS,AbardiaBernig:AdditiveFlag,Faifman:Contact,Faifman:Crofton}. 

The starting point for this rapid development of valuation theory was Alesker's proof of a much stronger and more useful form  of  McMullen's conjecture \cite{Alesker:Irreducibility} known today as the irreducibility theorem. To describe this result, let us first point out that by a classical result of McMullen \cite{McMullen:EulerType} the space of continuous and translation-invariant valuations is naturally graded by the degree
$$ \Val(\RR^n)= \bigoplus_{r=0}^n\Val_r(\RR^n)$$
 and the parity of a valuation
 $$ \Val_r(\RR^n)=\Val_r^0(\RR^n)\oplus \Val_r^1(\RR^n).$$
Alesker's  irreducibility theorem  is the statement that the  natural action of the general linear group on each graded component is  irreducible. This is  a very strong property with far-reaching consequences. For example, it implies that smooth valuations can be represented by smooth differential forms. 

A related  question of great importance is how the space  $\Val(\RR^n)$ decomposes under the action of the special orthogonal group $\SO(n)$.
By the Peter--Weyl theorem from  abstract harmonic analysis,   each representation $E$ of the compact group $\SO(n)$  decomposes  into isotypic components
$\bigoplus_{\lambda} E_\lambda$, where the sum is over all equivalence classes of irreducible representations of $\SO(n)$. The question thus becomes which  $\lambda$'s  actually appear in the isotypic decomposition of $\Val_r(\RR^n)$ and with what multiplicities. A complete answer to this question is given  by the Alesker--Bernig--Schuster decomposition theorem \cite{ABS:Harmonic}. 

 Here and in the following, the equivalence classes of irreducible representations of  $\SO(n)$ are identified with their highest weights, see \S\ref{ss:SO} below.

 \begin{theorem}[Alesker--Bernig--Schuster]
	\label{thm:ABSintro}
	Let $ l=\lfloor\frac n2\rfloor$ and  $ 0\leq r \leq  l
	$.
	The non-trivial $\SO(n)$-types in $\Val_{r}(\RR^n)$ and $\Val_{n-r}(\RR^n)$ are the same and given by the following set of highest weights:
	\begin{align*}
		\Lambda_r=\{(\lambda_1,\dots,\lambda_l)\in\Lambda\mid  |\lambda_2|\leq 2,\ |\lambda_j|\neq1 \text{ for }1\leq j\leq r,\ \lambda_j=0\text{ for }r< j\leq l\}.
	\end{align*}
	Moreover, each  appears with multiplicity one.
\end{theorem}

Theorem~\ref{thm:ABSintro} is a useful tool for many problems in valuation theory and integral geometry.
Its typical usage  is to reduce a question concerning  the existence of valuations with certain properties  to a problem in representation theory which can be solved. 
 See \cite{ABS:Harmonic,Wannerer:Extendability,BDS:Dimension} for a small sample of such applications.

Drawing a comparison with the decomposition of functions on the unit sphere into spherical harmonics, it is obvious that the information provided by Theorem~\ref{thm:ABSintro}, however useful, is only very limited. For many applications of the theory of spherical harmonics it is of crucial importance to have a  manageable description of at least one spherical harmonic in each degree,  namely the zonal harmonic, see, e.g., \cite{Groemer:Harmonics}.
One of the principal contributions of this paper is to provide something just as useful, namely a simple  explicit description of a highest weight vector in each $\SO(n)$-type of $\Val_r(\RR^n)$. A different approach to making the Alesker--Bernig--Schuster decomposition theorem more explicit was recently explored by Saienko \cite{Saienko}.

A problem of particular interest where the information provided by Theorem~\ref{thm:ABSintro} is insufficient---and our  main motivation to undertake this work---is to prove  the  Hodge--Riemann relations for valuations with Euclidean balls as reference bodies. 
Denoting by $*\colon \Val^\infty(\RR^n)\times \Val^\infty(\RR^n)\to \Val^\infty(\RR^n)$ the convolution of valuations, let us first formulate the hard Lefschetz theorem and the Hodge--Riemann relations for general reference bodies.  

Let $\calK(\RR^n)$ denote the space of convex bodies in $\RR^n$ and write $V(K_1,\ldots, K_n)$  for the mixed volume  of $K_1,\dots,K_n\in\calK(\RR^n)$.

\begin{conjecture}
	\label{con}
	Let $0\leq r\leq\lfloor\frac n2\rfloor$. Consider $C_0,\dots,C_{n-2r}\in\calK(\RR^n)$ with smooth and strictly positively curved boundary and denote for $i=0,\dots,n-2r$
	\begin{align*}
		\mu_{C_i}(K)=V(K,\dots,K,C_i),\quad K\in\calK(\RR^n).
	\end{align*}Then the following properties hold:
	\begin{enuma}
		\item \textnormal{Hard Lefschetz theorem}. The mapping $\Val_{n-r}^\infty(\RR^n)\rightarrow\Val_r^\infty(\RR^n)$ given by
		\begin{align*}
			\phi\mapsto\phi*\mu_{C_1}*\cdots *\mu_{C_{n-2r}}
		\end{align*}
		is an isomorphism of topological vector spaces.
		\item \textnormal{Hodge--Riemann relations}. The sesquilinear form
		\begin{align*}
			(\phi,\psi)\mapsto (-1)^r\,\phi*\b\psi*\mu_{C_1}*\cdots *\mu_{C_{n-2r}}
		\end{align*}
		is positive definite on
		\begin{align*}
			\left\{\phi\in\Val^\infty_{n-r}(\RR^n)\mid \phi*\mu_{C_0}*\cdots *\mu_{C_{n-2r}}=0\right\}.
		\end{align*}
	\end{enuma}
\end{conjecture}

Conjecture \ref{con} has first been published  only last year by the first-named author \cite{Kotrbaty:HR}, but it was certainly  considered before  as both statements are formally analogous to theorems in K\"ahler geometry. Moreover, they can be viewed as the  continuous analogs of two theorems of McMullen \cite{McMullen:SimplePolytopes} on the polytope algebra.  However,  in contrast to these cases, the algebra $\Val^\infty(\RR^n)$ is infinite-dimensional which makes numerous aspects of the problem somewhat more delicate.

Still, several instances of the conjecture are known to be true. The hard Lefschetz theorem was proved for Euclidean balls as reference bodies first by Alesker~\cite{Alesker:HLComplex} for even valuations and subsequently by Bernig and Br\"ocker \cite{BernigBroecker:Rumin} without this restriction. The authors \cite{KotrbatyWannerer:MixedHR} have recently established the  hard Lefschetz theorem for $r=1$. The Hodge--Riemann relations were proved for Euclidean balls as reference bodies by the first-named author \cite{Kotrbaty:HR} for even valuations and  for $r=1$. The authors \cite{KotrbatyWannerer:MixedHR} have also recently established the  Hodge--Riemann relations  for $r=1$.

The Hodge--Riemann relations imply both classical and new geometric inequalities for convex bodies. For first steps in this direction see \cite{Alesker:Kotrbaty, KotrbatyWannerer:MixedHR, Kotrbaty:HR}  where in particular the Alexandrov--Fenchel inequality together with a novel improvement for lower-dimensional convex bodies have been deduced from the Hodge--Riemann relations.

\subsection{Main results} Our contributions in this work  are threefold. First, 
we will refine the Alesker--Bernig--Schuster theorem by explicitly constructing a non-trivial  highest weight vector for each $\SO(n)$-type  in $\Val_r(\RR^n)$. 
Second, we will describe the action of a number of  important natural operations on valuations, namely  pullback, pushforward, Fourier transform,  Lefschetz operator, and Alesker--Poincar\'e pairing, on these highest weight vectors.  
Third,  we will accomplish our original goal, proving the Hodge--Riemann relations (Conjecture~\ref{con}(b)) for Euclidean balls as reference bodies.

Let us begin by outlining the construction of the highest weight vectors. Recall that in each irreducible representation $V$ of $\SO(n)$ there exists a (unique up to the choice of certain Lie-theoretic data) one-dimensional subspace $V_\lambda\subset V$ corresponding to the highest weight of $V$. Since according to the Alesker--Bernig--Schuster theorem the decomposition into $\SO(n)$-types is multiplicity-free, we thus have to construct precisely one highest weight vector for each $\SO(n)$-type  appearing in the isotypic decomposition  of $\Val_r(\RR^n)$.

Throughout the paper, we will work in Euclidean space of  dimension $n\geq 2$ and we will denote $l=\left\lfloor\frac n2\right\rfloor$.  For  the sake of clarity, let  us define
\begin{align*}
	\lambda_{k,m}=\begin{cases}(m,\underbrace{2,\dots,2}_{k-1},0,\dots, 0)\in\ZZ^l&\text{for }k=1,\dots,l,\\(m,2,\dots,2,-2)\in\ZZ^l&\text{for }k=-l.\end{cases}
\end{align*}
With this notation,  the Alesker--Bernig--Schuster theorem asserts that the set of highest weights appearing in  $\Val_r(\RR^n)$ is
\begin{align*}
	\Lambda_r=
	\begin{cases}
		\{\lambda_{k,m}\mid m\geq2, 1\leq k\leq l\}\cup\{0\}\cup\{\lambda_{-l,m}\mid m\geq 2\}&\text{if }n=2l\text{ and }r=l,\\[1ex]
		\{\lambda_{k,m}\mid m\geq2, 1\leq k\leq r\}\cup\{0\}&\text{otherwise}.
	\end{cases}
\end{align*}

Let $v_n$  be the volume of the $n$-dimensional Euclidean unit ball $D^n$ and $s_n$ the volume of the $n$-dimensional Euclidean unit sphere  $S^n$. We write $S\RR^n=\RR^n\times S^{n-1}$ for the sphere bundle of $\RR^n$ and denote by $N(K)$ the normal cycle of a convex body $K\in\calK(\RR^n)$. As a set, the normal cycle is the disjoint union of the outward unit normals  to $K$.
\begin{theorem}\label{thm:A}
For any $r,k,m\in\NN$ with $r\leq n-1$, $k\leq \min\{r,n-r\}$, and $m\geq2$, let $\omega_{r,k,m}\in\Omega^{n-1}(S\RR^n)$ be given by formula \eqref{eq:HWomega} below. The smooth valuation
\begin{align*}
\phi_{r,k,m}(K)=\frac{(\sqrt{-1})^{\lfloor\frac n2\rfloor}(\sqrt 2)^{m-2}}{s_{n+m-r-3}}\int_{N(K)}\omega_{r,k,m},\quad K\in\calK(\RR^n),
\end{align*}
is a non-trivial highest weight vector of weight $\lambda_{k,m}$ of the $\SO(n)$-representation $\Val_{r}(\RR^n)$.
\end{theorem}

Observe that there are highest weights appearing in $\Val_r(\RR^n)$ that are  formally not covered by  Theorem~\ref{thm:A}, namely the trivial highest weight and $\lambda_{-l,m}$. However,  applying  to $\phi_{l,l,m}$ the reflection in the coordinate hyperplane $e_n^\perp$  yields the highest weight vectors for the latter. 
 As for the former, the valuations corresponding to the trivial highest weight are precisely the intrinsic volumes, which have been  intensively studied   in convex geometry. Since they  do not fit naturally into the family of differential forms $\omega_{r,k,m}$  we introduce and since many much simpler and geometrically more appealing  descriptions of the instrinsic volumes are known (see, e.g., \cite{Schneider:BM}),  we have preferred to omit them from Theorem~\ref{thm:A}. 

 A brilliant insight behind much of the recent progress in integral geometry is the observation of Fu \cite{Fu:Unitary}  that problems in the spirit of the kinematic formula of Blaschke--Chern--Federer \cite{Blaschke:Integralgeometrie,Chern:Kinematic,Federer:CurvatureMeasures} can be transformed into the problem of evaluating the product, the convolution, and the Fourier transform of invariant valuations.  More recently, it was discovered that also problems concerning geometric inequalities for convex bodies can be reformulated in the  algebraic language of valuation theory \cite{Alesker:Kotrbaty,KotrbatyWannerer:MixedHR,Kotrbaty:HR}.  It should therefore come as no surprise that  computing the algebraic structures explicitly is typically a challenge. 
 In this connection, the main point of Theorem~\ref{thm:A} is that it provides an expression for the highest weight vectors that is simple enough to allow the evaluation  of most of the natural  operations on valuations.

Since in the following theorem we consider valuations both on $\calK(\RR^n)$ and on  $\calK(\RR^{n-1})$, we write $\phi^{(n)}_{r,k,m}$ instead of $\phi_{r,k,m}$ to reduce the risk of confusion.

\begin{theorem}\label{thm:B} For any $r,k,m\in\NN$ with $r\leq n-1$, $k\leq \min\{r,n-r\}$, and $m\geq2$,
the highest weight vectors 
constructed in Theorem~\ref{thm:A} satisfy the following properties:
\begin{enuma}
	\item \textnormal{Alesker--Poincar\'e pairing}.
	$$	\b{\phi_{r,k,m}^{(n)}}*\phi_{n-r,k,m}^{(n)}= (-1)^k(m+k-1)(n+m-k){n-2k\choose r-k}\frac{v_{n+2m-2}}{v_{n+m-r-2}v_{r+m-2}s_{2m-3}}.$$
	\item \textnormal{Pullback along the inclusion $\iota \colon \RR^{n-1}\to \RR^{n-1}\oplus \RR$.}
	$$ 	\iota^*\phi_{r,k,m}^{(n)}=\begin{cases}\phi_{r,k,m}^{(n-1)}&\text{if }r<n-1\text{ and } k<n-r, \\[1ex] \frac 1 2\phi_{k,k-1,m}^{(n-1)} & \text{if } k=\frac n2, \\[1ex]0&\text{otherwise}.\end{cases}$$
	\item \textnormal{Pushforward along the projection $\pi\colon \RR^{n-1}\oplus \RR \to \RR^{n-1}$}.
	$$ 	\pi_* \phi_{r,k,m}^{(n)} =\begin{cases}\phi_{r-1,k,m}^{(n-1)}&\text{if }k<r,\\[1ex]-\frac 12 \phi_{k-1,k-1,m}^{(n-1)}&\text{if }k=\frac n2,\\[1ex] 0&\text{otherwise}.\end{cases}$$
	\item \textnormal{Fourier transform}.
	$$ 	\FF\phi_{r,k,m}^{(n)}=(-1)^{k-1}(\sqrt{-1})^m\phi_{n-r,k,m}^{(n)}.$$
	\item \textnormal{Lefschetz operator}.
	$$ 	\Lambda\phi_{r,k,m}^{(n)}=\begin{cases}(n-r-k+1)\frac{v_{n+m-r-1}}{v_{n+m-r-2}}\phi_{r-1,k,m}^{(n)} & \text{if }k<r,\\[1ex] 0 & \text{if }k=r.\end{cases}$$
\end{enuma}

\end{theorem}
Several special cases and consequences of Theorem~\ref{thm:B} have already appeared---sometimes in different guises---in the literature.
Regarding the Fourier transform, the special cases  $k=1$ and $k=l$ for even $m$ have been treated by Bernig and Hug 
\cite[Theorem 1]{BernigHug:Tensor}  and Bernig and Solanes \cite[Corollary 2.5]{BernigSolanes:Kinematic}, respectively.  Moreover, in \cite[Proposition 2.8]{BernigSolanes:Kinematic} the action of $\Lambda^{n-2r}\circ \FF$ on highest weight vectors was determined for even $m$ and $n$. 
For $k=1$, our result on the Alesker--Poincar\'e pairing  is equivalent to \cite[Proposition~4.11]{BernigHug:Tensor}.  \cite[Theorem 3.1]{GoodeyHugWeil:AreaMeas} by Goodey, Hug, and Weil is equivalent to the action of  $\FF\circ \Lambda\circ \FF$ on highest weight vectors with $k=1$.

We expect that Theorems~\ref{thm:A} and \ref{thm:B} will open up  new avenues in convex and  integral geometry. 
In this paper we present  two  applications. First, we give a new proof of the hard Lefschetz theorem for Euclidean balls as reference bodies. As discussed above, in this case the hard Lefschetz theorem is already known to be true;  
the  point here is that Theorem~\ref{thm:B} yields a short and transparent proof. Second,  and here our results seem indeed indispensable, we deduce from Theorem~\ref{thm:B} the Hodge--Riemann relations for Euclidean balls as reference bodies. Altogether,

\begin{theorem}\label{thm:C}
Conjecture \ref{con} is true if $C_0=\cdots=C_{n-2r}=D^n$.
\end{theorem}

As pointed out to us by R.~van Handel, combining Theorem~\ref{thm:C} with the case $r=1$ of the Hodge--Riemann relations recently proved by the authors \cite{KotrbatyWannerer:MixedHR} implies  the validity of Conjecture~\ref{con} in yet another special case. 

\begin{corollary}\label{cor:D}For arbitrary $C_0$, 
Conjecture \ref{con} is true if $r=2$ and  $C_1=\cdots= C_{n-4}=D^n$. In particular, Conjecture~\ref{con} holds for $n\leq 4$.
\end{corollary}

Let us  point out that Theorem~\ref{thm:C} and Corollary~\ref{cor:D} immediately extend the scope of the geometric inequalities that Alesker \cite{Alesker:Kotrbaty} has recently deduced from the Hodge--Riemann relations thus providing in particular a continuous analog to previously established combinatorial inequalities  \cite{McMullen:SimplePolytopes,Timorin:Analogue,vanHandel:Shepard}.

\subsection{Further discussion}

The differential forms $\omega_{r,k,m}$ of Theorem~\ref{thm:A} are defined  as certain determinantal expressions in the differentials of  complexified  coordinate functions on $\RR^n\times \RR^{n}\supset S\RR^n$.  Apart from finding these differential forms, one major challenge we faced  at the beginning of this  work was to  manipulate these bulky determinantal expressions without too much pain. The notion of double form---a concept developed by de Rham \cite{deRham} and  a key tool in the work of Gray \cite{Gray:TubesChernClasses,Gray:Tubes} in integral geometry---turned out be perfectly suited for our purposes. Double forms  allow us to cleanly split up the proof of the  main identities  into a series of simpler auxiliary statements.

Another point that deserves mention is that the proof of the special case of Theorem~\ref{thm:B}(d) established by  Bernig and Hug \cite{BernigHug:Tensor}  rested on the idea of using the action of  general linear group to move  between different $\SO(n)$-types. 
 Much to our surprise, our proof of the general result is shorter and   uses only the fact that the Fourier transform commutes with the action of $\SO(n)$ and  is compatible with the operations of pullback and pushforward.

\subsection{Acknowledgments}

The authors wish to thank Ramon van Handel for illuminating discussions on the Hodge--Riemann relations.

\section{Representation theory}

  \subsection{Abstract harmonic analysis}
  
  \label{ss:AHA}

 Let us begin by collecting general information about smooth vectors, the decomposition into $K$-types, and  the convergence of Fourier series. Our reference for this material are books by Warner \cite{Warner:HarmonicI} and Knapp \cite{Knapp:Overview} and an article by Sugiura \cite{Sugiura:Fourier}.

 Let $G$ be a Lie group and $\rho$ be a continuous representation of $G$ on a Fr\'echet space $E$. An element $v\in E$ is called a smooth vector if the map $G\to E$, $g\mapsto \rho(g)v$ is smooth. The  invariant subspace of smooth vectors  is denoted $ E^\infty$ and is equipped with a natural Fr\'echet space topology. Moreover, $(E^\infty)^\infty= E^\infty$. 
 
 Let $K\subset G$ be a compact subgroup with Haar  probability measure $dk$. Denote by $\widehat K$ the set of all $K$-types, i.e., equivalence classes of  irreducible representations of $K$. Let further $\chi_\lambda$ denote the character and $d(\lambda)$ the (finite) dimension of $\lambda\in \widehat K$. Then
\begin{align*}
\pi_\lambda = d(\lambda) \int_{K} \overline{\chi_\lambda(k)} \rho(k) dk
\end{align*}
is a continuous projection of $E$ onto the $\lambda$-isotypic component $E_\lambda\subset E$, i.e., the  closure of  the span  of all irreducible subrepresentations of $E$ of type $\lambda$.  Smooth vectors satisfy the following important property.

\begin{theorem}[Harish-Chandra]
\label{thm:H-Ch}
Let $v\in E^\infty$. Then the Fourier series
\begin{align*}
\sum_{\lambda\in \widehat{K}}  \pi_\lambda v 
\end{align*}
convergences absolutely to $v$. 
\end{theorem}

We  will need more precise information on the decay of the Fourier coefficients.  To this end, let us assume $K$ is connected in which case $\widehat K$ is in one-to-one correspondence with the set of highest weights. Given a maximal torus $T\subset K$ with Lie algebra $\ftt$ and a system of positive roots $\Delta^+\subset\ftt^*_\CC$, the complexification of the Lie algebra $\fkk$ of $K$ decomposes as
\begin{align*}
\fkk_\CC=\fnn^-\oplus\ftt_\CC\oplus\fnn^+,
\end{align*}
where $\fnn^\pm=\bigoplus_{\alpha\in\Delta^+}\fkk_{\pm\alpha}$ is the direct sum of positive (negative) root spaces. Let $(V,\sigma)\in\lambda$ be an irreducible representation of $K$. There exist $\lambda\in\ftt_\CC^*$, called the highest weight, and a one-dimensional subspace $V_\lambda\subset V$, both unique, such that
\begin{align*}
d\sigma(H)v&=\lambda(H)v \quad \text{and}\quad\sigma(X)v=0
\end{align*}
holds for any $H\in\ftt_\CC$, $X\in\fnn^+$, and $v\in V_\lambda$. In fact, the highest weight characterizes the corresponding $K$-type uniquely, which justifies the abuse of notation. From now on these two meanings of $\lambda$ will be used interchangeably. The elements of $V_\lambda$ are referred to as highest weight vectors.

Fix an $\Ad(K)$-invariant complex bilinear form on $\fkk_\CC$ such that its restriction to $\sqrt{-1}\ftt$ is real and positive definite. For example, if   $K\subset U(n)$, then the trace form $B_0(X,Y)=\trace(XY)$ has this property. For every $\lambda\in \ftt_\CC^*$ which is  imaginary on $\ftt$ (such as a weight) define $H_\lambda \in \sqrt{-1}\ftt$ by 
$ \lambda(H)= \langle H, H_\lambda\rangle $ for all $H\in \sqrt{-1}\ftt$. Observe that $\langle \lambda, \mu\rangle = \langle  H_\lambda, H_\mu\rangle$  defines a positive definite inner product on $\sqrt{-1}\ftt^*\subset \ftt_\CC^*$ and denote the induced norm by $\norm\Cdot$. The following theorem is then essentially due to Sugiura \cite{Sugiura:Fourier}; although the results of \cite{Sugiura:Fourier} concern functions on $K$, the argument generalizes to an arbitrary continuous representation. For the sake of completeness, we will prove this more general statement here.

\begin{theorem}[Sugiura]  \label{thm:decay} Let $v\in E^\infty$. Then for every integer $q\geq 0$ and every continuous seminorm $|\cdot|$ on $E^\infty$, 
\begin{align*} \norm \lambda ^{2q}|\pi_\lambda v| \to 0 \quad \text{as } \norm\lambda \to \infty. 
\end{align*}
\end{theorem}

\begin{proof}
First, consider the extension of the representation $d\rho$ of $\fkk$ to the universal enveloping algebra $U(\fkk)$ and the Casimir operator $\Delta= \sum_{i,j=1}^r g^{ij} X_iX_j\in U(\fkk)$, where $X_1,\ldots, X_r$ is a basis of $\fkk_\CC$ and $(g^{ij})=(g_{ij})^{-1}$ with $g_{ij}=\langle X_i,X_j\rangle$. It is well known that $d\rho(\Delta)$ acts on each isotypic component $E_\lambda$ by multiplication by $\langle \lambda ,\lambda +2\delta\rangle$, where $\delta=\frac 12 \sum_{\alpha\in \Delta^+}\alpha$, and that
\begin{align*}
\langle \lambda ,\lambda + 2\delta\rangle\geq\norm\lambda^2,
\end{align*}
see, e.g., Lemma 1.1 and (1.10), respectively, in \cite{Sugiura:Fourier}. Second, because $K$ is compact, the family of operators $\{\rho(k)\mid k\in K\}$ is equicontinuous. Consequently, for every continuous seminorm $|\Cdot|$ on $E^\infty$, there exists another continuous seminorm $|\Cdot|_0$ such that 
\begin{align*}
|\rho(k)v| \leq |v|_0
\end{align*}
holds for all $k\in K$ and $v\in E^\infty$. Third, recall that since the absolute value of the trace of a unitary matrix cannot exceed its rank, we have for any $\lambda\in\widehat K$ and $k\in K$
\begin{align*}
|\chi_\lambda(k)|\leq d(\lambda).
\end{align*}
Finally, according to Weyl's dimension formula, there exist $c,j\in\NN$ with
\begin{align*}
d(\lambda)\leq c\norm\lambda^j,\quad \lambda\in\widehat K,
\end{align*}
see (1.17) in \cite{Sugiura:Fourier}. Combining these facts together, we obtain for every $\lambda\in\widehat K$, $i\in\NN$, and $v\in E^\infty$
\begin{align*} 
\norm\lambda^{2i}|\pi_\lambda v | &\leq\langle \lambda ,\lambda +2\delta\rangle ^i|\pi_\lambda v |  \\
&  = \left| d\rho(\Delta)^i\pi_\lambda v \right| \\
&= \left| \pi_\lambda d\rho(\Delta)^i v \right|\\
& =d(\lambda) \left| \int_K \overline{\chi_\lambda(k)}   \rho(k) d\rho (\Delta)^i v \,dk\right|\\
&\leq d(\lambda)^2 \left| d\rho(\Delta)^i v\right|_0\\
&\leq c^2\norm\lambda^{2j} \left| d\rho(\Delta)^i v\right|_0,
\end{align*}
and choosing $i\in\NN$ such that $i>j+q$ thus yields the result.
\end{proof}

Finally, we will need the following consequence of Schur's lemma.

\begin{proposition}
\label{pro:Schur}
	Let $V$ and $W$ be irreducible representations of $K$ and $q\maps{V\times W}{\CC}$
 a $K$-invariant sesquilinear form.
\begin{enuma}	
	\item If $V\ncong W$ then $q=0$ identically.
	\item Assume $V=W$. If there is $x\in V$ such that $q(x,x)>0$ then $q$ is a Hermitian inner product on $V$.
\end{enuma}
\end{proposition}
\begin{proof}
Observe first that there exists a  $K$-invariant Hermitian inner product on $ W$: Take any Hermitian inner product $h$ on $W$ and set
\begin{align*}
h_K(y,z)= \int_K h(ky,kz) dk.
\end{align*}
In particular, $h_K$ induces a  $K$-equivariant isomorphism  $\b W^*\cong W$. Consequently, we have $q\in(V^*\otimes \b W^*)^K\cong \Hom_K(V, W)$. By Schur's lemma,
\begin{align*}
\dim\Hom_K(V, W)=\begin{cases}1&\text{if }V\cong W,\\0&\text{otherwise}.\end{cases}
\end{align*}
This directly implies item (a). As for item (b), because $h_K\in(V^*\otimes \b V^*)^K$, there is $c\in\CC$ with $q=ch_K$. If $q(x,x)>0$ for some $x\in V$ then $c>0$ and hence $q$ is a Hermitian inner product itself.
\end{proof}

\subsection{Special orthogonal group}

\label{ss:SO}

Later we will apply the above general theory to the Fr\'echet  spaces of continuous and smooth valuations and the special orthogonal group $K=\SO(n)$. To this end, let us now take a closer look at the latter, following Knapp's monograph \cite{Knapp:Overview}. We distinguish two cases according to the parity of $n$.

Assume first $n=2l$ and let us fix the maximal torus $T\subset\SO(2l)$ consisting of matrices of the form 
\begin{align*}
\bdiag\big(R(t_1),\dots,R(t_l)\big),\quad\text{where}\quad R(t)=\begin{pmatrix}\cos t&-\sin t\\ \sin t&\cos t\end{pmatrix}.
\end{align*}
We  further denote by $E_{i,j}$ the $(i,j)$-th element of the standard basis of $\CC^{n\times n}$ and choose a basis $H_1,\dots, H_l$ of $\ftt_\CC$ and  the dual basis $\epsilon_1,\dots\epsilon_l$ of $\ftt^*_\CC$ as follows:
\begin{align*} 
 H_i&= \sqrt{-1}E_{2i-1,2i}-\sqrt{-1}E_{2i,2i-1},\\
\epsilon_i(H_j)&= \delta_{ij}.
\end{align*}
 We introduce  an ordering on the roots by declaring a linear function $\lambda\in \ftt_\CC^*$ that is real on $\sqrt{-1}\ftt$ to be positive if 
$ \lambda(H_1)>0$, or $\lambda(H_1)=0$ and $\lambda(H_2)>0$, \ldots,  or $\lambda(H_1)=\cdots = \lambda(H_{n-1})=0$ and $\lambda(H_n)>0$. The set of positive roots is then 
\begin{align*}
\Delta^+=\{\epsilon_i\pm\epsilon_j\mid1\leq i<j\leq l\}.
\end{align*}
Finally,  set
\begin{align*}
x_-=\begin{pmatrix}1&\sqrt{-1}\\-\sqrt{-1}&1\end{pmatrix}\quad\text{and}\quad x_+=\begin{pmatrix}1&-\sqrt{-1}\\-\sqrt{-1}&-1\end{pmatrix}.
\end{align*}
For $i<j$ the root space $\fkk_{\epsilon_i\pm\epsilon_j}\subset\fkk_\CC$ is spanned by the $l\times l$ block matrix $X_{\epsilon_i \pm\epsilon_j }$ given (block-wise) as
\begin{align}
\label{eq:Xepsilonij}
\left(X_{\epsilon_i \pm\epsilon_j }\right)_{a,b}=\begin{cases}\frac12x_\pm&\text{if }a=i\text{ and }b=j,\\[1ex] -\frac12 x^t_\pm&\text{if }a=j\text{ and }b=i,\\[1ex]0&\text{otherwise}.\end{cases}
\end{align}

 Second, if $n=2l+1$, we choose the maximal torus $T\subset\SO(2l+1)$ to consist of
\begin{align*}
\bdiag\big(R(t_1),\dots,R(t_l),1\big).
\end{align*}
 The bases for $\ftt_\CC$ and $\ftt^*_\CC$ are still (formally) given as above. The root system, on the contrary, has a different structure; namely,
\begin{align*}
\Delta^+=\{ \epsilon_i \pm \epsilon_j\mid 1\leq i<j\leq l\}\cup\{\epsilon_i\mid 1\leq i\leq l\}.
\end{align*}
Consequently, each positive root space $\fkk_\alpha\subset\fkk_\CC$ is spanned by a matrix $X_\alpha$, where $X_{\epsilon_i\pm\epsilon_j}$ is the image of \eqref{eq:Xepsilonij} under the embedding $\CC^{2l\times 2l}\rightarrow\CC^{(2l+1)\times (2l+1)}$ into the left upper corner, while
\begin{align*}
X_{\epsilon_i}=\frac1{\sqrt2}\left(E_{2i-1,n}-E_{n,2i-1}-\sqrt{-1}E_{2i,n}+\sqrt{-1}E_{n,2i}\right).
\end{align*}

To conclude, recall that in the case $K=\SO(n)$, $\widehat K$ is parametrized by the set
\begin{align*}
\Lambda=\begin{cases}\big\{(\lambda_1,\dots,\lambda_l)\in\ZZ^l\mid\lambda_1\geq\lambda_2\geq\cdots\lambda_{l-1}\geq|\lambda_l|\big\}&\text{if }n=2l,\\[1ex]
\big\{(\lambda_1,\dots,\lambda_l)\in\ZZ^l\mid\lambda_1\geq\lambda_2\geq\cdots\lambda_l\geq0\big\}&\text{if }n=2l+1.
\end{cases}
\end{align*}
Explicitly, $(\lambda_1,\dots,\lambda_l)$ corresponds to the highest weight $\sum_{i=1}^l\lambda_i\epsilon_i$.

\section{Valuations on convex bodies}

\label{s:val}

To fix notation and for quick later reference we review the theory of translation-invariant valuations on convex bodies. For more information on valuations see the books by Alesker   \cite{Alesker:Kent}, Klain and Rota \cite{KlainRota}, and Schneider \cite{Schneider:BM}. 

We denote by 
$$ v_n = \frac{\pi^{\frac n 2}}{\Gamma(\frac n 2 + 1)}$$
the volume of the $n$-dimensional Euclidean unit ball $D^n\subset \RR^{n}$ and  by $s_n$ the volume of the Euclidean unit sphere $S^n=\partial D^{n+1}$. We will use the well-known relations
\begin{equation}\label{eq:volumeBall} v_n= \frac{2\pi}{n} v_{n-2}, \quad  s_{n-1} = n v_n, \quad \text{and}\quad s_n = \frac{2\pi}{n-1} s_{n-2}.\end{equation}
Moreover, we will need the definite integral
\begin{equation}\label{eq:integralSpheres} \int_0^{\frac \pi 2}  \cos^m(\theta) \sin^n(\theta) d\theta =   \frac{s_{m+n+1}}{s_ms_n}.\end{equation}

\subsection{Continuous translation-invariant valuations}

Let $\calK(\RR^n)$ denote the family of convex bodies in $\RR^n$. Equipped with the Hausdorff metric, $\calK(\RR^n)$ becomes  a locally compact topological space.  A function $\phi\colon \calK(\RR^n)\to \CC$ is called a valuation if 
$$ \phi(K\cup L)= \phi(K)+ \phi(L)-\phi(K\cap L)$$
 holds for $K,L\in\calK(\RR^n)$, whenever $K\cup L$ is convex. The space of continuous and translation-invariant valuations is denoted $\Val(\RR^n)$. 
 
 It should be mentioned that valuations defined on various other classes of sets have been investigated in the literature and that, more recently, valuations on function spaces have also attracted  considerable attention. We refer the reader to the books cited at the beginning of this section and to the  articles \cite{CLM:Hadwiger,CLM:Homogeneous,Knoerr:Support,Knoerr:Smooth} for more information on these subjects.

A fundamental classical result is McMullen's decomposition theorem asserting that 
$$ \Val(\RR^n)=\bigoplus_{r=0}^n \Val_r(\RR^n)$$
where $\Val_r(\RR^n)$ denotes the subspace of $r$-homogeneous valuations, i.e., satisfying $\phi(tK)=t^r \phi(K)$ for all $t>0$ and  $K\in\calK(\RR^n)$. 
The above can be further refined by decomposing 
$$ \Val_r(\RR^n)= \Val_r^0(\RR^n)\oplus \Val_r^1(\RR^n)$$
into even and odd valuations with respect to the reflection in the origin. More precisely, $\phi\in  \Val^s_r(\RR^n)$ if and only if $\phi(-K)=(-1)^s \phi(K)$ for all $K\in \calK(\RR^n)$.

It is  easy to see that  $\Val_0(\RR^n)$ is spanned by the constant valuation $\chi(K)=1$, $K\in\calK(\RR^n)$. 
 By a theorem of Hadwiger, $\Val_n(\RR^n)$  is also $1$-dimensional and spanned by the volume $\vol_n$.
The $r$-th intrinsic volume $\mu_r$ defined by Steiner's formula 
\begin{align*}
\vol_n(K+tD^n) = \sum_{r=0}^n v_{n-r}t^{n-r} \mu_r(K),\quad K\in\calK(\RR^n),t\geq 0,
\end{align*}
is an important example of valuation of degree $r$.  The mixed volume $V$ given by
\begin{align*}
\vol_n(t_1K_1+\cdots+t_nK_n)=\sum_{i_1,\dots,i_n=1}^nt_{i_1}\cdots t_{i_n}V(K_{i_1},\dots,K_{i_n}),\quad K_i\in\calK(\RR^n),t_i\geq 0,
\end{align*}
is further a $1$-homogeneous valuation in every argument. Observe that
\begin{align*}
2\mu_{n-1}(K)=nV(K,\dots,K,D^n),\quad K\in\calK(\RR^n).
\end{align*}

The space $\Val(\RR^n)$ carries a natural topology, namely, the topology of uniform convergence on compact subsets. McMullen's decomposition  theorem implies that  it is in fact induced by the Banach space norm
$$ \|\phi\|=\sup_{K\subset D^n} |\phi(K)|.$$
The  natural action of the general linear group $\GL(n)$ on $\Val(\RR^n)$  is 
$$ (g\cdot \phi)(K) = \phi(g^{-1}K)$$
and  thus $\Val(\RR^n)$ becomes a continuous Banach space representation. Alesker's celebrated irreducibility theorem \cite{Alesker:Irreducibility} asserts that $\Val_r^0(\RR^n)$ and $\Val_r^1(\RR^n)$ are  irreducible, i.e., contain no non-trivial closed $\GL(n)$-invariant subspaces. 
 
 General theory dictates that under the action of the  compact group $\SO(n)$, the representations  $\Val_r^0(\RR^n)$ and $\Val_r^1(\RR^n)$ decompose into isotypic components. The  following theorem characterizes the $\SO(n)$-types that appear.

 \begin{theorem}[Alesker--Bernig--Schuster {\cite[Theorem 1]{ABS:Harmonic}}]
 	\label{thm:ABS}
 	Let $ l=\lfloor\frac n2\rfloor$ and  $ 0\leq r \leq  l
 	$.
 	The non-trivial $\SO(n)$-types in $\Val_{r}(\RR^n)$ and $\Val_{n-r}(\RR^n)$ are the same and given by the following set of highest weights:
 	\begin{align*}
 		\Lambda_r=\{(\lambda_1,\dots,\lambda_l)\in\Lambda\mid  |\lambda_2|\leq 2,\ |\lambda_j|\neq1 \text{ for }1\leq j\leq r,\ \lambda_j=0\text{ for }r< j\leq l\}.
 	\end{align*}
 	Moreover, each  appears with multiplicity one.
 \end{theorem}

\begin{remark}
\label{re:evenWeights}
Theorem \ref{thm:ABS} was first proved by Alesker \cite[Proposition 6.3]{Alesker:Irreducibility} for even valuations: Replacing $\Val_j(\RR^n)$ by $\Val_j^0(\RR^n)$, $j\in\{r,n-r\}$, $\Lambda_r$ has to be replaced by
\begin{align*}
\Lambda_r^0=\{(\lambda_1,\dots,\lambda_l)\in\Lambda_r\mid \lambda_1\in 2\ZZ\}.
\end{align*}
This shows that the parity of valuations in the irreducible subspace corresponding to $\lambda_{k,m}$ is given by the parity of $m$. Below we will give a new proof of this statement by constructing for each $\SO(n)$-type a non-trivial highest weight vector.
\end{remark}

\subsection{Integration over the normal cycle}
\subsubsection{Normal cycles}
Integration over the normal cycle is an important means of constructing valuations. Let $S\RR^n=\RR^n\times S^{n-1}$ denote the sphere bundle of $\RR^n$. As a set, the normal cycle of a convex body $K\in\calK(\RR^n)$ is 
$$N(K)= \{ (x,u)\in S\RR^n\colon u \text{ outward normal to } K \text{ at } x\}.$$
Since $N(K)$ is also  a naturally oriented $(n-1)$-dimensional Lipschitz submanifold, smooth differential $(n-1)$-forms on $S\RR^n$ may be integrated over the normal cycle.
 Let $\Omega^*(S\RR^n)^{\tr}\subset\Omega^*(S\RR^n)$ denote the subspace of translation-invariant forms. The key fact is that  each form $\omega\in \Omega^{n-1}(S\RR^n)^{\tr}$ defines via
\begin{align}
\label{eq:normalcycle}
\phi(K)= \int_{N(K)} \omega,\quad  K\in\calK(\RR^n)
\end{align}
a continuous, translation-invariant valuation (see, e.g., \cite{Alesker:VMfdsIII} or \cite{Fu:IGRegularity}).  Observe that the space of differential forms on $S\RR^n$ is naturally bi-graded and that $\phi\in\Val_r(\RR^n)$ provided $\omega\in\Omega^{r,n-1-r}(S\RR^n)^{\tr}$.

\subsubsection{The kernel theorem}
The differential forms $\omega$ for which \eqref{eq:normalcycle}  yields the trivial valuation can be characterized in terms of the Rumin differential \cite{Rumin:Contact}. This natural second order differential operator is defined for $(n-1)$-forms  on a contact manifold $M$ of dimension $2n-1$. Assume for simplicity that the contact structure  on $M$ is defined by a global contact form $\alpha$. Then, restricted to  each contact hyperplane 
$$ H_x= \ker \alpha_x, \quad x\in M,$$
$d\alpha$ is non-degenerate. By the Lefschetz decomposition theorem applied to the exterior algebra $\largewedge H_x^* $, multiplication by $d\alpha$ defines an isomorphism $\largewedge^{n-2} H_x^* \to \largewedge^{n} H_x^*$. 
Consequently, there exists a unique form  $\alpha \wedge \xi\in \Omega^{n-1}(S\RR^n)$ such that 
$$ d (\omega + \alpha \wedge \xi)$$
is divisble by $\alpha$. The latter $n$-form is denoted $D\omega $ and called the Rumin differential of $\omega$. 
The following is a special case of the general Bernig--Br\"ocker kernel theorem.
\begin{theorem}[Bernig--Br\"ocker \cite{BernigBroecker:Rumin}]\label{thm:kernel} Let  $1\leq r\leq n-1$.  The valuation defined by $\omega\in \Omega^{r,n-r-1}(S\RR^n)^{\tr}$ is identically zero if and only if $D\omega=0$.
\end{theorem}

\subsubsection{Smooth valuations}
A particularly important technical innovation introduced by Alesker is the notion of smooth valuation.  A valuation $\phi\in \Val(\RR^n)$ is called smooth if it is a smooth vector under the action of the general linear group. The dense subspace of  smooth valuations is denoted by $\Val^\infty(\RR^n)$.  Similarly, we define $\Val_r^\infty(\RR^n)$ and $\Val_r^{s,\infty}(\RR^n)$.

 An important consequence of the irreducibility theorem and a big advantage of working with smooth valuations is that they can be represented by differential forms. More precisely, given $\phi\in \Val_r^\infty(\RR^n)$, $r<n$, there exists  $\omega\in\Omega^{r,n-r-1}(S\RR^n)^{\tr}$ such that \eqref{eq:normalcycle} holds, see  \cite{Alesker:VMfdsIII}. Conversely,  it is  not hard to verify that every valuation of this form is smooth.

\subsection{Product and convolution of smooth valuations}

\subsubsection{Alesker product}

One of the most remarkable properties of the space of smooth valuations is that it carries the structure of a graded commutative algebra. Given $A\in \calK(\RR^n)$ define $\phi_A=\vol_n (\Cdot + A)$.  It is readily verified that $\phi_A\in\Val(\RR^n)$. In fact, $\phi_A$ is smooth provided the convex body $A$ has  a smooth and strictly positively curved boundary.

\begin{theorem}[Alesker \cite{Alesker:Product}]
There exists a unique continuous  and bilinear  product $\Val^\infty(\RR^n)\times  \Val^\infty(\RR^n)\to \Val^\infty(\RR^n)$ such that
  $$(\phi_A\cdot \phi_B)(K)= \vol_{2n}( \Delta(K) + A\times B),$$
 where $\Delta\maps{\RR^n}{\RR^n\times\RR^n}$ is the diagonal embedding, holds for all $A,B\in\calK(\RR^n)$ with smooth and strictly positively curved boundary.
Moreover,
\begin{enuma}
\item  $\cdot$ is commutative and associative;
	\item $\chi$ is the identity element;
	\item $\Val^\infty_r (\RR^n)\cdot \Val^\infty_q(\RR^n) \subset \Val_{r+q}^\infty(\RR^n)$;
	\item the product pairing $ \Val_r^\infty(\RR^n)\times \Val_{n-r}^\infty(\RR^n)\to  \Val_n^\infty(\RR^n)\cong\CC$ is perfect.
\end{enuma}
\end{theorem}

 Item (d) of the preceding theorem is a version of Poincar\'e duality for valuations. The pairing is usually referred to as the Alesker--Poincar\'e pairing.

A general formula for the product of valuations in terms of differential forms exists (see \cite{AleskerBernig:Product}) but it seems too unwieldy for explicit computations.  However, in an important special case it boils down to a  more manageable expression.

\begin{theorem}[{Bernig \cite[Theorem 4.1]{Bernig:Formula}}]
\label{thm:BernigFormula}
Let $1\leq r\leq  n-1$ and $s\in\{0,1\}$. Assume that $\phi_1\in\Val_r^{\infty,s}(\RR^n)$ and $\phi_2\in\Val^{\infty,s}_{n-r}(\RR^n)$ are represented  by $\omega_1\in\Omega^{r,n-r-1}(S\RR^n)^{\tr}$ and $\omega_2\in \Omega^{n-r,r-1}(S\RR^n)^{\tr}$, respectively. If the standard orientation on $\RR^n$ is used to identify $\largewedge^n(\RR^n)^*$ and $\Val_n(\RR^n)$, then
\begin{align*} \phi_1\cdot \phi_2 = (-1)^{r+s}\int_{S^{n-1}}\omega_1\wedge D\omega_2.
\end{align*}
\end{theorem}

Let us spell out explicitly how to interpret the integral  in the previous theorem. Observe that $\omega_1\wedge D\omega_2\in\Omega^{n,n-1}(S\RR^n)^{\tr}$ and that this space is naturally identified with $\largewedge^n(\RR^n)^*\otimes \Omega^{n-1}(S^{n-1})$. For $\alpha\in\largewedge^n(\RR^n)^*$ and $\beta\in\Omega^{n-1}(S^{n-1})$ one defines
\begin{align*}
\int_{S^{n-1}}(\alpha\otimes\beta)=\left(\int_{S^{n-1}}\beta\right)\alpha\in\largewedge^n(\RR^n)^*.
\end{align*}

\subsubsection{Bernig--Fu convolution}

Remarkably, there is  yet another continuous product on  the space of smooth valuations. 
We have the following theorem.

\begin{theorem}[{Bernig--Fu \cite{BernigFu:Convolution}}]
	There exists a unique continuous  and bilinear product $\Val^\infty(\RR^n)\times  \Val^\infty(\RR^n)\to \Val^\infty(\RR^n)$ such that
	$$(\phi_A*\phi_B)(K) = \vol_n(K + A+ B)$$
 for all $A,B\in\calK(\RR^n)$ with  smooth and strictly positively curved boundary.  Moreover,
\begin{enuma}
\item $*$ is commutative and associative;
	\item $\vol_n$ is the identity element;
	\item $\Val_{n-r}^\infty(\RR^n)* \Val_{n-q}^\infty(\RR^n) \subset \Val_{n-r-q}^\infty(\RR^n)$.
\end{enuma}	
\end{theorem}

Let us point out that, twisting the space $\Val^\infty$ by the one-dimensional space of densities, the definition of convolution can be reformulated without reference to any  Euclidean structure  so that it commutes with the  natural action of the general linear group. This will be briefly touched upon in \S\ref{ss:FF} below. In the above form, however, both convolution and product define a multiplicative structure on the same space. These algebras are different, but isomorphic (see \S\ref{sss:FF}). Moreover, the induced pairings coincide on $\Val_r^{s,\infty}(\RR^n)$ up to sign, after suitable identifications are made.

\begin{proposition}[{\cite[Lemma 2.3]{BernigFu:Convolution} and \cite[Proposition 4.2]{Wannerer:UnitaryAreaMeasures}}]  	Let  $0\leq r\leq  n$ and $s \in\{0,1\}$.  If $\Val_0(\RR^n)$ and $\Val_n(\RR^n)$ are identified with $\CC$ via $\chi$ and $\vol_n$, then  
	\label{pro:PvsC}
for any $\phi_1\in\Val_r^{s,\infty}(\RR^n)$ and $\phi_2\in\Val^{s,\infty}_{n-r}(\RR^n)$  one has
	\begin{align*}
		\phi_1\cdot\phi_2=(-1)^s \phi_1*\phi_2.
	\end{align*}
\end{proposition}

The convolution is typically easier to evaluate than the product. One example that will be relevant for us is the following.  Recall that the intrinsic volumes are smooth and consider the degree $-1$ linear operator on $\Val^\infty(\RR^n)$ given by
\begin{align}
\label{eq:Lambda}
\Lambda\phi=2\mu_{n-1}*\phi.
\end{align}
It is an easy consequence of Steiner's formula that for any $\phi\in\Val^\infty(\RR^n)$ one has
\begin{align*}
(\Lambda\phi)(K)=\dt \phi(K+tD^n),\quad K\in\calK(\RR^n),
\end{align*}
see \cite[Corollary 1.8]{BernigFu:Convolution}. In particular,
\begin{align}
\label{eq:Lambdaton}
\Lambda^n\vol_n=n!v_n\chi.
\end{align}
In the language of differential forms, $\Lambda$ is nothing else than the Lie derivative with respect to the Reeb vector field $T_{(x,\xi)}= \sum_{i=1}^n \xi_i \pder{x_i}$ on $S\RR^n$:

\begin{proposition}[{Bernig--Br\"ocker \cite[Lemma 3.4]{BernigBroecker:Rumin}}]
\label{pro:convLefschetz}
Assume $\phi\in \Val^\infty(\RR^n)$ is represented by $\omega\in \Omega^{n-1}(S\RR^n)^{\tr}$. Then
\begin{align*}
(\Lambda\phi)(K)= \int_{N(K)} \calL_T \omega,\quad K\in\calK(\RR^n).
\end{align*} 
\end{proposition}

\subsubsection{Hard Lefschetz theorem}

Using Proposition \ref{pro:convLefschetz} together with Kähler identities from Kähler geometry, Bernig and Bröcker proved that the operator $\Lambda$ satisfies the hard Lefschetz property.
\begin{theorem}[{Bernig--Br\"ocker \cite[Theorem 2]{BernigBroecker:Rumin}}]
\label{thm:convHL}
Let $0\leq r\leq\lfloor\frac n2\rfloor$. The map $\Lambda^{n-2r}\maps{\Val_{n-r}^\infty(\RR^n)}{\Val_r^\infty(\RR^n)}$ is an isomorphism of topological vector spaces.
\end{theorem}

In \S\ref{s:HLHR} below we will give an alternative prove of this result and supplement it with the Hodge--Riemann relations. Our proof relies on Proposition \ref{pro:convLefschetz} as well as on the Alesker--Bernig--Schuster decomposition theorem.

\begin{remark}
\label{re:HLinABS}
Strictly speaking, the proof of Theorem \ref{thm:ABS} given in \cite{ABS:Harmonic} depends on Theorem \ref{thm:convHL}; namely, the Lefschetz isomorphism is used to extend the validity of the statement from $\Val_{n-r}(\RR^n)$ to $\Val_r(\RR^n)$. Observe, however, that the use of the hard Lefschetz theorem can easily be avoided, using only Alesker--Poincar\'e duality  instead: For any $\lambda\in\Lambda$ let $\wt E_\lambda\subset\Val_r(\RR^n)$ and $E_\lambda\subset\Val_{n-r}(\RR^n)$ be the $\lambda$-isotypic components. By Alesker--Poincar\'e duality and Proposition \ref{pro:Schur} (a), there is a non-degenerate $\SO(n)$-invariant sesquilinear pairing $\wt E_\lambda\times E_\lambda\to\CC$. If Theorem \ref{thm:ABS} holds for $\Val_{n-r}(\RR^n)$, we have $\dim E_\lambda<\infty$ and consequently $\wt E_\lambda\cong \b E^*_\lambda\cong E_\lambda$.
\end{remark}

\subsection{Pullback, pushforward, and the Fourier transform}

\label{ss:FF}

\subsubsection{Fourier transform}

\label{sss:FF}

The Fourier transform for valuations discovered by Alesker is a linear isomorphism that respects  the  action of the general linear group, intertwines the  product and the convolution, and  satisfies an analog of the Fourier inversion theorem. The behavior with respect to the action of $\GL(n)$ will  be critical for evaluating the  Fourier transform of highest weight vectors. To trace the correct behavior under the group action, we first do not fix an Euclidean inner product  and work with an abstract $n$-dimensional real vector space $V$ instead. As a consequence, there exists no canonical choice of volume  on $V$, but a one-dimensional space of Haar measures. The complexification of this space is  denoted by $\Dens(V)$  and its elements are called densities on $V$. Note that $\Val_n(V)\cong \Dens(V)$  by Hadwiger's theorem.

\begin{theorem}[{Alesker \cite{Alesker:Fourier}}]
\label{thm:AFT}
	There exists an isomorphism of topological vector spaces 
	$$ \FF_V \colon \Val^\infty(V)\to \Val^\infty(V^*)\otimes \Dens(V)$$
	which satisfies the following properties:
	\begin{enuma}
		\item $\FF_V$ commutes with the natural action of $\GL(V)$ on both spaces;
		\item $\FF_V$ is an isomorphism of algebras when the source is equipped with the product and the target with the convolution;
		\item after the canonical identification $ \Dens(V^*)\otimes \Dens(V)\cong\CC$, the composition
		\begin{align*}
		\calE_V=\left(\FF_{V^*}\otimes Id_{\Dens(V)}\right)\circ\FF_V\maps{\Val^\infty(V)}{\Val^\infty(V)}
		\end{align*}
		satisfies 
		  $$ (\calE_V\phi)(K)= \phi(-K).$$ 
	\end{enuma}
\end{theorem}

 We will reduce our computation of the Fourier transform recursively to the case $V=\RR^2$ where an explicit description is available as follows. First, by an enhancement of a classical theorem of Hadwiger,  every valuation $\phi\in \Val^\infty(\RR^2)$ can be  (uniquely) expressed as 
$$ \phi(K) = c_0\chi(K)  + \int_{S^1} f(u) dS_1(K,u)+ c_2 \vol_2(K),$$
where $c_0,c_2\in \CC$ are constants and $f\colon S^1\to \CC$ is a smooth function orthogonal to the  two-dimensional space of linear functionals restricted to $S^1$. $S_1(K,\Cdot)$ denotes here the area measure of  $K\in\calK(\RR^2)$.  Decompose $f=f_0+f_1$ into the even and odd part and the latter $f_1=f_{1,h}+f_{1,a}$ into the holomorphic and anti-holomorphic part, where $\RR^2\cong\CC$ as usual. We further use the inner product on $\RR^2$ to identify $\Val^\infty\big((\RR^2)^*\big)\otimes \Dens(\RR^2) \cong \Val^\infty(\RR^2)$. Then one has
\begin{align*}
(\FF\phi)(K)=c_0\vol_2(K)+\int_{S^1} (f_0+ f_{1,h}- f_{1,a})(Ju)  dS_1(K,u)+c_2\chi(K),
\end{align*}
where $J\colon \RR^2\to \RR^2$ denotes  the counter-clockwise rotation by $\pi/2$.  

\subsubsection{Pullback and pushforward}

To evaluate the Fourier transform on highest weight vectors we will make use of its compatibility properties with respect to the pullback and pushforward of valuations. Since we will need the pullback under monomophisms and the pushforward along epimorphisms exclusively, we only define them in these special cases here. For the general case see \cite{Alesker:Fourier}.

The pullback is easy to define. If $f\colon V\to W$ is an injective linear map of vector spaces, then $f^*\colon \Val^\infty(W)\to \Val^\infty(V)$ is a morphism of algebras with respect to the product defined by
\begin{align*}
(f^* \phi)(K)= \phi \big(f(K)\big).
\end{align*}

The definition of the pushforward, on the contrary, is  slightly more involved. Let $f\colon V\to W$ be a surjective linear map and  choose $U\subset V$ such that $V=\ker f \oplus U $. $g=f|_U\maps UW$ is an isomorphism which yields $\Dens(U)\cong\Dens(W)$. Observe further that there is a natural isomorphism $\Dens(V^*)\cong \Dens(V)^*$ and that
$$ \Dens(V^*)\cong \Dens(\ker f)^* \otimes \Dens(W^*).$$
Then $f_* \colon \Val^\infty(V)\otimes \Dens(V^*)\to \Val^\infty(W)\otimes \Dens(W^*)$ is a morphism of algebras with respect to the convolution defined as follows. For
\begin{align*}
\phi\otimes \vol_{\ker f}^* \otimes\vol_{W^*}\in \Val^\infty(V)\otimes \Dens(\ker f)^* \otimes \Dens(W^*)\cong   \Val^\infty(V)\otimes \Dens(V^*)
\end{align*}
and any $K\in\calK(W)$ we set
\begin{align*}
(f_* \phi) (K) = \frac{1}{k!}\left. \frac{d^k}{dt^k} \right|_{t=0} \phi\big(t S+ g^{-1} K\big) \otimes \vol_{W^*},
\end{align*}
where $k=\dim \ker f$ and $S\in \calK(\ker f)$ is chosen so that $\vol_{\ker f} (S)=1$. One can check that the definition does not depend on the choices  of $U$ and $S$. 

\begin{theorem}[{Alesker \cite{Alesker:Fourier}}]
\label{thm:FFf}
	Let $f\colon V\to W$ be an injective linear map and let $f^\vee\colon W^*\to V^*$ denote the dual map. Then 
	$$ \FF_W\circ f^* = f^\vee _* \circ \FF_V.$$
\end{theorem}

\section{Construction of highest weight vectors}
The following two sections constitute the technical heart of this paper. Namely, we explicitly construct in each $\SO(n)$-type appearing in $\Val_r(\RR^n)$ a  non-zero highest weight vector. To this end, and this is the goal of the current section, we first construct a highest weight vector in $ \Omega^{r,n-r-1}(S\RR^n)^{\tr}$. Integration over the normal cycle then  yields a highest weight vector in $\Val_r(\RR^n)$. To prove the latter is non-trivial, it suffices by the Bernig--Br\"ocker kernel theorem to show that Rumin differential of the corresponding differential form is non-zero. This step of our argument, to be completed in the next section, is combinatorially challenging  but the formalism of double forms allows us to conveniently handle the rather complicated determinantal expressions appearing below.

\subsection{Double forms}

Let $M$ be a smooth manifold and let $V$ be a  finite-dimensional, real vector space.  We denote by $\Omega^k(M,V)$ the space of smooth  $V$-valued differential $k$-forms, i.e., the smooth sections of the bundle $\CC\otimes \largewedge^k T^*M\otimes V $, where the tensor products are taken over $\RR$. The  complex conjugation, pullback, exterior derivative, and Lie derivative of a $V$-valued form are defined in the usual way. Recall  also that there exists a natural  product 
$$ \Omega^k(M,V) \otimes \Omega^l(M,W)\to \Omega^{k+l}(M,V\otimes W)$$
of $V$- and $W$-valued forms. In the following we will make use of differential forms with values in $\largewedge V$, the exterior algebra of a vector space $V$. In this setting, the composition of the above general product with the wedge product on $\largewedge V$ yields a wedge product on $ \Omega^*(M,\largewedge V)= \bigoplus_k \Omega^k(M,\largewedge V)$. We will refer to elements of  this algebra as double forms.
A double form of type $(p,q)$ is a smooth section of the  bundle $ \CC\otimes\largewedge^pT^*M\otimes\largewedge^q V$. Note that 
\begin{align*}
	\omega \theta=(-1)^{pr+qs}\theta\omega
\end{align*}
if $\omega$ is of type $(p,q)$ and $\theta$ of $(r,s)$, where here and in the following we suppress the wedge symbol. For any $(1,1)$-double form $\beta$ and $j\in\NN_0$ we denote
\begin{align*}
	\beta^{[j]}=\frac1{j!}\beta^j.
\end{align*}
Then for any such double forms $\beta_1,\beta_2$ and $m\in\NN_0$ one has
\begin{align}
	\label{eq:binomial}
	(\beta_1+\beta_2)^{[m]}=\sum_{j=0}^m\beta_1^{[j]}\beta_2^{[m-j]}.
\end{align}
Moreover, it will be convenient and natural to set $\beta^{[-1]}=0$ and ${i\choose j}=0$ if $i<j$.

{ For our purposes, $M$ will usually be either $\RR^n\times\RR^n$ or $S\RR^n$ and the double forms will take values in $V= \CC\otimes \largewedge( \RR^n\times \RR^n)^*$. In this connection, let us fix once and for all the following notation. First, the canonical coordinates in $\RR^n \times \RR^n$ are denoted by $x_1,\dots,x_n,\xi_1,\dots,\xi_n$. Second, put $l=\lfloor \frac n2\rfloor$ and define  for $j\in \{1,\ldots, l\}$
\begin{align*}
z_j &= \frac{1}{\sqrt2}\left(x_{2j-1} + \sqrt{-1}  x_{2j}\right),\\
 z_{\b j} &= \frac{1}{\sqrt2}\left(x_{2j-1} -\sqrt{-1}  x_{2j}\right),\\
\zeta_j &= \frac{1}{\sqrt2}\left(\xi_{2j-1} + \sqrt{-1}  \xi_{2j}\right),\\
 \zeta_{\b j} &= \frac{1}{\sqrt2}\left(\xi_{2j-1} -\sqrt{-1}  \xi_{2j}\right).
\end{align*}
If $n=2l+1$, then we set in addition
\begin{align*}
 z_{l+1} &= x_{2l+1}, \\
\zeta_{l+1}&=\xi_{2l+1}.
\end{align*}
Further, consider the set  of indices
\begin{align*}
	\calI=
	\begin{cases}
		\{1,\b 1,\dots, l, \b l\}\quad&\text{if }n=2l,\\[2ex]
		\{1,\b 1,\dots, l, \b l,l+1\}\quad&\text{if }n=2l+1,
	\end{cases}
\end{align*}
 and equip it with the total order $1\prec\b 1\prec\cdots\prec l\prec\b l  \prec l+1$. For any subset $I=\{i_1,\dots,i_j\}\subset\calI$ with $i_1\prec\cdots \prec i_j$ we define the following (double) forms:
\begin{align*}
\Theta_{1,I}&=dz_{i_1}\cdots dz_{i_j},\\
\Theta_{2,I}&=d\zeta_{i_1}\cdots d\zeta_{i_j},\\
	\Theta_I&=\Theta_{1,I}\,\Theta_{2,I},\\
	\alpha_I&=\sum_{i\in I}\zeta_{\b i}dz_i,\\
	\gamma_I&=\sum_{i\in I}\zeta_{\b i}d\zeta_i,\\
	\nu_I&=\sum_{i\in I}\zeta_{\b i}\zeta_i,\\
	\zeta_I&=\sum_{i\in I}\zeta_i\otimes dz_i,\\
	z_I&=\sum_{i\in I}z_i\otimes dz_i,\\
	\eta_I&=\sum_{i\in I}\zeta_i\otimes d\zeta_i,\\
	w_I&=\sum_{i\in I}z_i\otimes d\zeta_i,
\end{align*}
where $\bar{\bar i}=i$ and $\b{l+1}=l+1$. Notice that $\alpha=\alpha_\calI$ is the contact form on $S\RR^n$. Similarly, we set
\begin{align*}
\Theta_1=\Theta_{1,\calI},\quad\Theta_2=\Theta_{2,\calI},\quad\Theta=\Theta_\calI,\quad	\gamma=\gamma_\calI,\quad\text{and}\quad\nu=\nu_\calI.
\end{align*} 
Assume $k\in\NN$, $k\leq l$. Of particular interest will be the following subsets of $\calI$:
\begin{align*}
	K=\{1,2,\dots, k\}, \quad J=\calI\setminus K,\quad\text{and}\quad L=J\setminus \b K,
\end{align*}
where $\b{\{i_1,\dots,i_j\}}=\{\b{i_1},\dots,\b{i_j}\}$; for $n=2l$ we will use in addition
\begin{align*}
	M=\{1,2,\dots,l-1, \b l\}.
\end{align*}
Observe that $\calI=K\cup\b K\cup L$, $ J=\b K\cup L$, and $\calI=M\cup\b M$ (if $n=2l$) are disjoint unions. Finally, for any $I\subset\calI$ and $a_1,\dots,a_p\in I$ we will denote
\begin{align*}
	I_{a_1,\dots,a_p}=I\setminus\{a_1,\dots,a_p\}.
\end{align*}

In general, we will  not notationally distinguish  (double) forms on $\RR^{n}\times\RR^n$ from their restrictions to $S\RR^n$; however, an explicit distinction will be made whenever this difference is relevant. In this connection, let us state for later use the following obvious fact.
\begin{lemma} \label{lemma:RestrictionSphere}
Let $\iota\colon \RR^n\times \RR^n \hookrightarrow S\RR^n$ denote the inclusion map and let $\omega\in \Omega^*(\RR^n\times \RR^n)$. If $\omega \wedge \gamma =0$, then 
$\iota^*\omega=0$. 
\end{lemma}

\subsection{Action of the orthogonal group}

\label{ss:SOaction}

Recall that if $G\times M\to M$ is a smooth left action of a Lie group  on a smooth manifold and if $X$ is an element of the Lie algebra $\fgg$ of $G$, then 
$$ \widetilde X_p = \dt \exp(-tX) \cdot p$$ 
is called the fundamental vector field induced by $X$. On smooth differential forms $\omega\in \Omega^k(M,V)$ the group  acts by $ g\cdot \omega =  (g^{-1})^*\omega$ . Thus the infinitesimal action of $\fgg$ on differential forms is given by 
$$ X\cdot \omega= \calL_{\widetilde X} \omega,$$
where $\calL$ denotes the Lie derivative. By linearity, this can be extended to an action of $\fgg_\CC$, the complexification of $\fgg$. 

The natural left action of $\SO(n)$ on  $\RR^n\times \RR^n$ and $S\RR^n$ is $g\cdot (x,\xi)=  (gx, g\xi)$. A direct computation shows that the corresponding fundamental vector fields (on either $\RR^n\times \RR^n$ or $S\RR^n$) induced by the basis vectors of $\fnn^+\oplus\ftt_\CC\subset\so(n)_\CC$ introduced in \S\ref{ss:SO} are given by
\begin{align}
\label{eq:BorelAction}
\begin{split}
\widetilde X_{\epsilon_i -\epsilon_j } & = z_{\b i} \pder{z_{\b j}}  - z_j \pder{ z_i}+\zeta_{\b i} \pder{\zeta_{\b j}}  - \zeta_j \pder{ \zeta_i},\\
\widetilde  X_{\epsilon_i +\epsilon_j} & = z_{\b i} \pder{ z_j}  - z_{\b j} \pder{ z_i}+ \zeta_{\b i} \pder{ \zeta_j}  - \zeta_{\b j} \pder{ \zeta_i},\\
\widetilde  X_{\epsilon_i} & = z_{\b i} \pder{ z_{l+1}}  - z_{l+1} \pder{ z_i}+ \zeta_{\b i} \pder{ \zeta_{l+1}}  - \zeta_{l+1} \pder{ \zeta_i}\quad\text{(only if $n=2l+1$ is odd)},\\
\widetilde H_i&=  z_{\b i}\pder{z_{\b i}}-z_i\pder{z_i} +\zeta_{\b i}\pder{\zeta_{\b i}}-\zeta_i\pder{\zeta_i},
\end{split}
\end{align}
 where 
\begin{align*} \pder{z_j} & = \frac{1}{\sqrt 2} \left( \pder{x_{2j-1}} - \sqrt{-1} \pder{x_{2j}}\right), & \pder{z_{\b j}} &= \frac{1}{\sqrt 2} \left( \pder{x_{2j-1}} + \sqrt{-1} \pder{x_{2j}}\right),\\
 \pder{\zeta_j} & = \frac{1}{\sqrt 2} \left( \pder{\xi_{2j-1}} - \sqrt{-1} \pder{\xi_{2j}}\right), & \pder{\zeta_{\b j}} &= \frac{1}{\sqrt 2} \left( \pder{\xi_{2j-1}} + \sqrt{-1} \pder{\xi_{2j}}\right)
\end{align*}
for $j\in \{ 1,\ldots, l\}$ and, provided $n=2l+1$, 
$$ \pder{z_{l+1}}= \pder{x_{2l+1} }, \quad    \pder{\zeta_{l+1}}= \pder{\xi_{2l+1} }.$$

\subsection{Highest weight vectors}

We  will now describe a family of highest weight vectors in $\Omega^{r,n-r-1}(S\RR^n)^{\tr}$.} Already here double forms are a convenient tool for organizing our computation and stating our results.

\begin{theorem}
\label{thm:HWforms}
 For any $r,k,m\in\NN$ with $r\leq n-1$, $k\leq \min\{r,n-r\}$, and $m\geq2$,
\begin{align}
\label{eq:HWomega}
\omega_{r,k,m}=\zeta_{\b1}^{m-2}\omega_{r,k}\in\Omega^{r,n-r-1}(S\RR^n)^{\tr}
\end{align}
is a highest weight vector of weight $\lambda_{k,m}$, where
\begin{align*}
\omega_{r,k}\otimes\Theta_1=\zeta_J(d\zeta_J)^{[n-r-1]}(dz_J)^{[r-k]}(\b{dz_K})^{[k]}.
\end{align*}
If $n=2l$, then in addition
\begin{align*}
\omega_{l,-l,m}=\zeta_{\b1}^{m-2}\omega_{l,-l}\in\Omega^{l,l-1}(S\RR^n)^{\tr}
\end{align*}
is a highest weight vector of weight $\lambda_{-l,m}$, where
\begin{align*}
\omega_{l,-l}\otimes\Theta_1=\zeta_{\b M}(d\zeta_{\b M})^{[l-1]}(\b{dz_M})^{[l]}.
\end{align*}
\end{theorem}

 Before we will proceed to the proof of this theorem, let us first show  two simple auxiliary statements. 

\begin{proposition}
\label{pro:HWf}
For  any $r,k\in\NN$ with $r\leq n-1$ and $k \leq \min\{ r,n-r\}$,
\begin{align} 
\label{eq:HWf01}
\begin{split}
\omega_{r,k}\otimes\Theta_1&=\zeta_{\b K}(d\zeta_{\b K})^{[k-1]}(d\zeta_L)^{[n-r-k]}(dz_L)^{[r-k]}(\b{dz_K})^{[k]}\\
&\quad+\zeta_{L}(d\zeta_{\b K})^{[k]}(d\zeta_L)^{[n-r-k-1]}(dz_L)^{[r-k]}(\b{dz_K})^{[k]}.
\end{split}
\end{align}
\end{proposition}

\begin{proof}
We have
\begin{align*}
&\zeta_J(d\zeta_J)^{[n-r-1]}(dz_J)^{[r-k]}(\b{dz_K})^{[k]}\\
&\quad=\zeta_J(d\zeta_J)^{[n-r-1]}(dz_L)^{[r-k]}(\b{dz_K})^{[k]}\\
&\quad=\zeta_{\b K}(d\zeta_J)^{[n-r-1]}(dz_L)^{[r-k]}(\b{dz_K})^{[k]}+\zeta_L(d\zeta_J)^{[n-r-1]}(dz_L)^{[r-k]}(\b{dz_K})^{[k]}\\
&\quad=\zeta_{\b K}(d\zeta_{\b K})^{[k-1]}(d\zeta_L)^{[n-r-k]}(dz_L)^{[r-k]}(\b{dz_K})^{[k]}\\
&\qquad+\zeta_L(d\zeta_{\b K})^{[k]}(d\zeta_L)^{[n-r-k-1]}(dz_L)^{[r-k]}(\b{dz_K})^{[k]}.
\end{align*}
\end{proof}

\begin{remark}
\label{re:HWf}
\begin{enuma}
\item
Since some version of the same argument will appear numerous times throughout the following pages, the previous proof is perhaps worth a brief comment: Observe that three elementary facts are used here, namely the binomial formula \eqref{eq:binomial}, the disjointness of the union $J=\b K\cup L$, and that one has $\bigwedge^aV=0$  for any vector space $V$ and any integer  $a>\dim V$.
\item Observe that one can define $\omega_{r,k}$ also for $ k>n-r$. However, the same argument as above shows that in this case $\omega_{r,k}\otimes\Theta_1 =0$.
\end{enuma}
\end{remark}

\begin{lemma}
\label{lem:HWf}
Let $I\subset\calI$ be non-empty; choose $a\in I$ and $b\in\calI$ with $a\neq b$; and denote $q=|I|$. Let  $Z_{a,b}=z_a\pder{z_b}+ \zeta_a \pder{\zeta_b}$. Then  
\begin{align}
\label{eq:HWf02}
\calL_{Z_{a,b}}(d\zeta_I)^{[q]}&=\calL_{Z_{a,b}}(dz_I)^{[q]}=0,\\
\label{eq:HWf03}
\calL_{Z_{a,b}}\zeta_I(d\zeta_I)^{[q-1]}&=0.
\end{align}
\end{lemma}

\begin{proof}
We may assume $b\in I$; otherwise the claim is obvious. In this case, \eqref{eq:HWf02} follows easily from skew-symmetry. As for \eqref{eq:HWf03}, we have
\begin{align*}
\calL_{Z_{a,b}}\zeta_I(d\zeta_I)^{[q-1]}&=(\zeta_a\otimes dz_b)(d\zeta_{I_b})^{[q-1]}+\zeta_{I_b}(d\zeta_{I_{a,b}})^{[q-2]}(d\zeta_a\otimes dz_b)\\
&=(\zeta_a\otimes dz_b)(d\zeta_{I_b})^{[q-1]}+(\zeta_a\otimes dz_a)(d\zeta_{I_{a,b}})^{[q-2]}(d\zeta_a\otimes dz_b)\\
&=(\zeta_a\otimes dz_b)(d\zeta_{I_b})^{[q-1]}-(\zeta_a\otimes dz_b)(d\zeta_{I_{a,b}})^{[q-2]}(d\zeta_a\otimes dz_a)\\
&=(\zeta_a\otimes dz_b)(d\zeta_{I_b})^{[q-1]}-(\zeta_a\otimes dz_b)(d\zeta_{I_b})^{[q-1]}\\
&=0.
\end{align*}
\end{proof}

\begin{proof}[Proof of Theorem \ref{thm:HWforms}]
First of all, observe that $\zeta_{\b1}^{m-2}$ is obviously a highest weight vector of weight $\lambda_{1,m-2}$ and there is, thus, no loss of generality in assuming $m=2$. There are several cases to be considered separately.

First, it follows at once from Proposition \ref{pro:HWf} and Lemma \ref{lem:HWf} that $X_{\epsilon_i\pm\epsilon_j}\omega_{r,k}=0$ for any $1\leq i<j\leq k$. Furthermore, from Lemma \ref{lem:HWf} we also easily infer that $X_{\epsilon_i\pm\epsilon_j}\omega_{l,-l}=0$ for any $1\leq i<j\leq l$ provided $n=2l$.

Second, let $\omega_{r,k}\otimes\Theta_1=\Omega_1+\Omega_2$ be the decomposition \eqref{eq:HWf01}. If $a\in K$ and $b\in L$, for the vector field $Y_{a,b}=z_{\b a}\pder{z_{\b b}}-z_b\pder{z_a}+\zeta_{\b a}\pder{\zeta_{\b b}}-\zeta_b\pder{\zeta_a}$ one has
\begin{align*}
\calL_{Y_{a,b}}\Omega_1&=\zeta_{\b K}(d\zeta_{\b K})^{[k-1]}(d\zeta_L)^{[n-r-k-1]}(d\zeta_{\b a}\otimes dz_{\b b})(dz_L)^{[r-k]}(\b{dz_K})^{[k]}\\
&=(\zeta_{\b a}\otimes dz_{\b a})(d\zeta_{\b K_{\b a}})^{[k-1]}(d\zeta_L)^{[n-r-k-1]}(d\zeta_{\b a}\otimes dz_{\b b})(dz_L)^{[r-k]}(\b{dz_K})^{[k]}\\
&=-(\zeta_{\b a}\otimes dz_{\b b})(d\zeta_{\b K})^{[k]}(d\zeta_L)^{[n-r-k-1]}(dz_L)^{[r-k]}(\b{dz_K})^{[k]}\\
&=-\calL_{Y_{a,b}}\Omega_2.
\end{align*}
This shows that $X_{\epsilon_i\pm\epsilon_j}\omega_{r,k}=X_{\epsilon_i}\omega_{r,k}=0$ for $1\leq i \leq k < j\leq l$.

Finally, let $a,b\in L$ with $a\neq b$ and consider the vector field $Z_{a,b}$ from Lemma \ref{lem:HWf}. If $ n-r-k\geq 1$, then 
\begin{align*}
\calL_{Z_{a,b}}(d\zeta_L)^{[n-r-k]}(dz_L)^{[r-k]}&=(d\zeta_{L_{a,b}})^{[n-r-k-1]}(d\zeta_a\otimes dz_b)(dz_{L_b})^{[r-k]}\\
&\quad+(d\zeta_{L_b})^{[n-r-k]}(dz_a\otimes dz_b)(dz_{L_{a,b}})^{[r-k-1]}\\
&=(d\zeta_{L_{a,b}})^{[n-r-k-1]}(d\zeta_a dz_a\otimes dz_bdz_a)(dz_{L_{a,b}})^{[r-k-1]}\\
&\quad+(d\zeta_{L_{a,b}})^{[n-r-k-1]}(d\zeta_adz_a\otimes dz_adz_b)(dz_{L_{a,b}})^{[r-k-1]}\\
&=0.
\end{align*}
If $n-r-k=0$, we get the same result directly from Lemma~\ref{lem:HWf}.
If $n-r-k\geq 2$, we similarly have
\begin{align*}
& \calL_{Z_{a,b}}\zeta_L(d\zeta_L)^{[n-r-k-1]}(dz_L)^{[r-k]}\\
&\quad=(\zeta_{a}\otimes dz_b)(d\zeta_{L_b})^{[n-r-k-1]}(dz_{L_b})^{[r-k]}\\
&\qquad+\zeta_{L_b}(d\zeta_{L_{a,b}})^{[n-r-k-2]}(d\zeta_a\otimes dz_b)(dz_{L_b})^{[r-k]}\\
&\qquad+\zeta_{L_b}(d\zeta_{L_{b}})^{[n-r-k-1]}(dz_a\otimes dz_b)(dz_{L_{a,b}})^{[r-k-1]}\\
&\quad=(\zeta_{a}\otimes dz_b)(d\zeta_{L_{a,b}})^{[n-r-k-1]}(dz_a\otimes dz_a)(dz_{L_{a,b}})^{[r-k-1]}\\
&\qquad+(\zeta_{a}\otimes dz_b)(d\zeta_{L_{a,b}})^{[n-r-k-2]}(d\zeta_a\otimes dz_a)(dz_{L_{a,b}})^{[r-k]}\\
&\qquad+(\zeta_a\otimes dz_a)(d\zeta_{L_{a,b}})^{[n-r-k-2]}(d\zeta_a\otimes dz_b)(dz_{L_{a,b}})^{[r-k]}\\
&\qquad+\zeta_{L_{a,b}}(d\zeta_{L_{a,b}})^{[n-r-k-2]}(d\zeta_adz_a\otimes dz_bdz_a)(dz_{L_{a,b}})^{[r-k-1]}\\
&\qquad+(\zeta_a\otimes dz_a)(d\zeta_{L_{a,b}})^{[n-r-k-1]}(dz_a\otimes dz_b)(dz_{L_{a,b}})^{[r-k-1]}\\
&\qquad+\zeta_{L_{a,b}}(d\zeta_{L_{a,b}})^{[n-r-k-2]}(d\zeta_a dz_a\otimes dz_a dz_b)(dz_{L_{a,b}})^{[r-k-1]}\\
&\quad=0.
\end{align*}
If $n-r-k=1$, we get the same directly from Lemma~\ref{lem:HWf}.
Hence $X_{\epsilon_i\pm\epsilon_j}\omega_{r,k}=0$ for $k<i<j\leq l$ as well as $X_{\epsilon_i}\omega_{r,k}=0$ for $k<i\leq l$.

Altogether, we have shown that both $\omega_{r,k}$ and $\omega_{l,-l}$ are highest weight vectors in fact.  Using Proposition \ref{pro:HWf}, one immediately verifies that their weights are $\lambda_{k,2}$ and $\lambda_{-l,2}$, respectively.
\end{proof}

\begin{definition}
\label{def:hwv}
For any  $r,k,m\in\NN$ with $r\leq n-1$, $k\leq \min\{r,n-r\}$, and $m\geq2$, let us define a valuation by
\begin{align*}
\phi_{r,k,m}(K)=\frac{(\sqrt{-1})^{\lfloor\frac n2\rfloor}(\sqrt 2)^{m-2}}{s_{n+m-r-3}}\int_{N(K)}\omega_{r,k,m},\quad K\in\calK(\RR^n).
\end{align*}
 If $n=2l$, then we define in addition
\begin{align*}
\phi_{l,-l,m}(K)=\frac{(\sqrt{-1})^l (\sqrt 2)^{m-2}}{s_{l+m-3}}\int_{N(K)}\omega_{l,-l,m},\quad K\in\calK(\RR^n).
\end{align*}
\end{definition}
 Clearly, the mapping $\Omega^{n-1}(S\RR^n)^{\tr}\to \Val^\infty(\RR^n)$ given by \eqref{eq:normalcycle} commutes with the $\SO(n)$-action on the two spaces. Hence, by Theorem~\ref{thm:HWforms}, $\phi_{r,k,m}\in\Val^\infty_r(\RR^n)$, is a highest weight vector of weight $\lambda_{k,m}$ and similarly $\phi_{l,-l,m}\in\Val^{\infty}_l(\RR^{2l})$ is a highest weight vector of weight $\lambda_{-l,m}$.  Observe that $\phi_{r,k,m}\in\Val_r^s(\RR^n)$ as well as $\phi_{l,-l,m}\in\Val_r^s(\RR^{2l})$ for $m\equiv s\mod2$, cf.\ Remark \ref{re:evenWeights}. However, it is a priori not clear that these valuations are not identically zero; this will only become evident from Theorem~\ref{thm:PD} below.

\section{The Rumin differential}

\label{ss:RuminDI}

We will now compute the action of the Rumin differential $D$ on the highest weight vectors $\omega_{r,k,m}$  constructed in the previous section. We will assume
\begin{align}
\label{eq:assumption}
r,k,m\in\NN,\quad r\leq n-1,\quad k\leq \min\{r,n-r\},\quad\text{and}\quad m\geq2
\end{align}
throughout the entire section. Let us also define more (double) forms as follows:
\begin{align*}
\delta_{r,k}\otimes\Theta_1&=(d\zeta_J)^{[n-r]}(dz_J)^{[r-k]}(\b{dz_K})^{[k]},\\
\theta_{r,k}\otimes\Theta_1&=\b{\zeta_K}(\b{d\zeta_K})^{[k-1]}\zeta_J(d\zeta_J)^{[n-r-k]}(dz_J)^{[r-1]},\\
\sigma_{r,k}\otimes\Theta_1&=(\b{d\zeta_K})^{[k]}\zeta_J(d\zeta_J)^{[n-r-k]}(dz_J)^{[r-1]},\\
\tau_{r,k}\otimes\Theta_1&=\b{\zeta_K}(\b{d\zeta_K})^{[k-1]}(d\zeta_J)^{[n-r-k+1]}(dz_J)^{[r-1]}.
\end{align*}

\begin{theorem}
\label{thm:RuminD}
Denote $c_{r,m}=(-1)^{n+1}(n+m-r-2)$. Then
\begin{align*}
D\omega_{r,k,m}&=d\left(\omega_{r,k,m}+c_{r,m}\zeta_{\b1}^{m-2}\theta_{r,k}\alpha\right)\\
&=c_{r,m}\zeta_{\b 1}^{m-2}\left[(m+k-1)\sigma_{r,k}+(-1)^{k+1}(n-r-k+1)\tau_{r,k}\right]\alpha.
\end{align*}
\end{theorem}

The key ingredient in the proof of this result will be the relation \eqref{eq:RuminD1} among $\delta_{r,k}$, $\theta_{r,k}$, and $\sigma_{r,k}$ below whose proof, in turn, decays into a series of auxiliary identities. Let us begin with stating and verifying those.

\begin{proposition}
One has
\begin{align}
\label{eq:RuminD01}
\b{\zeta_K}(\b{d\zeta_K})^{[k-1]}d\zeta_{\b K}&=(-1)^{k-1}(\b{d\zeta_K})^{[k]} \zeta_{\b K},\\[2ex]
\label{eq:RuminD02}
\b{\zeta_K}(\b{d\zeta_K})^{[k-1]}\gamma_{\b K}&=(-1)^{k-1}(\b{d\zeta_K})^{[k]} \nu_K,\\[2ex]
\label{eq:RuminD03}
\delta_{r,k}\otimes\Theta_1&=(-1)^k(\b{d\zeta_K})^{[k]}(d\zeta_L)^{[n-r-k]}(dz_L)^{[r-k]}(dz_{\b K})^{[k]},\\[2ex]
\label{eq:RuminD04}
\begin{split}
\sigma_{r,k}\otimes\Theta_1&=(\b{d\zeta_K})^{[k]}\zeta_{\b K}(d\zeta_L)^{[n-r-k]}(dz_L)^{[r-k]}(dz_{\b K})^{[k-1]}\\
&\quad+(\b{d\zeta_K})^{[k]}\zeta_L(d\zeta_L)^{[n-r-k]}(dz_L)^{[r-k-1]}(dz_{\b K})^{[k]},
\end{split}\\[2ex]
\label{eq:RuminD05}
(d\zeta_L)^{[n-r-k]}(dz_L)^{[r-k]}\alpha_L&=-(d\zeta_L)^{[n-r-k-1]}(dz_L)^{[r-k+1]}\gamma_L,\\[2ex]
\label{eq:RuminD06}
\begin{split}
\zeta_L(d\zeta_L)^{[n-r-k]}(dz_L)^{[r-k-1]}\alpha_L&=(-1)^{n-1}(d\zeta_L)^{[n-r-k]}(dz_L)^{[r-k]}\nu_L\\
&\quad-\zeta_L(d\zeta_L)^{[n-r-k-1]}(dz_L)^{[r-k]}\gamma_L,
\end{split}\\[2ex]
\label{eq:RuminD07}
\zeta_J(d\zeta_J)^{[n-r-k]}(dz_J)^{[r-1]}d\alpha_L&=(-1)^{n-k}(d\zeta_J)^{[n-r-k]}(dz_J)^{[r]}\gamma_L,\\[2ex]
\label{eq:RuminD08}
\begin{split}
\theta_{r,k}d\alpha_{\b K}\otimes\Theta_1&=(-1)^{n-k}\b{\zeta_K}(\b{d\zeta_K})^{[k-1]}(d\zeta_L)^{[n-r-k]}(dz_J)^{[r]}\gamma_K\\
&\quad+(-1)^n(\b{d\zeta_K})^{[k]}\zeta_L(d\zeta_L)^{[n-r-k-1]}(dz_J)^{[r]}\gamma_K.
\end{split}
\end{align}
\end{proposition}

\begin{proof}
To show \eqref{eq:RuminD01}, observe that for each $i\in K$ we have
\begin{align*}
\b{\zeta_K}(\b{d\zeta_K})^{[k-1]}(d\zeta_{\b i}\otimes dz_{\b i})&=\b{\zeta_K}(\b{d\zeta_{K_i}})^{[k-1]}(d\zeta_{\b i}\otimes dz_{\b i})\\
&=(\zeta_{\b i}\otimes dz_i)(\b{d\zeta_{K_i}})^{[k-1]}(d\zeta_{\b i}\otimes dz_{\b i})\\
&=(-1)^{k-1}(\b{d\zeta_{K_i}})^{[k-1]}(d\zeta_{\b i}\otimes dz_i)(\zeta_{\b i}\otimes dz_{\b i})\\
&=(-1)^{k-1}(\b{d\zeta_K})^{[k]}(\zeta_{\b i}\otimes dz_{\b i})
\end{align*}
and sum over $i$. The proof of \eqref{eq:RuminD02} is completely analogous. As for \eqref{eq:RuminD03}, one has
\begin{align*}
(d\zeta_J)^{[n-r]}(dz_J)^{[r-k]}(\b{dz_K})^{[k]}&=(d\zeta_J)^{[n-r]}(dz_L)^{[r-k]}(\b{dz_K})^{[k]}\\
&=(d\zeta_{\b K})^{[k]}(d\zeta_L)^{[n-r-k]}(dz_L)^{[r-k]}(\b{dz_K})^{[k]}
\end{align*}
from which the claim follows since $(d\zeta_{\b K})^{[k]}(\b{dz_K})^{[k]}=(-1)^k(\b{d\zeta_K})^{[k]}(dz_{\b K})^{[k]}$. \eqref{eq:RuminD04} is proven in precisely the same way as \eqref{eq:HWf01} above. Further, for any $i\in L$ we have
\begin{align*}
(d\zeta_L)^{[n-r-k]}(dz_L)^{[r-k]}\zeta_{\b i}dz_i&=(d\zeta_{L_i})^{[n-r-k-1]}(d\zeta_i\otimes dz_i)(dz_{L_i})^{[r-k]}\zeta_{\b i}dz_i\\
&=-(d\zeta_L)^{[n-r-k-1]}(dz_L)^{[r-k+1]}\zeta_{\b i}d\zeta_i
\end{align*}
and
\begin{align*}
&\zeta_L(d\zeta_L)^{[n-r-k]}(dz_L)^{[r-k-1]}\zeta_{\b i}dz_i \\
&\quad=(\zeta_i\otimes dz_i)(d\zeta_{L_i})^{[n-r-k]}(dz_{L_i})^{[r-k-1]}\zeta_{\b i}dz_i\\
&\qquad+\zeta_{L_i}(d\zeta_{L_i})^{[n-r-k-1]}(d\zeta_i\otimes dz_i)(dz_{L_i})^{[r-k-1]}\zeta_{\b i}dz_i \\
&\quad=(-1)^{n-1}(d\zeta_{L_i})^{[n-r-k]}(dz_L)^{[r-k]}\zeta_i\zeta_{\b i}-\zeta_{L_i}(d\zeta_{L_i})^{[n-r-k-1]}(dz_L)^{[r-k]}\zeta_{\b i}d\zeta_i \\
&\quad=(-1)^{n-1}(d\zeta_L)^{[n-r-k]}(dz_L)^{[r-k]}\zeta_i\zeta_{\b i}-(\zeta_i\otimes dz_i)(d\zeta_{L_i})^{[n-r-k-1]}(dz_L)^{[r-k]}\zeta_{\b i}d\zeta_i\\
&\qquad-\zeta_{L_i}(d\zeta_{L_i})^{[n-r-k-1]}(dz_L)^{[r-k]}\zeta_{\b i}d\zeta_i \\
& \quad = (-1)^{n-1}(d\zeta_L)^{[n-r-k]}(dz_L)^{[r-k]}\zeta_i\zeta_{\b i}-\zeta_{L}(d\zeta_{L})^{[n-r-k-1]}(dz_L)^{[r-k]}\zeta_{\b i}d\zeta_i,
\end{align*}
which proves \eqref{eq:RuminD05} and \eqref{eq:RuminD06}, respectively. Similarly, for any $i\in L$ we have
\begin{align*}
&\zeta_J(d\zeta_J)^{[n-r-k]}(dz_J)^{[r-1]}d\zeta_{\b i}dz_i\\
&\quad =(\zeta_i\otimes dz_i)(d\zeta_J)^{[n-r-k]}(dz_{J_i})^{[r-1]}d\zeta_{\b i}dz_i+\zeta_{J_i}(d\zeta_J)^{[n-r-k]}(dz_{J_i})^{[r-1]}d\zeta_{\b i}dz_i \\
&\quad =(-1)^{n-k}(d\zeta_{J_i})^{[n-r-k]}(dz_{J_i})^{[r-1]}(dz_i\otimes dz_i)\zeta_id\zeta_{\b i}\\
&\qquad+\zeta_{J_i}(d\zeta_{J_i})^{[n-r-k-1]}(dz_{J_i})^{[r-1]}(dz_i\otimes dz_i)d\zeta_id\zeta_{\b i} \\
&\quad =(-1)^{n-k}(d\zeta_{J_i})^{[n-r-k]}(dz_J)^{[r]}\zeta_id\zeta_{\b i}+\zeta_{J_i}(d\zeta_{J_i})^{[n-r-k-1]}(dz_J)^{[r]} d\zeta_id\zeta_{\b i} \\
&\quad =(-1)^{n-k}(d\zeta_J)^{[n-r-k]}(dz_J)^{[r]}\zeta_id\zeta_{\b i}\\
&\qquad-(-1)^{n-k}(d\zeta_{J_i})^{[n-r-k-1]}(d\zeta_i\otimes dz_i)(dz_J)^{[r]}\zeta_id\zeta_{\b i}\\
&\qquad+\zeta_J(d\zeta_{J_i})^{[n-r-k-1]}(dz_J)^{[r]} d\zeta_id\zeta_{\b i} \\
&\qquad-(\zeta_i\otimes dz_i)(d\zeta_{J_i})^{[n-r-k-1]}(dz_J)^{[r]} d\zeta_id\zeta_{\b i} \\
&\quad =(-1)^{n-k}(d\zeta_J)^{[n-r-k]}(dz_J)^{[r]}\zeta_id\zeta_{\b i}+\zeta_J(d\zeta_J)^{[n-r-k-1]}(dz_J)^{[r]} d\zeta_id\zeta_{\b i}.
\end{align*}
Taking the sum over $i\in L$, the second term in the resulting expression clearly does not contribute while the first yields \eqref{eq:RuminD07}. Finally, let us prove \eqref{eq:RuminD08}. First, for any $i\in K$ we have
\begin{align*}
&\b{\zeta_K}(\b{d\zeta_K})^{[k-1]}\zeta_J(d\zeta_J)^{[n-r-k]}(dz_J)^{[r-1]}d\zeta_i dz_{\b i}\\
&\quad=\b{\zeta_K}(\b{d\zeta_K})^{[k-1]}\zeta_J(d\zeta_L)^{[n-r-k]}(dz_{J_{\b i}})^{[r-1]}d\zeta_i dz_{\b i}\\
&\qquad+\b{\zeta_K}(\b{d\zeta_K})^{[k-1]}\zeta_Jd\zeta_{\b K}(d\zeta_L)^{[n-r-k-1]}(dz_{J_{\b i}})^{[r-1]}d\zeta_i dz_{\b i}.
\end{align*}
Concerning the first term,
\begin{align*}
&\b{\zeta_K}(\b{d\zeta_K})^{[k-1]}\zeta_J(d\zeta_L)^{[n-r-k]}(dz_{J_{\b i}})^{[r-1]}d\zeta_i dz_{\b i}\\
&\quad=\b{\zeta_K}(\b{d\zeta_K})^{[k-1]}(\zeta_{\b i}\otimes dz_{\b i})(d\zeta_L)^{[n-r-k]}(dz_{J_{\b i}})^{[r-1]}d\zeta_i dz_{\b i}\\
&\quad=(-1)^{n-k}\b{\zeta_K}(\b{d\zeta_K})^{[k-1]}(d\zeta_L)^{[n-r-k]}(dz_{J_{\b i}})^{[r-1]}(dz_{\b i}\otimes dz_{\b i})\zeta_{\b i}d\zeta_i \\
&\quad=(-1)^{n-k}\b{\zeta_K}(\b{d\zeta_K})^{[k-1]}(d\zeta_L)^{[n-r-k]}(dz_J)^{[r]}\zeta_{\b i}d\zeta_i.
\end{align*}
As for the second, by \eqref{eq:RuminD01} we have
\begin{align*}
&\b{\zeta_K}(\b{d\zeta_K})^{[k-1]}\zeta_Jd\zeta_{\b K}(d\zeta_L)^{[n-r-k-1]}(dz_{J_{\b i}})^{[r-1]}d\zeta_i dz_{\b i}\\
&\quad= (-1)^k(\b{d\zeta_K})^{[k]}\zeta_{\b K}\zeta_L(d\zeta_L)^{[n-r-k-1]}(dz_{J_{\b i}})^{[r-1]}d\zeta_i dz_{\b i}\\
&\quad= (-1)^k(\b{d\zeta_K})^{[k]}(\zeta_{\b i}\otimes dz_{\b i})\zeta_L(d\zeta_L)^{[n-r-k-1]}(dz_{J_{\b i}})^{[r-1]}d\zeta_i dz_{\b i}\\
&\quad= (-1)^n(\b{d\zeta_K})^{[k]}\zeta_L(d\zeta_L)^{[n-r-k-1]}(dz_J)^{[r]}\zeta_{\b i}d\zeta_i
\end{align*}
and summing over $i$ thus finishes the proof.
\end{proof}

\begin{lemma}
\label{lem:RuminDRelations}
The following forms are trivial on $\RR^{n}\times\RR^n$:
\begin{enuma}
\item $-\theta_{r,k}d\alpha_K+(-1)^n\sigma_{r,k}\alpha_K$,
\item $\delta_{r,k}\nu_K+(-1)^n\sigma_{r,k}\alpha_{\b K}$,
\item $\theta_{r,k}d\alpha_{\b K}\gamma_K$,

\item $\sigma_{r,k}\gamma_{\b K}$,
\item $\delta_{r,k}\gamma_{\b K}$,

\item $\delta_{r,k}\nu_K\gamma_K-\theta_{r,k}d\alpha_{\b K}\gamma_{\b K}$,
\item $\delta_{r,k}\nu_L\gamma_L+(-1)^n\sigma_{r,k}\alpha_L\gamma_L$,
\item $\delta_{r,k}\nu_K\gamma_L-\theta_{r,k}d\alpha_L\gamma_{\b K}$,
\item $\theta_{r,k}d\alpha_L\gamma_L$,
\item $\left(\delta_{r,k}\nu_L-\theta_{r,k}d\alpha_L+(-1)^n\sigma_{r,k}\alpha_L\right)\gamma_K-\theta_{r,k}d\alpha_{\b K}\gamma_L$.
\end{enuma}
\end{lemma}

\begin{proof}
Obviously, each $\beta\in\Omega^*(\RR^n\times\RR^n)$ is trivial if and only if $\beta\otimes\Theta_1=0$. This rule will be used throughout the proof.
\begin{enuma}
\item For each $i\in K$ one has
\begin{align*}
&\b{\zeta_K}(\b{d\zeta_K})^{[k-1]}\zeta_J(d\zeta_J)^{[n-r-k]}(dz_J)^{[r-1]}d\zeta_{\b i}dz_i\\
&\quad=\b{\zeta_K}(\b{d\zeta_{K_i}})^{[k-1]}\zeta_J(d\zeta_J)^{[n-r-k]}(dz_J)^{[r-1]}d\zeta_{\b i}dz_i\\
&\quad=(\zeta_{\b i}\otimes dz_i)(\b{d\zeta_{K_i}})^{[k-1]}\zeta_J(d\zeta_J)^{[n-r-k]}(dz_J)^{[r-1]}d\zeta_{\b i}dz_i\\
&\quad=(-1)^n(d\zeta_{\b i}\otimes dz_i)(\b{d\zeta_{K_i}})^{[k-1]}\zeta_J(d\zeta_J)^{[n-r-k]}(dz_J)^{[r-1]}\zeta_{\b i}dz_i\\
&\quad=(-1)^n(\b{d\zeta_K})^{[k]}\zeta_J(d\zeta_J)^{[n-r-k]}(dz_J)^{[r-1]}\zeta_{\b i}dz_i
\end{align*}
and the claim follows by summing over $i$.

\item According to \eqref{eq:RuminD04}, for any $i\in K$ we have
\begin{align*}
&(\b{d\zeta_K})^{[k]}\zeta_J(d\zeta_J)^{[n-r-k]}(dz_J)^{[r-1]}\zeta_{i}dz_{\b i}\\
&\quad=(\b{d\zeta_K})^{[k]}(\zeta_{\b i}\otimes dz_{\b i})(d\zeta_L)^{[n-r-k]}(dz_L)^{[r-k]}(dz_{\b K_{\b i}})^{[k-1]}\zeta_{i}dz_{\b i}\\
&\quad=(-1)^{n+k-1}(\b{d\zeta_K})^{[k]}(d\zeta_L)^{[n-r-k]}(dz_L)^{[r-k]}(dz_{\b K_{\b i}})^{[k-1]}(dz_{\b i}\otimes dz_{\b i})\zeta_{i}\zeta_{\b i}\\
&\quad=(-1)^{n+k-1}(\b{d\zeta_K})^{[k]}(d\zeta_L)^{[n-r-k]}(dz_L)^{[r-k]}(dz_{\b K})^{[k]}\zeta_{i}\zeta_{\b i}\\
&\quad=(-1)^{n-1}(d\zeta_{\b K})^{[k]}(d\zeta_L)^{[n-r-k]}(dz_J)^{[r-k]}(\b {dz_K})^{[k]}\zeta_{i}\zeta_{\b i}\\
&\quad=(-1)^{n-1}(d\zeta_J)^{[n-r]}(dz_J)^{[r-k]}(\b {dz_K})^{[k]}\zeta_{i}\zeta_{\b i}
\end{align*}
and the claim follows by summing over $i$.

\item An immediate consequence of \eqref{eq:RuminD08}.

\item Obvious.

\item This follows at once when one employs the relation \eqref{eq:RuminD03}.

\item Using \eqref{eq:RuminD02} for the first equality and \eqref{eq:RuminD03} for the last, for any $i\in K$ we have
\begin{align*}
&\b{\zeta_K}(\b{d\zeta_K})^{[k-1]}\zeta_J(d\zeta_J)^{[n-r-k]}(dz_J)^{[r-1]}d\zeta_{i}dz_{\b i}\gamma_{\b K} \\
&\quad=(-1)^n(\b{d\zeta_K})^{[k]}\zeta_J(d\zeta_L)^{[n-r-k]}(dz_{J_{\b i}})^{[r-1]}d\zeta_{i}dz_{\b i}\nu_K\\
&\quad=(-1)^n(\b{d\zeta_K})^{[k]}(\zeta_{\b i}\otimes dz_{\b i})(d\zeta_L)^{[n-r-k]}(dz_{J_{\b i}})^{[r-1]}d\zeta_{i}dz_{\b i}\nu_K\\
&\quad=(-1)^k(\b{d\zeta_K})^{[k]}(d\zeta_L)^{[n-r-k]}(dz_J)^{[r]}\zeta_{\b i}d\zeta_{i}\nu_K\\
&\quad=(-1)^k(\b{d\zeta_K})^{[k]}(d\zeta_L)^{[n-r-k]}(dz_L)^{[r-k]}(dz_{\b K})^{[k]}\zeta_{\b i}d\zeta_{i}\nu_K\\
&\quad=(d\zeta_J)^{[n-r]}(dz_J)^{[r-k]}(\b{dz_K})^{[k]}\zeta_{\b i}d\zeta_{i}\nu_K.
\end{align*}

\item From \eqref{eq:RuminD04}, \eqref{eq:RuminD05}, \eqref{eq:RuminD06}, and \eqref{eq:RuminD03} we infer 
\begin{align*}
&(\b{d\zeta_K})^{[k]}\zeta_J(d\zeta_J)^{[n-r-k]}(dz_J)^{[r-1]}\alpha_L\\
&\quad\equiv (-1)^{n+1}(d\zeta_J)^{[n-r]}(dz_J)^{[r-k]}(\b{dz_K})^{[k]}\nu_L\mod \gamma_L
\end{align*}
which is clearly equivalent to the claim.

\item
Using \eqref{eq:RuminD02} and \eqref{eq:RuminD07} for the first equality, we obtain
\begin{align*}
&\b{\zeta_K}(\b{d\zeta_K})^{[k-1]}\zeta_J(d\zeta_J)^{[n-r-k]}(dz_J)^{[r-1]}d\alpha_L\gamma_{\b K}\\
&\quad=(-1)^k(\b{d\zeta_K})^{[k]}(d\zeta_J)^{[n-r-k]}(dz_J)^{[r]}\nu_K\gamma_L\\
&\quad=(-1)^k(\b{d\zeta_K})^{[k]}(d\zeta_L)^{[n-r-k]}(dz_L)^{[r-k]}(dz_{\b K})^k\nu_K\gamma_L,
\end{align*}
which by \eqref{eq:RuminD03} implies the claim.
\item An immediate consequence of \eqref{eq:RuminD07}.

\item First, according to \eqref{eq:RuminD07} and \eqref{eq:RuminD01},
\begin{align*}
\theta_{r,k}d\alpha_L\otimes\Theta_1&=(-1)^{n-k}\b{\zeta_K}(\b{d\zeta_K})^{[k-1]}(\b{d\zeta_L})^{[n-r-k]}(dz_J)^{[r]}\gamma_L\\
&\quad+(-1)^{n+1}(\b{d\zeta_K})^{[k]}\zeta_{\b K}(\b{d\zeta_L})^{[n-r-k-1]}(dz_J)^{[r]}\gamma_L.
\end{align*}
Second, using \eqref{eq:RuminD03}--\eqref{eq:RuminD06} we infer
\begin{align*}
&\big[(-1)^n\sigma_{r,k}\alpha_L+\delta_{r,k}\nu_L\big]\otimes\Theta_1\\
&\quad=(-1)^{n+1}(\b{d\zeta_K})^{[k]}\zeta_{\b K}(\b{d\zeta_L})^{[n-r-k-1]}(dz_L)^{[r-k+1]}(dz_{\b K})^{[k-1]}\gamma_L\\
&\qquad+(-1)^{n+1}(\b{d\zeta_K})^{[k]}\zeta_L(\b{d\zeta_L})^{[n-r-k-1]}(dz_L)^{[r-k]}(dz_{\b K})^{[k]}\gamma_L.
\end{align*}
Finally, by \eqref{eq:RuminD08},
\begin{align*}
\theta_{r,k}d\alpha_{\b K}\otimes\Theta_1&=(-1)^{n-k}\b{\zeta_K}(\b{d\zeta_K})^{[k-1]}(d\zeta_L)^{[n-r-k]}(dz_J)^{[r]}\gamma_K\\
&\quad+(-1)^n(\b{d\zeta_K})^{[k]}\zeta_L(d\zeta_L)^{[n-r-k-1]}(dz_J)^{[r]}\gamma_K.
\end{align*}
Multiplying by $\gamma_K$ and $\gamma_L$, respectively, and summing everything up, the claim easily follows.
\end{enuma}
\end{proof}

\begin{corollary}
The following relation holds in $ \Omega^*(S\RR^n)$:
\begin{align}
\label{eq:RuminD1}
\theta_{r,k}d\alpha=\delta_{r,k}+(-1)^n\sigma_{r,k}\alpha.
\end{align}
\end{corollary}

\begin{proof}
Notice that $\alpha=\alpha_K+\alpha_{\b K}+\alpha_L$ (and analogously for $\gamma$) and $\nu=2\nu_K+\nu_L$. Consequently, one immediately verifies that the form
\begin{align*}
\delta_{r,k}\nu\gamma-\theta_{r,k}d\alpha\gamma+(-1)^n\sigma_{r,k}\alpha\gamma
\end{align*}
 belongs to the ideal of $ \Omega^*(\RR^n\times\RR^n)$ generated by items (a)--(j) of Lemma \ref{lem:RuminDRelations}. According to the statement of the lemma, it must be trivial, which  by Lemma~\ref{lemma:RestrictionSphere} implies the claim if we take into account that $\nu=1$ on $S\RR^n$.
\end{proof}

Let us show three more auxiliary identities before we will finally proceed to the proof of the theorem.
\begin{proposition}
One has
\begin{align}
\label{eq:RuminD2}
d\omega_{r,k}&=(n-r)\delta_{r,k},\\
\label{eq:RuminD3}
d\zeta_{\b 1}\omega_{r,k}&=\zeta_{\b1}\delta_{r,k},\text{ and}\\
\label{eq:RuminD4}
d\zeta_{\b 1}\theta_{r,k}&=\zeta_{\b 1}\sigma_{r,k},\\
\label{eq:RuminD5}
d\theta_{r,k}&=k\sigma_{r,k}+(-1)^{k+1}(n-r-k+1)\tau_{r,k}.
\end{align}
\end{proposition}

\begin{proof}
\eqref{eq:RuminD2} and \eqref{eq:RuminD5} follow by direct computation. To prove \eqref{eq:RuminD3}, observe that 
\begin{align*}
&d\zeta_{\b 1}\zeta_J(d\zeta_J)^{[n-r-1]}(dz_J)^{[r-k]}(\b{dz_K})^{[k]}\\
&\quad=d\zeta_{\b 1}(\zeta_{\b1}\otimes dz_{\b 1})(d\zeta_{J_{\b1}})^{[n-r-1]}(dz_L)^{[r-k]}(\b{dz_K})^{[k]}\\
&\quad=\zeta_{\b 1}(d\zeta_{\b1}\otimes dz_{\b 1})(d\zeta_{J_{\b 1}})^{[n-r-1]}(dz_{J})^{[r-k]}(\b{dz_K})^{[k]}\\
&\quad=\zeta_{\b 1}(d\zeta_J)^{[n-r]}(dz_J)^{[r-k]}(\b{dz_K})^{[k]}.
\end{align*}
Finally, as for \eqref{eq:RuminD4}, one has
\begin{align*}
&d\zeta_{\b 1}\b{\zeta_K}(\b{d\zeta_K})^{[k-1]}\zeta_J(d\zeta_J)^{[n-r-k]}(dz_J)^{[r-1]}\\
&\quad=d\zeta_{\b 1}(\zeta_{\b1}\otimes dz_1)(\b{d\zeta_{K_1}})^{[k-1]}\zeta_J(d\zeta_J)^{[n-r-k]}(dz_J)^{[r-1]}\\
&\quad=\zeta_{\b 1}(\b{d\zeta_K})^{[k]}\zeta_J(d\zeta_J)^{[n-r-k]}(dz_J)^{[r-1]}.
\end{align*}
\end{proof}

\begin{proof}[Proof of Theorem \ref{thm:RuminD}]
Using the relations \eqref{eq:RuminD1}--\eqref{eq:RuminD4}, we compute
\begin{align*}
&d\left(\omega_{r,k,m}+c_{r,m}\zeta_{\b1}^{m-2}\theta_{r,k}\alpha\right)\\
&=(m-2)\zeta_{\b1}^{m-3}d\zeta_{\b1}(\omega_{r,k}+c_{r,m}\theta_{r,k}\alpha)+\zeta_{\b1}^{m-2}\big(d\omega_{r,k}+c_{r,m}d\theta_{r,k}\alpha+(-1)^nc_{r,m}\theta_{r,k}d\alpha\big)\\
&=(m-2)\zeta_{\b 1}^{m-2}(\delta_{r,k}+c_{r,m}\sigma_{r,k}\alpha)\\
&\quad+\zeta_{\b1}^{m-2}\big((n-r)\delta_{r,k}+c_{r,m}d\theta_{r,k}\alpha+(-1)^nc_{r,m}\delta_{r,k}+c_{r,m}\sigma_{r,k}\alpha\big)\\
&=\zeta_{\b1}^{m-2}\big(m-2+n-r+(-1)^nc_{r,m}\big)\delta_{r,k}+c_{r,m}\zeta_{\b1}^{m-2}\big((m-1)\sigma_{r,k}+d\theta_{r,k}\big)\alpha.
\end{align*}
The first term in the resulting expression clearly vanishes and the claim follows at once from \eqref{eq:RuminD5}.
\end{proof}

It remains to determine the action of $D$ on the highest weight vectors $\omega_{l,-l,m}$. However, this turns out to be an easy consequence of a special case of Theorem \ref{thm:RuminD}. In analogy with the previous definitions, let us denote
\begin{align*}
\theta_{l,-l}\otimes\Theta_1&=\b{\zeta_M}(\b{d\zeta_M})^{[l-1]}\zeta_{\b M}(dz_{\b M})^{[l-1]},\\
\sigma_{l,-l}\otimes\Theta_1&=(\b{d\zeta_M})^{[l]}\zeta_{\b M}(dz_{\b M})^{[l-1]},\\
\tau_{l,-l}\otimes\Theta_1&=\b{\zeta_M}(\b{d\zeta_M})^{[l-1]}d\zeta_{\b M}(dz_{\b M})^{[l-1]}.
\end{align*}

\begin{corollary}
\label{cor:RuminD}
 Assume $n=2l$. For any $m\geq2$ one has
\begin{align*}
D\omega_{l,-l,m}&=d\left(\omega_{l,-l,m}+(2-l-m)\zeta_{\b1}^{m-2}\theta_{l,-l}\alpha\right)\\
&=(2-l-m)\zeta_{\b 1}^{m-2}\left[(m+l-1)\sigma_{l,-l}+(-1)^{l+1}\tau_{l,-l}\right]\alpha.
\end{align*}
\end{corollary}

\begin{proof}
Let $e_1,\dots,e_n$ be the canonical basis of $\RR^n$ and let $R\in\OO(n)$ be the reflection in the hyperplane $e_n^\perp$. Observe that $R^*\omega_{l,l,m}=\omega_{l,-l,m}$, $R^*\theta_{l,l}=\omega_{l,-l}$, $R^*\sigma_{l,l}=\sigma_{l,-l}$,  $R^*\tau_{l,l}=\tau_{l,-l}$, and $R^*\alpha=\alpha$. Thus, since both $d$ and $D$ commute with $R^*$, the claim follows at once from Theorem \ref{thm:RuminD}.
\end{proof}

\section{The Alesker--Poincar\'e pairing}

The goal of this section is to evaluate the Alesker--Poincar\'e pairing of the highest weight vectors constructed  above. It will again suffice to consider only the valuations $\phi_{r,k,m}$ for $k\geq1$ inasmuch as the remaining cases  ($n=2l$ and $k=-l$) can be then easily deduced by applying a hyperplane reflection.  In this connection, let us keep the assumption \eqref{eq:assumption} throughout the entire section again.

\begin{theorem}
\label{thm:PD}
One has
\begin{align*}
\b{\phi_{r,k,m}}*\phi_{n-r,k,m} = (-1)^k c_{r,k,m}\frac{v_{n+2m-2}}{v_{n+m-r-2}v_{r+m-2}s_{2m-3}},
\end{align*}
where
\begin{align*}
c_{r,k,m}=(m+k-1)(n+m-k){n-2k\choose r-k}.
\end{align*}
\end{theorem}

Proceeding similarly as in the previous section, we will first take the advantage of the double-form formalism in order to prove the important auxiliary relation \eqref{eq:wDw} below. Its proof, again, splits further into a series of simpler statements with which we will start.

\begin{proposition}
One has
\begin{align}
\label{eq:PD001}
\b{\eta_K}(\b{d\eta_K})^{[k-1]}d\eta_{\b K}&=(-1)^{k-1}(\b{d\eta_K})^{[k]} \eta_{\b K},\\
\label{eq:PD002}
\b{\eta_K}(\b{d\eta_K})^{[k-1]}\gamma_{\b K}&=(-1)^{k-1}(\b{d\eta_K})^{[k]} \nu_K,\\
\label{eq:PD002.5}
\b{\zeta_{\b K}}(\b{d\zeta_{\b K}})^{[k-1]}\gamma_{K}&=(-1)^{k-1}(\b{d\zeta_{\b K}})^{[k]} \nu_K,\\
\label{eq:PD003}
\eta_{\b K}(dw_{\b K})^{[k-1]}\alpha_{\b K}&=(-1)^{k-1}(dw_{\b K})^{[k]} \nu_K.
\end{align}
\end{proposition}

\begin{proof}
See the proof of  \eqref{eq:RuminD01}.
\end{proof}

\begin{proposition}
One has
\begin{align} 
\label{eq:PD004}
\begin{split}
\b{\omega_{r,k}}\otimes\Theta_1&=\b{\zeta_{\b K}}(\b{d\zeta_{\b K}})^{[k-1]}(\b{d\zeta_L})^{[n-r-k]}(\b{dz_L})^{[r-k]}(dz_K)^{[k]}\\
&\quad+\b{\zeta_{L}}(\b{d\zeta_{\b K}})^{[k]}(\b{d\zeta_L})^{[n-r-k-1]}(\b{dz_L})^{[r-k]}(dz_K)^{[k]},
\end{split}\\[2ex]
\label{eq:PD005}
\begin{split}
\sigma_{n-r,k}\otimes\Theta_2&=(\b{d\eta_K})^{[k]}\eta_{\b K}(d\eta_L)^{[r-k]}(dw_{\b K})^{[k-1]}(dw_L)^{[n-r-k]}\\
&\quad+(\b{d\eta_K})^{[k]}\eta_L(d\eta_L)^{[r-k]}(dw_{\b K})^{[k]}(dw_L)^{[n-r-k-1]},
\end{split}\\[2ex]
\label{eq:PD006}
\begin{split}
\tau_{n-r,k}\otimes\Theta_2&=\b{\eta_K}(\b{d\eta_K})^{[k-1]}(d\eta_L)^{[r-k+1]}(dw_{\b K})^{[k]}(dw_L)^{[n-r-k-1]}\\
&\quad+(-1)^{k-1}(\b{d\eta_K})^{[k]}\eta_{\b K}(d\eta_L)^{[r-k]}(dw_{\b K})^{[k-1]}(dw_L)^{[n-r-k]}.
\end{split}
\end{align}
\end{proposition}

\begin{proof}
 \eqref{eq:PD004} and \eqref{eq:PD005} follow from \eqref{eq:HWf01} and \eqref{eq:RuminD04}, respectively. \eqref{eq:PD006} follows from
\begin{align*}
&\b{\eta_K} (\b{d\eta_K})^{[k-1]} (d\eta_J)^{[r-k+1]} (dw_J)^{[n-r-1]}\\
 &\quad = \b{\eta_K} (\b{d\eta_K})^{[k-1]} (d\eta_L)^{[r-k+1]} (dw_J)^{[n-r-1]}\\
&\qquad+ \b{\eta_K} (\b{d\eta_K})^{[k-1]}  d\eta_{\b K}(d\eta_L)^{[r-k]} (dw_J)^{[n-r-1]}\\
&\quad = \b{\eta_K} (\b{d\eta_K})^{[k-1]} (d\eta_L)^{[r-k+1]}(dw_{\b K})^{[k]}(dw_L)^{[n-r-k-1]}\\
& \qquad +  \b{\eta_K} (\b{d\eta_K})^{[k-1]}  d\eta_{\b K}(d\eta_L)^{[r-k]} (dw_{\b K})^{[k-1]}(dw_L)^{[n-r-k]},
\end{align*}
using \eqref{eq:PD001}.
\end{proof}

\begin{proposition}
One has
\begin{align}
\label{eq:PD007}
(d\eta_L)^{[r-k]}(dw_L)^{[n-r-k]}\alpha_L&=-(d\eta_L)^{[r-k-1]}(dw_L)^{[n-r-k+1]}\gamma_L,\\[2ex]
\label{eq:PD008}
\begin{split}
\eta_L(d\eta_L)^{[r-k]}(dw_L)^{[n-r-k-1]}\alpha_L&=(-1)^{n-1}(d\eta_L)^{[r-k]}(dw_L)^{[n-r-k]}\nu_L\\
&\quad-\eta_L(d\eta_L)^{[r-k-1]}(dw_L)^{[n-r-k]}\gamma_L,
\end{split}\\[2ex]
\label{eq:PD009}
\begin{split}
(\b{d\zeta_{\b K}})^{[k]}(dz_K)^{[k]}(\b{d\eta_K})^{[k]}(dw_{\b K})^{[k]}&=\Theta_K\otimes\Theta_K,
\end{split}\\[2ex]
\label{eq:PD010}
\begin{split}
(\b{d\zeta_L})^{[n-r-k]}(\b{dz_L})^{[r-k]}(d\eta_L&)^{[r-k]}(dw_L)^{[n-r-k]}\\
&=(-1)^{n+l+r}{n-2k\choose r-k}\Theta_L\otimes\Theta_L,
\end{split}\\[2ex]
\label{eq:PD011}
\begin{split}
\b{\zeta_{L}}(\b{d\zeta_L})^{[n-r-k-1]}(\b{dz_L})^{[r-k]}(d\eta_L&)^{[r-k]}(dw_L)^{[n-r-k]}\gamma_L\\
&=(-1)^{n+l+r+1}{n-2k-1\choose r-k}\nu_L\Theta_L\otimes\Theta_L,
\end{split}\\[2ex]
\label{eq:PD012}
\begin{split}
\eta_L(\b{d\zeta_L})^{[n-r-k]}(\b{dz_L})^{[r-k]}(d\eta_L&)^{[r-k]}(dw_L)^{[n-r-k-1]}\alpha_L\\
&=(-1)^{n+l+r+1}{n-2k-1\choose r-k}\nu_L\Theta_L\otimes\Theta_L.
\end{split}
\end{align}
\end{proposition}

\begin{remark}
Observe that if $k=n-r$, both sides of \eqref{eq:PD011} as well as of \eqref{eq:PD012} are zero by convention.
\end{remark}

\begin{proof}
\eqref{eq:PD007} and \eqref{eq:PD008} are analogous to \eqref{eq:RuminD05} and \eqref{eq:RuminD06}, respectively.
\eqref{eq:PD009} is straightforward.  To prove \eqref{eq:PD010}, observe first
\begin{align*}
(\b{d\zeta_L})^{[n-r-k]}(\b{dz_L})^{[r-k]}&=\sum_{\substack{I\subset L\\|I|=r-k}}(\b{d\zeta_{L\setminus I}})^{[n-r-k]}(\b{dz_I})^{[r-k]}\\
&=\sum_{\substack{I\subset L\\|I|=r-k}}\b{\Theta_{2,L\setminus I}}\,\b{\Theta_{1,I}}\otimes\Theta_{1,L\setminus I}\Theta_{1,I}.
\end{align*}
Second, since $\b L=L$, for any $I\subset L$ with $|I|=r-k$ there are $\varepsilon_1,\varepsilon_2\in\{\pm1\}$ with
\begin{align*}
\b{\Theta_{i,L\setminus I}}=\varepsilon_1\Theta_{i,L\setminus \b I}\quad\text{and}\quad \b{\Theta_{i,I}}=\varepsilon_2\Theta_{i,\b I},\quad i=1,2,
\end{align*}
and therefore,
\begin{align*}
&\left(\b{\Theta_{2,L\setminus I}}\,\b{\Theta_{1,I}}\otimes\Theta_{1,L\setminus I}\Theta_{1,I}\right)(dw_L)^{[n-r-k]}(d\eta_L)^{[r-k]}\\
&\quad=\varepsilon_1\varepsilon_2\Theta_{2,L\setminus \b I}\;\Theta_{1,\b I}\;\Theta_{1,L\setminus \b I}\;\Theta_{2,\b I}\otimes\Theta_{1,L\setminus I}\;\Theta_{1,I} \;\Theta_{2,L\setminus \b I}\;\Theta_{2,\b I}\\
&\quad=\Theta_{2,L\setminus \b I}\;\Theta_{1,\b I}\;\Theta_{1,L\setminus \b I}\;\Theta_{2,\b I}\otimes\b{\Theta_{1,L\setminus \b I}}\;\b{\Theta_{1,\b I}}\;\Theta_{2,L\setminus \b I}\;\Theta_{2,\b I}\\
&\quad=(-1)^{n-r-k}\Theta_{1,L\setminus \b I}\;\Theta_{1,\b I}\;\Theta_{2,L\setminus \b I}\;\Theta_{2,\b I}\otimes\b{\Theta_{1,L\setminus \b I}}\;\b{\Theta_{1,\b I}}\;\Theta_{2,L\setminus \b I}\;\Theta_{2,\b I}\\
&\quad=(-1)^{n-r-k}\Theta_{1,L}\Theta_{2,L}\otimes\b{\Theta_{1,L}}\Theta_{2,L}\\
&\quad=(-1)^{n+l+r}\Theta_{1,L}\Theta_{2,L}\otimes\Theta_{1,L}\Theta_{2,L}\\
&\quad=(-1)^{n+l+r}\Theta_L\otimes\Theta_L.
\end{align*}
Since there are ${n-2k\choose r-k}$ subsets of $ L$ of cardinality $r-k$, the claim follows.  As for \eqref{eq:PD011}, repeating the preceding argument shows that for any $i\in L$ with $i\neq l+1$ one  has
 \begin{align*} & (\b{d\zeta_{L_i}})^{[n-r-k-1]}(\b{dz_{L_i}})^{[r-k]}(d\eta_{L_{\b i}})^{[r-k]}(dw_{L_{\b i}})^{[n-r-k-1]}\\
 & \quad  = (-1)^{n+l+r}\binom{n-2k-1}{r-k}\Theta_{1,L_{\b i}}\Theta_{2,L_{\b i}}\otimes\Theta_{1,L_i}\Theta_{2,L_{\b i}}
 \end{align*}
 while, if $n=2l+1$,
 \begin{align*} & (\b{d\zeta_{L_{l+1}}})^{[n-r-k-1]}(\b{dz_{L_{l+1}}})^{[r-k]}(d\eta_{L_{l+1}})^{[r-k]}(dw_{L_{l+1}})^{[n-r-k-1]}\\
 & \quad  = (-1)^{n+l+r+1}\binom{n-2k-1}{r-k}\Theta_{1,L_{l+1}}\Theta_{2,L_{l+1}}\otimes\Theta_{1,L_{l+1}}\Theta_{2,L_{l+1}}.
 \end{align*}
Altogether, for any $i\in L$,
\begin{align*}
&\b{\zeta_{L}}(\b{d\zeta_L})^{[n-r-k-1]}(\b{dz_L})^{[r-k]}(d\eta_L)^{[r-k]}(dw_L)^{[n-r-k]}\zeta_id\zeta_{\b i}\\
&\quad=\b{\zeta_{L}}(\b{d\zeta_{L_i}})^{[n-r-k-1]}(\b{dz_L})^{[r-k]}(d\eta_{L_{\b i}})^{[r-k]}(dw_L)^{[n-r-k]}\zeta_id\zeta_{\b i}\\
&\quad=\b{\zeta_{L}}(\b{d\zeta_{L_i}})^{[n-r-k-1]}(\b{dz_{L_i}})^{[r-k]}(d\eta_{L_{\b i}})^{[r-k]}(dw_{L_{\b i}})^{[n-r-k-1]}(dz_{\b i}\otimes d\zeta_{\b i})\zeta_id\zeta_{\b i}\\
&\quad=(\zeta_{\b i}\otimes dz_i)(\b{d\zeta_{L_i}})^{[n-r-k-1]}(\b{dz_{L_i}})^{[r-k]}(d\eta_{L_{\b i}})^{[r-k]}(dw_{L_{\b i}})^{[n-r-k-1]}(dz_{\b i}\otimes d\zeta_{\b i})\zeta_id\zeta_{\b i}\\
&\quad=(\b{d\zeta_{L_i}})^{[n-r-k-1]}(\b{dz_{L_i}})^{[r-k]}(d\eta_{L_{\b i}})^{[r-k]}(dw_{L_{\b i}})^{[n-r-k-1]}(dz_{\b i}d\zeta_{\b i}\otimes dz_id\zeta_{\b i})\zeta_i\zeta_{\b i}\\
&\quad=(-1)^{n+l+r+1}{n-2k-1\choose r-k}(\Theta_L\otimes\Theta_L)\zeta_i\zeta_{\b i}.
\end{align*}
The proof of \eqref{eq:PD012} is completely similar.
\end{proof}

\begin{proposition}
One has
\begin{align}
\label{eq:PD01}
\b{\omega_{r,k}}\alpha_K&=0,\\
\label{eq:PD02}
\sigma_{n-r,k}\gamma_{\b K}&=0,\\
\label{eq:PD03}
\tau_{n-r,k}\alpha_{\b K}\gamma_{\b K}&=0,\\
\label{eq:PD04}
\b{\omega_{r,k}}\gamma_K\otimes\Theta_1&=(-1)^{n+1}(\b{d\zeta_{\b K}})^{[k]}(\b{d\zeta_L})^{[n-r-k]}(\b{dz_L})^{[r-k]}(dz_K)^{[k]}\nu_K,\\
\label{eq:PD05}
\sigma_{n-r,k}\alpha_{\b K}\otimes\Theta_2&=(-1)^{n+k+1}(\b{d\eta_K})^{[k]}(d\eta_L)^{[r-k]}(dw_{\b K})^{[k]}(dw_L)^{[n-r-k]}\nu_K,\\
\label{eq:PD06}
\tau_{n-r,k}\alpha_{\b K}\otimes\Theta_2&=(-1)^n(\b{d\eta_K})^{[k]}(d\eta_L)^{[r-k]}(dw_{\b K})^{[k]}(dw_L)^{[n-r-k]}\nu_K,\\
\label{eq:PD07}
\begin{split}
\sigma_{n-r,k}\alpha_L\otimes\Theta_2&\equiv(-1)^{n+k+1}(\b{d\eta_K})^{[k]}(d\eta_L)^{[r-k]}(dw_{\b K})^{[k]}(dw_L)^{[n-r-k]}\nu_L\\
&\quad\mod\gamma_L,
\end{split}\\
\label{eq:PD08}
\begin{split}
\tau_{n-r,k}\alpha_L\otimes\Theta_2&=(-1)^{k}(\b{d\eta_K})^{[k]}\eta_{\b K}(dw_{\b K})^{[k-1]}(d\eta_L)^{[r-k-1]}(dw_L)^{[n-r-k+1]}\gamma_L\\
&\quad-\b{\eta_K}(\b{d\eta_K})^{[k-1]}(dw_{\b K})^{[k]}(d\eta_L)^{[r-k]}(dw_L)^{[n-r-k]}\gamma_L.
\end{split}
\end{align}
\end{proposition}

\begin{proof}
\eqref{eq:PD01} and \eqref{eq:PD02} are obvious. \eqref{eq:PD03}  follows at once from \eqref{eq:PD006}.  \eqref{eq:PD04} further follows from \eqref{eq:PD004} and \eqref{eq:PD002.5}; \eqref{eq:PD05} similarly from \eqref{eq:PD005} and \eqref{eq:PD003}; and \eqref{eq:PD06} from \eqref{eq:PD006} and \eqref{eq:PD003}. For \eqref{eq:PD07}, one uses \eqref{eq:PD005}, \eqref{eq:PD007}, and \eqref{eq:PD008}. Finally, \eqref{eq:PD08} follows easily from \eqref{eq:PD006} and \eqref{eq:PD007}.
\end{proof}

Let us define
	$$ \vol_{S\RR^n} = dx_1\cdots dx_n \sum_{i=1}^n (-1)^{i+1} \xi_i d\xi_1 \cdots \widehat{d\xi_i} \cdots d\xi_n \in \Omega^{2n-1}(\RR^n\times \RR^n),$$
where $\widehat{d\xi_i}$ means that $d\xi_i	$ is omitted. Restricted to $S\RR^n$, this form is the product of the standard Riemannian volume forms on $\RR^n$ and $S^{n-1}$ (cf. \cite[Example 15.22]{Lee:SM}).

\begin{lemma}
\label{lem:PD}
One has
\begin{align}
\label{eq:wDw}
\b{\omega_{r,k,m}}\wedge D\omega_{n-r,k,m}=(-1)^{r+k}|\zeta_1|^{2(m-2)}\left(a_{r,k,m}\nu_K^2+b_{r,k,m}\nu_L\right)\vol_{S\RR^n},
\end{align}
where
\begin{align*}
a_{r,k,m}&=(m+r-2)(m+r){n-2k\choose r-k},\\
b_{r,k,m}&=(m+r-2)(m+k-1){n-2k-1\choose r-k}.
\end{align*}
\end{lemma}

\begin{proof}
In view of Theorem \ref{thm:RuminD}, consider on  $\RR^n\times\RR^n$ the form
\begin{equation}\label{eq:wDw_gamma}
c_{n-r,m}|\zeta_1|^{2(m-2)}\,\b{\omega_{r,k}}\left[(m+k-1)\sigma_{n-r,k}+(-1)^{k+1}(r-k+1)\tau_{n-r,k}\right]\alpha\gamma,
\end{equation}
where
\begin{align*}
c_{n-r,m}=(-1)^{n+1}(m+r-2).
\end{align*}
We will treat the two summands in \eqref{eq:wDw_gamma} separately. Let us begin with the one proportional to $\sigma_{n-r,k}$.  Note that according to \eqref{eq:PD01} and \eqref{eq:PD02},
\begin{align*}
\b{\omega_{r,k}}\sigma_{n-r,k}\alpha\gamma=\b{\omega_{r,k}}\sigma_{n-r,k}(\alpha_{\b K}+\alpha_L)(\gamma_K+\gamma_L).
\end{align*}
Using \eqref{eq:PD04}, \eqref{eq:PD05}, \eqref{eq:PD009}, and \eqref{eq:PD010}, we get
\begin{align*}
&\b{\omega_{r,k}}\sigma_{n-r,k}\alpha_{\b K}\gamma_K\otimes\Theta=(-1)^{k+l+r}{n-2k\choose r-k}\nu_K^2\Theta\otimes\Theta.
\end{align*}
According to \eqref{eq:PD004}, \eqref{eq:PD05}, \eqref{eq:PD009}, and \eqref{eq:PD011},
\begin{align*}
&\b{\omega_{r,k}}\sigma_{n-r,k}\alpha_{\b K}\gamma_L\otimes\Theta=(-1)^{k+l+r}{n-2k-1\choose r-k}\nu_K\nu_L\Theta\otimes\Theta.
\end{align*}
Similarly, according to \eqref{eq:PD005}, \eqref{eq:PD04}, \eqref{eq:PD009} and \eqref{eq:PD012},
\begin{align*}
&\b{\omega_{r,k}}\sigma_{n-r,k}\alpha_{L}\gamma_K\otimes\Theta=(-1)^{k+l+r}{n-2k-1\choose r-k}\nu_K\nu_L\Theta\otimes\Theta.
\end{align*}
Finally, from \eqref{eq:PD004}, \eqref{eq:PD07}, \eqref{eq:PD009}, and \eqref{eq:PD011} we infer
\begin{align*}
\b{\omega_{r,k}}\sigma_{n-r,k}\alpha_{L}\gamma_L\otimes\Theta=(-1)^{k+l+r}{n-2k-1\choose r-k}\nu_L^2\Theta\otimes\Theta.
\end{align*}
Altogether, since $\nu=2\nu_K+\nu_L$, one has
\begin{align}
\label{eq:wDw1}
\b{\omega_{r,k}}\sigma_{n-r,k}\alpha\gamma=(-1)^{k+l+r}\left[{n-2k\choose r-k}\nu_K^2+{n-2k-1\choose r-k}\nu_L\nu\right]\Theta.
\end{align}

Proceeding to the second term of \eqref{eq:wDw}, by \eqref{eq:PD01}, \eqref{eq:PD03}, and \eqref{eq:PD08}, we have
\begin{align*}
\b{\omega_{r,k}}\tau_{n-r,k}\alpha\gamma=\b{\omega_{r,k}}\tau_{n-r,k}\big[\alpha_{\b K}(\gamma_K+\gamma_L)+\alpha_L(\gamma_K+\gamma_{\b K})\big].
\end{align*}
According to \eqref{eq:PD04}, \eqref{eq:PD06}, \eqref{eq:PD009}, and \eqref{eq:PD010},
\begin{align*}
&\b{\omega_{r,k}}\tau_{n-r,k}\alpha_{\b K}\gamma_K\otimes\Theta=(-1)^{l+r+1}{n-2k\choose r-k}\nu_K^2\Theta\otimes\Theta.
\end{align*}
By \eqref{eq:PD004}, \eqref{eq:PD06}, \eqref{eq:PD009}, and \eqref{eq:PD011},
\begin{align*}
&\b{\omega_{r,k}}\tau_{n-r,k}\alpha_{\b K}\gamma_L\otimes\Theta=(-1)^{l+r+1}{n-2k-1\choose r-k}\nu_K\nu_L\Theta\otimes\Theta.
\end{align*}
Further, using \eqref{eq:PD006} and \eqref{eq:PD04}, we infer that for degree reasons
\begin{align*}
&\b{\omega_{r,k}}\tau_{n-r,k}\alpha_{L}\gamma_K\otimes\Theta=0.
\end{align*}
Finally, according to \eqref{eq:PD08}, \eqref{eq:PD002}, \eqref{eq:PD004}, \eqref{eq:PD009}, and \eqref{eq:PD011},
\begin{align*}
&\b{\omega_{r,k}}\tau_{n-r,k}\alpha_{L}\gamma_{\b K}\otimes\Theta=(-1)^{l+r}{n-2k-1\choose r-k}\nu_K\nu_L\Theta\otimes\Theta.
\end{align*}
Together, therefore,
\begin{align}
\label{eq:wDw2}
\b{\omega_{r,k}}\tau_{n-r,k}\alpha\gamma=(-1)^{l+r+1}{n-2k\choose r-k}\nu_K^2\Theta.
\end{align}

Plugging \eqref{eq:wDw1} and \eqref{eq:wDw2} back in \eqref{eq:wDw_gamma}, we get (on $\RR^n\times\RR^n$)
\begin{align*}
\eqref{eq:wDw_gamma} &=(-1)^{l+r+k}c_{n-r,m}|\zeta_1|^{2(m-2)}\\
&\quad\times\left[(m+r){n-2k\choose r-k}\nu_K^2+(m+k-1){n-2k-1\choose r-k}\nu_L\nu\right]\Theta.
\end{align*}
Further, using
\begin{align*}
\Theta=(\sqrt{-1})^{2l}dx_1\cdots dx_nd\xi_1\cdots d\xi_n=(-1)^{n+l+1}\vol_{S\RR^n}\gamma,
\end{align*}
 Lemma~\ref{lemma:RestrictionSphere}, and the fact that $\nu=1$ on $S\RR^n$, we conclude that (on $S\RR^n$)
\begin{align*}
&\b{\omega_{r,k,m}}\wedge D\omega_{n-r,k,m}\\
&\quad=(-1)^{r+k}(m+r-2)|\zeta_1|^{2(m-2)}\\
&\qquad\times\left[(m+r){n-2k\choose r-k}\nu_K^2+(m+k-1){n-2k-1\choose r-k}\nu_L\right]\vol_{S\RR^n},
\end{align*}
which exactly was to be shown.
\end{proof}

\begin{proof}[Proof of Theorem \ref{thm:PD}]
According to Theorem \ref{thm:BernigFormula}, Proposition \ref{pro:PvsC}, and Lemma \ref{lem:PD},
\begin{align*}
&\b{\phi_{r,k,m}}* \phi_{n-r,k,m}\\
&\quad=\frac{(-1)^k}{s_{n+m-r-3}s_{r+m-3}}\int_{S^{n-1}}2^{m-2}|\zeta_1|^{2(m-2)}\left(a_{r,k,m}\nu_K^2+b_{r,k,m}\nu_L\right)d\sigma.
\end{align*}
After change of variables and using $ s_a= \frac{2\pi}{a-1} s_{a-2}$ we have
\begin{align*}
&\int_{S^{n-1}} 2^{m-2}|\zeta_1|^{2(m-2)}\left(a_{r,k,m}\nu_K^2+b_{r,k,m}\nu_L\right)d\sigma\\
&\quad=s_{n-2k-1}  \int_0^{\pi/2} \left(\frac14 a_{r,k,m} \cos^4 t + b_{r,k,m} \sin^2 t \right) \cos^{2m+2k-5} t \sin^{n-2k-1} t dt \\
&\qquad\times \int_{S^{2k-1}}  2^{m-2} |\zeta_1|^{2(m-2) }   d\sigma  \\
&\quad=s_{n-2k-1}\left(\frac14 a_{r,k,m}\frac{s_{n+2m-1}}{s_{2m+2k-1}s_{n-2k-1}}+b_{r,k,m}\frac{s_{n+2m-3}}{s_{2m+2k-5}s_{n-2k+1}}\right)\\
&\qquad\times s_1s_{2k-3}\int_0^{\pi/2} \cos^{2m-3} t  \sin^{2k-3} t dt  \\
&\quad=2\pi s_{n-2k-1}\left(  \frac14 a_{r,k,m} \frac{s_{n+2m-1}}{s_{2m+2k-1}s_{n-2k-1}} + b_{r,k,m} \frac{s_{n+2m-3}}{s_{2m+2k-5}s_{n-2k+1}}\right) \frac{s_{2m+2k-5}}{s_{2m-3}}\\
&\quad=\left[  a_{r,k,m} \frac{(m+k-1)(m+k-2)}{n+2m-2} + b_{r,k,m}(n-2k)\right] \frac{s_{n+2m-3}}{s_{2m-3}}\\
&\quad=\frac{(r+m-2)(m+k-1)(n-r+m-2)(n+m-k)}{n+2m-2}{n-2k\choose r-k}\frac{s_{n+2m-3}}{s_{2m-3}}
\end{align*}
and hence the claim follows at once from $s_{a-1}=av_a$.
\end{proof}

\section{Pullback and pushforward}

\label{s:PP}

The goal of this section is to show that our choice of highest weight vectors in $\Val(\RR^n)$ is compatible with the natural transition between different dimensions, more precisely, with pullback and pushforward under inclusions $\RR^{n}\rightarrow\RR^N$ and projections $\RR^N\rightarrow\RR^n$, respectively.

 Throughout the section we will assume $n\geq3$ and identify $\RR^n=\RR^{n-1}\oplus\langle e_n\rangle$. To avoid any confusion, $N^{(d)}(K)$ denotes the normal cycle with respect to the $d$-dimensional ambient space and similarly the superscript $(d)$ is used to highlight the dimension valuations and (double) forms are considered in.

\subsection{Pullback}

Let $\iota\maps{\RR^{n-1}}{\RR^n}$ be the inclusion.

\begin{theorem}
\label{thm:pullback}
For any $r,k,m\in\NN$ with $r \leq n-1$, $k\leq \min\{r,n-r\}$, and $m\geq2$,
\begin{align*}
\iota^*\phi_{r,k,m}^{(n)}&=\begin{cases}\phi_{r,k,m}^{(n-1)}&\text{if }r<n-1\text{ and } k<n-r, \\[1ex] \frac 1 2\phi_{k,k-1,m}^{(n-1)} & \text{if } k=\frac n2, \\[1ex]0&\text{otherwise}.\end{cases}
\end{align*}
\end{theorem}

For the proof, we will consider the manifold $M=\RR^{n-1}\times\RR^{n-1}\times [-\pi/2, \pi/2]$ and the map $F\maps{M}{\RR^n\times \RR^n}$ given by
\begin{align*}
F(x,u,\vartheta) = \big(\iota(x),\cos (\vartheta) \iota(u) + \sin(\vartheta) e_n\big).
\end{align*} 

\begin{lemma}
\label{pro:Fstar}
If $k<\frac n2$, then
\begin{align*}
	   F^* z_K^{(n)} &  = z_K^{(n-1)},\\
	  F^* z_J^{(n)} &  =z_J^{(n-1)},\\
	     F^*  \zeta_J ^{(n)}& =  \cos (\vartheta) \, \zeta_J^{(n-1)} + \begin{cases}
		\sqrt{-1} \sin(\vartheta) \otimes \sqrt{-1} dx_n,& n=2l,\\
		\sin(\vartheta) \otimes dx_n, & n=2l+1.
		\end{cases}
\end{align*}
If $k=\frac n 2$, then
\begin{align*} 
	F^* z_K^{(n)}& =
	 z_{K_l}^{(n-1)} + \frac{1}{\sqrt 2} x_{2l-1} \otimes dz_l,\\
	F^* \zeta_{\b K}^{(n)}& = \cos (\vartheta)  
	 \zeta_{\b{K_l}}^{(n-1)} + \frac{1}{\sqrt 2}\big(\cos(\vartheta) \zeta_{2l-1}-\sqrt{-1} \sin(\vartheta)\big) \otimes  dz_{\b l}.
\end{align*}
\end{lemma}

\begin{proof}
Straightforward.
\end{proof}

\begin{proof}[Proof of Theorem \ref{thm:pullback}]
Let $K\in\calK(\RR^{n-1})$ have smooth boundary and positive Gauss curvature. Denote
\begin{align*}
N_0^\pm= \Big\{ \big(\iota (x),\pm e_n\big)\mid x\in \operatorname{int} K\Big\}.
\end{align*}
Then, as currents,
\begin{align*} 
N^{(n)}(\iota K)= N_0^+ +  N_0^- +  F_{*}\big(N^{(n-1)}(K)\times[-\pi/2,\pi/2]\big).
\end{align*}
It is obvious if $r<n-1$ and follows from Proposition \ref{pro:HWf} and from our assumption $n>2$ if $r=n-1$ that $ \int_{N_0^\pm}\omega_{r,k,m}^{(n)}=0$.
Consequently, 
\begin{align*}
\int_{N^{(n)}(\iota K)}\omega_{r,k,m}^{(n)} & =\int_{N^{(n-1)}(K)\times[-\pi/2,\pi/2]}F^*\omega_{r,k,m}^{(n)}\\
& =\int_{N^{(n-1)}(K)}(-1)^n\left(\int_{-\pi/2}^{\pi/2}  i_{\pder{\vartheta}} F^*\omega_{r,k,m}^{(n)} d\vartheta\right).
\end{align*}
In particular, $\iota^*\phi_{r,k,m}^{(n)}=0$ if $r=n-1$. We  will assume from now on that $r<n-1$.

 Assume moreover $k<\frac n2 $.  If we write $F^* d\zeta_J^{(n)}= \cos\vartheta d\zeta_J^{(n-1)} + \upsilon$, then 
$$ F^* \zeta_J^{(n)} \wedge  \upsilon =\zeta_J^{(n-1)}  (\wt c_n  d\vartheta\otimes  \wt c_n dx_n)
$$ 
 with
 \begin{align}
 	\label{eq:tildecn}
 	\wt c_n=\begin{cases}\sqrt{-1}&\text{if }n=2l,\\[1ex] 1 &\text{if }n=2l+1.\end{cases}
 \end{align}
 Observe also that 
 \begin{equation}
 \label{eq:vol_dimension}
 \Theta_1^{(n)}=\wt c_n (-1)^{n+1}  \Theta_1^{(n-1)}dx_{n}.
 \end{equation}
 Using Lemma~\ref{pro:Fstar},
  we thus  compute 
\begin{align*}
&(-1)^n i_{\pder{\vartheta}} F^*\omega_{r,k,m}^{(n)}\otimes\Theta_1^{(n)}\\
&\quad=(-1)^n i_{\pder{\vartheta}}F^*  \left[  \left(\zeta_{\b 1}^{(n)}\right)^{m-2}\zeta_J^{(n)}\left(d\zeta_J^{(n)}\right)^{[n-r-1]}\left(d z_J^{(n)}\right)^{[r-k]}\left(\b{d z_K^{(n)}}\right)^{[k]}\right] \\
&\quad= \cos^{m+n-r-4} (\vartheta) \left(\zeta_{\b 1}^{(n-1)}\right)^{m-2}\zeta_J^{(n-1)}\\
&\qquad\wedge\left(d\zeta_J^{(n-1)}\right)^{[n-r-2]}\left(d z_J^{(n-1)}\right)^{[r-k]}\left(\b{d  z_K^{(n-1)}}\right)^{[k]} (\wt c_n \otimes  \wt c_n dx_n) \\
& \quad =   \wt c_n \cos^{m+n-r-4} (\vartheta)  \omega_{r,k,m}^{(n-1)} \otimes  \wt c_n  \Theta_1^{(n-1)} dx_n \\
& \quad =  \wt c_n(-1)^{n+1}    \cos^{m+n-r-4} (\vartheta)  \omega_{r,k,m}^{(n-1)} \otimes  \Theta_1^{(n)}  \\
\end{align*}
and hence we get
\begin{align*}
(-1)^n \int_{-\pi/2}^{\pi/2}  i_{\pder{\vartheta}} F^*\omega_{r,k,m}^{(n)} d\vartheta=\wt c_n(-1)^{n+1}\frac{s_{m+n-r-3}}{s_{m+n-r-4}}\omega_{r,k,m}^{(n-1)}.
\end{align*}
For $k=n-r$ we  thus have $\iota^*\phi_{r,k,m}^{(n)}=0$, see Remark \ref{re:HWf}(b). For $k<n-r$, however, 
\begin{align*}
\phi_{r,k,m}^{(n)}(\iota K)=
\begin{cases}
\frac{(\sqrt{-1})^{l-1}(\sqrt 2)^{m-2}}{s_{n+m-r-4}}\int_{N^{(n-1)}(K)}\omega_{r,k,m}^{(n-1)}&\text{if }n=2l,\\[1ex]
\frac{(\sqrt{-1})^{l}(\sqrt 2)^{m-2}}{s_{n+m-r-4}}\int_{N^{(n-1)}(K)}\omega_{r,k,m}^{(n-1)}&\text{if }n=2l+1,
\end{cases}
\end{align*}
i.e., $\phi_{r,k,m}^{(n)}(\iota K)=\phi_{r,k,m}^{(n-1)}(K)$. Since any convex body can be approximated by smooth, positively curved ones, the claim follows.

Let us consider the remaining case $k=\frac n2$. Similarly as above, writing 
$$F^*\left(d \zeta_{\b K} ^{(n)}\right) =  \cos(\vartheta) d\zeta_{\b{K_l}}^{(n-1)} + 
\frac{1}{\sqrt 2}\cos(\vartheta)d\xi_{2l-1}\otimes dz_{\b l}  + \upsilon,$$  we have 
$$ F^* \zeta_{\b K}^{(n)} \wedge  \upsilon =-\frac{\sqrt{-1}}{\sqrt 2} \zeta_{\b{K_l}} ^{(n-1)} (d\vartheta\otimes dz_{\b l}). $$
Using Proposition~\ref{pro:HWf}  and Lemma \ref{pro:Fstar}, we obtain 
\begin{align*}
& i_{\pder{\vartheta}} F^*\omega_{l,l,m}^{(n)}\otimes\Theta_1^{(n)}\\
&\quad= i_{\pder{\vartheta}}F^*  \left[  \left(\zeta_{\b 1}^{(n)}\right)^{m-2}\zeta_{\b K}^{(n)}\left(d\zeta_{\b K}^{(n)}\right)^{[l-1]}\left(\b{d z_K^{(n)}}\right)^{[l]}\right] \\
&\quad = -\frac{\sqrt{-1}}{2} \cos^{l+m-4}(\vartheta) \left(\zeta_{\b 1}^{(n-1)}\right)^{m-2}\zeta_{\b{K_l} }^{(n-1)} \\
& \qquad \wedge 
\left(d\zeta_{\b{ K_l}}^{(n-1)}\right)^{[l-2 ]}\left(\b{d z_{K_l}^{(n-1)}}\right)^{[l-1]}  (dx_{2l-1} \otimes   -\sqrt{-1} dx_{2l-1} dx_{2l})\\
& \quad = -\frac{\sqrt{-1}}{2} \cos^{l+m-4}(\vartheta) \; \omega_{l,l-1, m}^{(n-1)} \otimes   -\sqrt{-1} \Theta_1^{(n-1)}  dx_{2l}\\
& \quad =-\frac{\sqrt{-1}}{2} \cos^{l+m-4}(\vartheta) \; \omega_{l,l-1, m}^{(n-1)} \otimes \Theta_1^{(n)}  
\end{align*}
Hence 
\begin{align*}
 \int_{-\pi/2}^{\pi/2}  i_{\pder{\vartheta}} F^*\omega_{l,l,m}^{(n)} d\vartheta &= -\frac{\sqrt{-1}}{ 2} \int_{-\pi/2}^{\pi/2} \cos^{l+m-2}(\vartheta) 
	\omega_{l,l-1, m}^{(n-1)} 
	= -\frac{\sqrt{-1}}{ 2} \frac{s_{l+m-3}}{s_{l+m-4}}  	\omega_{l,l-1, m}^{(n-1)}
\end{align*}
and therefore
$$ \phi_{l,l,m}^{(n)}(
\iota K) = \frac{1 }{ 2}\frac{ (\sqrt{-1})^{l-1}  (\sqrt 2)^{m- 2}}{ s_{l+m-4}} \int_{N^{(n-1)}(K)} \omega_{l,l-1,m}^{(n-1)} = 
 \frac{1 }{ 2} \phi_{l,l-1,m}^{(n-1)}(K).
 $$

\end{proof}

\subsection{Pushforward}

Let $\pi\maps{\RR^n}{\RR^{n-1}}$ be the orthogonal projection.

\begin{theorem}
\label{thm:pushforward}
For any $r,k,m\in\NN$ with $r \leq n-1$, $k\leq \min\{r,n-r\}$, and $m\geq2$,
\begin{align*}
\pi_* \phi_{r,k,m}^{(n)} &= \begin{cases}\phi_{r-1,k,m}^{(n-1)}&\text{if }k<r,\\[1ex]-\frac 12 \phi_{k-1,k-1,m}^{(n-1)}&\text{if }k=\frac n2,\\[1ex] 0&\text{otherwise}.\end{cases}
\end{align*}

\end{theorem}

Before carrying out the explicit computation, let us first prove a general lemma. For its proof we will consider the family of maps $f_t\maps{\RR^n\times\RR^n}{\RR^n\times\RR^n}$, $t\in\RR$,
\begin{align*}
f_t(x,u)=(x+te_n,u).
\end{align*}
Similarly as above we also define $M=\RR^{n-1}\times\RR^{n-1}\times\RR$ and $F\maps{M}{\RR^n\times \RR^n}$ by
\begin{align*}
F(x,u,t)=f_t\big(\iota(x),\iota(u)\big).
\end{align*}

\begin{lemma}
\label{lem:pushforward}
Let $\phi\in\Val^\infty(\RR^n)$ be represented by a translation-invariant form $\omega\in \Omega^{n-1}(S\RR^n)^{\tr}$. Then $\pi_* \phi$ is represented by $ (-1)^n \iota^* \left(i_{\pder{x_n}} \omega\right)$.
\end{lemma}

\begin{proof}
Take any $K\in\calK(\RR^{n-1})$ with smooth boundary and positive Gauss curvature and denote
\begin{align*}
N_0^\pm&= \Big\{ \big(\iota(x),\pm e_n\big)\mid x\in \operatorname{int} K\Big\},\\
N^\pm &= \Big\{ \big(\iota(x),\cos (\vartheta)\iota( u)\pm \sin (\vartheta) e_n \big)\mid (x,u)\in N^{(n-1)}(K), \vartheta\in (0,\pi/2]\Big \}.
\end{align*}
First, the normal cycle of $K_t=\iota K+t[0,e_n]$ splits into the following  sum of currents:
\begin{align*}
N^{(n)}(K_t)= (f_t)_*(N_0^+) + N_0^- + (f_t)_*(N^+)+ N^-  + F_*\big(N^{(n-1)}(K)\times [0,t]\big).
\end{align*}
Second, the translation invariance of $\omega$ implies in particular $f_t^*\omega=\omega$. Third, 
 we clearly have $F^*dx_i=\iota^*dx_i$ for $1\leq i<n$, $F^*dx_n=dt$, and $F^*d\xi_i=\iota^*d\xi_i$ for $1\leq i\leq n$. Altogether, we conclude that
\begin{align*}
\phi(K_t)-\phi(K_0)=\int_{N^{(n-1)}(K)\times [0,t]} F^* \omega= t\int_{N^{(n-1)}(K)}  (-1)^n \iota^* \left(i_{\pder{x_n}} \omega\right).
\end{align*}
Since by definition $(\pi_* \phi)(K)  =\dt \phi(K_t)$, $K\in\calK(\RR^{n-1})$, the claim follows by approximation.
\end{proof}

\begin{lemma}\label{lemma:contraction}
If $k<\frac n2$, then 
$$ i_{\pder{x_n}} dz_J^{(n)} =  \begin{cases}
\sqrt{-1}  \otimes \sqrt{-1} dx_n,& n=2l,\\
1 \otimes dx_n, & n=2l+1
\end{cases}
$$ and
\begin{align*}
\iota^* z_J^{(n)} & = z_J^{(n-1)}, \\
\iota^* \zeta_J^{(n)} & = \zeta_J^{(n-1)}.
\end{align*}  
If $k=\frac n2$, then 
$$  i_{\pder{x_n}} dz_K^{(n)}= \frac{\sqrt{-1}}{\sqrt{2}} \otimes dz_l$$
and 
\begin{align*}
	\iota^* z_K^{(n)} & = z_{K_l}^{(n-1)}+ \frac{1}{\sqrt 2} x_{2l-1}\otimes dz_l, \\
	\iota^* \zeta_{\b K}^{(n)} & = \zeta_{\b{K_l}}^{(n-1)}+  \frac{1}{\sqrt 2 } \xi_{2l-1} \otimes dz_{\b l}.
\end{align*}  

\end{lemma}
\begin{proof}
Straightforward.
\end{proof}

Given a double form $\omega$ and a smooth vector field $X$ on the same manifold, the contraction $i_X \omega$ acts by definition on the first factor.  If $\omega$ is a $(p,q)$-form,  then $i_X(\omega\theta)= (i_X\omega)\theta + (-1)^p \omega (i_X \theta)$. In particular, if $\omega$ is a $(1,1)$-form, then
$$ i_X \omega^{[m]}=  (i_X \omega) \omega^{[m-1]}.$$

\begin{proof}[Proof of Theorem \ref{thm:pushforward}]
Let us first assume $k<\frac n2 $. For $k<r$, using Lemma~\ref{lemma:contraction} and  \eqref{eq:vol_dimension}, we obtain
\begin{align*}
 &\iota^*  \left(i_{\pder{x_n}}  \omega_{r,k,m}^{(n)}\right)\otimes \Theta_1^{(n)} \\
&\quad= \iota^*  \left[i_{\pder{x_n}}\left(\left(\zeta_{\b 1}^{(n)}\right)^{m-2}\zeta_J^{(n)} \left(d\zeta_J^{(n)}\right)^{[n-r-1]} \left(dz_J^{(n)}\right)^{[r-k]} \left(\overline{dz_K^{(n)}}\right)^{[k]} \right)\right] \\
 &\quad= (-1)^{n} \iota^* \left[\left(\zeta_{\b 1}^{(n)}\right)^{m-2}\zeta_J^{(n)} \left(d\zeta_J^{(n)}\right)^{[n-r-1]} \left(dz_J^{(n)}\right)^{[r-k-1]} \left(\overline{dz_K^{(n)}}\right)^{[k]}\left(\wt c_n \otimes \wt c_n dx_n\right) \right] \\
 &\quad= (-1)^{n}  \left(\zeta_{\b 1}^{(n-1)}\right)^{m-2}\zeta_J^{(n-1)}\\
&\qquad\wedge \left(d\zeta_J^{(n-1)}\right)^{[n-r-1]} \left(dz_J^{(n-1)}\right)^{[r-k-1]} \left(\overline{dz_K^{(n-1)}}\right)^{[k]}\left(\wt c_n \otimes \wt c_n dx_n\right)\\
 &\quad=(-1)^n\wt c_n \omega_{r-1,k,m}^{(n-1)}\otimes  \wt c_n \Theta_1^{(n-1)} dx_n^{(n)}\\
 &\quad=-\wt c_n  \omega_{r-1,k,m}^{(n-1)}\otimes \Theta_1^{(n)},
\end{align*}
where $\wt c_n$ is given by \eqref{eq:tildecn}.  By Lemma \ref{lem:pushforward}, for any $K\in\calK(\RR^{n-1})$ we thus have

\begin{align*}
\left(\pi_*\phi_{r,k,m}^{(n)}\right)(K)=
\begin{cases}
\frac{(\sqrt{-1})^{l-1}(\sqrt 2)^{m-2}}{s_{n+m-r-3}}\int_{N^{(n-1)}(K)}\omega_{r-1,k,m}^{(n-1)}&\text{if }n=2l,\\[1ex]
\frac{(\sqrt{-1})^l(\sqrt 2)^{m-2}}{s_{n+m-r-3}}\int_{N^{(n-1)}(K)}\omega_{r-1,k,m}^{(n-1)}&\text{if }n=2l+1,
\end{cases}
\end{align*}
i.e., $\pi_*\phi_{r,k,m}^{(n)}=\phi_{r-1,k,m}^{(n-1)}$ which exactly was to be shown. 
If $k=r$, the same computation clearly gives zero.

Assume now that  $k=\frac n2 $.  Then, by Lemma \ref{lemma:contraction},
\begin{align*}
	&\iota^*  \left(i_{\pder{x_n}} \omega_{l,l,m}^{(n)}\right)\otimes\Theta_1^{(n)}\\
&\quad= \iota^*  \left[i_{\pder{x_n}}\left(\left(\zeta_{\b 1}^{(n)}\right)^{m-2}\zeta_{\b K}^{(n)} \left(d\zeta_{\b K}^{(n)}\right)^{[l-1]}  \left(\overline{dz_K^{(n)}}\right)^{[l]} \right)\right] \\
 &\quad= -\frac{ \sqrt{-1}}{\sqrt{2}}  \iota^* \left[\left(\zeta_{\b 1}^{(n)}\right)^{m-2} \zeta_{\b K}^{(n)} \left(d\zeta_{\b K}^{(n)}\right)^{[l-1]}  \left(\overline{dz_K^{(n)}}\right)^{[l-1]} (1 \otimes dz_l) \right] \\
  &\quad= -\frac{ \sqrt{-1}}{2}  \left(\zeta_{\b 1}^{(n-1)}\right)^{m-2} \zeta_{ \b{K_l}\cup\{l\}}^{(n-1)} \left(d\zeta_{ \b{K_l}\cup\{l\}}^{(n-1)}\right)^{[l-1]}  \left(\overline{dz_{K_l}^{(n-1)}}\right)^{[l-1]} (1 \otimes \sqrt{-1} dx_{2l})  \\
  & \quad = -\frac{ \sqrt{-1}}{2} \omega_{l-1,l-1,m}^{(n-1)} \otimes \sqrt{-1} \Theta_1^{(n-1)} dx_{2l}\\
   & \quad = \frac{ \sqrt{-1}}{2} \omega_{l-1,l-1,m}^{(n-1)} \otimes  \Theta_1^{(n)} 
\end{align*}
Thus we have
\begin{align*}
	\left(\pi_*\phi_{l,l,m}^{(n)}\right)(K)=-\frac{1}{2}\frac{(\sqrt{-1})^{l-1}(\sqrt 2)^{m-2}}{s_{l+m-3}}\int_{N^{(n-1)}(K)}\omega_{r-1,k,m}^{(n-1)}=-\frac{1}{2}\phi_{l-1,l-1,m}^{(n-1)}(K),
\end{align*}
as claimed.
\end{proof}

\section{The Fourier transform}

\label{s:Fourier}

In the section to follow, an explicit description of the Fourier transform in terms of its action on the highest weight vectors is presented. Let us fix an inner product to identify $V\cong V^*\cong\RR^n$ and $\Dens(V)\cong \Dens(V^*)\cong\CC$, and thus regard $\FF$ as a linear operator on $\Val^\infty(\RR^n)$.

\begin{theorem}
\label{thm:FFhwv}
 For any $r,k,m\in\NN$ with $r\leq n-1$, $k\leq \min\{r,n-r\}$, and $m\geq2$,
\begin{align}
\label{eq:FFhwv1}
\FF\phi_{r,k,m}=(-1)^{k-1}(\sqrt{-1})^m\phi_{n-r,k,m}.
\end{align}
 If $n=2l$, then further
\begin{align}
\label{eq:FFhwv2}
\FF\phi_{l,-l,m}=(-1)^l(\sqrt{-1})^m\phi_{l,-l,m}.
\end{align}
\end{theorem}

Since equivariant maps take highest weight vectors to highest weight vectors  of the same weight, according to Theorem \ref{thm:ABS} there are constants $f_{r,k,m}^{(n)}\in\CC$ such that
\begin{align*}
\FF\phi_{r,k,m}^{(n)}=f_{r,k,m}^{(n)}\phi_{ n-r,k,m}^{(n)},
\end{align*}
where the notation is kept from \S\ref{s:PP}. We first prove Theorem \ref{thm:FFhwv} in a special case.

\begin{proposition}
\label{pro:FFdim2}
For any $m\geq2$ one has
\begin{align*}
f_{1,1,m}^{(2)}=(\sqrt{-1})^m.
\end{align*}
\end{proposition}

\begin{proof}
Put  $\rho_1=\xi_1dx_2-\xi_2dx_1\in\Omega^1(S\RR^2)$. It is easily verified (e.g., by approximation by polygons) that
\begin{align*}
\int_{S^1}h(\xi)dS_1(K,\xi)=\int_{N(K)}h\rho_1
\end{align*}
holds for any $K\in\calK(\RR^2)$ and $h\in C^\infty(S^1)$, cf. \cite{BernigHug:Tensor}. Since
\begin{align*}
2\zeta_1dz_{\b 1}=\alpha-\sqrt{-1}\rho_1,
\end{align*}
we have on $S\RR^2$
\begin{align*}
\omega_{1,1,m}= -\zeta_{\b1}^{m-1}dz_{\b1}=-2\zeta_{\b1}^m\zeta_1dz_{\b1}\equiv \sqrt{-1}\zeta_{\b1}^m\rho_1\mod\alpha
\end{align*}
and consequently, for some $c_m\in\CC$,
\begin{align*}
\phi_{1,1,m}^{(2)}=c_m\int_{S^1}\zeta_{\b 1}^{m}dS_1(\Cdot,\xi).
\end{align*}
Let $J\maps{S^1}{S^1}$ be the counter-clockwise rotation by $\pi/2$. If $m$ is even, the function $\zeta_{\b 1}^{m}\maps{S^1}{\CC}$ is even and
\begin{align*}
\zeta_{\b 1}^{m}\circ J=(-\sqrt{-1})^m\zeta_{\b 1}^{m}=(\sqrt{-1})^m\zeta_{\b 1}^{m}.
\end{align*}
If $m$ is odd, then $\zeta_{\b 1}^{m}$ is odd and antiholomorphic and
\begin{align*}
\zeta_{\b 1}^{m}\circ J=(-\sqrt{-1})^m\zeta_{\b 1}^{m}=-(\sqrt{-1})^m\zeta_{\b 1}^{m}.
\end{align*}
Now the claim follows at once from the explicit description of the Fourier transform in $\RR^2$ given in \S\ref{sss:FF} above.
\end{proof}

\begin{lemma}
\label{lem:FFdimn-1}
 For any $r,k,m\in\NN$ with $r\leq n-2$, $k\leq \min\{r,n-r-1\}$, and $m\geq2$,
\begin{align*}
f^{(n)}_{r,k,m} =  f^{(n-1)}_{r,k,m}.
\end{align*}
 Moreover,  if $n$ is even, then for any $l\geq2$ and  $m\geq 2$,
$$f_{l,l,m}^{(2l)}=-f_{l,l-1,m}^{(2l-1)}.$$
\end{lemma}

\begin{proof}
Identifying $V\cong V^*\cong\RR^n$, the inclusion $\iota\maps{\RR^{n-1}}{\RR^n}$ and the orthogonal projection $\pi\maps{\RR^n}{\RR^{n-1}}$ are dual to each other. Consequently, Theorems \ref{thm:pushforward}, \ref{thm:FFf}, and \ref{thm:pullback} imply, respectively,
\begin{align*}
f_{r,k,m}^{(n)}\phi_{n-r-1,k,m}^{(n-1)}=\pi_*\FF\phi_{r,k,m}^{(n)}=\FF\iota^*\phi_{r,k,m}^{(n)}=f_{r,k,m}^{(n-1)}\phi_{n-r-1,k,m}^{(n-1)}.
\end{align*}
 Similarly,
$$ - \frac 12 f_{l,l,m}^{(2l)} \phi^{(2l-1)}_{l-1,l-1,m} = \pi_* \FF\phi^{(2l)}_{l,l,m}  = \FF\iota^*\phi^{(2l)}_{l,l,m}= \frac 12 f_{l,l-1,m}^{(2l-1)} \phi^{(2l-1)}_{l-1,l-1,m}.  $$
\end{proof}

\begin{proof}[Proof of Theorem \ref{thm:FFhwv}]
We will first show \eqref{eq:FFhwv1}, using induction on $n$. For $n=2$, the claim was proven in Proposition \ref{pro:FFdim2} above. Assume thus \eqref{eq:FFhwv1} holds for $n-1\geq 2$. First, by  the first part of Lemma \ref{lem:FFdimn-1} we have
\begin{align}
\label{eq:fspecial}
f_{r,k,m}^{(n)}=(-1)^{k-1}(\sqrt{-1})^m
\end{align}
for $ r\leq n-2$ and $ k\leq\min\{r,n-1-r\}$. Second, using the relation
\begin{align*}
f^{(n)}_{r,k,m}f^{(n)}_{n-r,k,m}=(-1)^m
\end{align*}
which follows directly from Theorem \ref{thm:AFT} (c), we extend the validity of \eqref{eq:fspecial} to  $r=n-1$ and $k\leq\min\{r,n-r\}$; indeed, in this case one necessarily has $k=1$ but $f_{1,1,m}^{(n)}$ was already covered by \eqref{eq:fspecial}. Similarly, if $k=n-r<r$ for some $r\leq n-2$, one has $n-r\leq n-2$ and $k\leq \min\{n-r,n-(n-r)-1\}$ and so we can compute $f^{(n)}_{r,k,m}=(-1)^m\big(f_{n-r,k,m}^{(n)}\big)^{-1}$. Finally, $k=n-r=r$ implies $n=2l$ and $r=k=l$, and  Lemma \ref{lem:FFdimn-1} therefore finishes the proof of \eqref{eq:FFhwv1}.

Finally, assume $n=2l$ and consider the reflection $R\in\OO(n)$ from the proof of Corollary \ref{cor:RuminD}. Since $\det R=-1$, after the identifications we have $\FF\circ R=-R\circ\FF$ and \eqref{eq:FFhwv2} thus follows from \eqref{eq:FFhwv1}.
\end{proof}

\section{The Lefschetz operator $\Lambda$}

\label{s:Lefschetz}

In the section to follow, the action of the operator $\Lambda$ defined by \eqref{eq:Lambda} will be determined explicitly on the highest weight vectors. This result will later be used to prove that $\Lambda$ satisfies the hard Lefschetz theorem and Hodge--Riemann relations. Our starting point is Proposition \ref{pro:convLefschetz}. 

\begin{theorem}
\label{thm:actionLambda}
 For any $r,k,m\in\NN$ with $r\leq n-1$, $k\leq \min\{r,n-r\}$, and $m\geq2$,
\begin{align*}
\Lambda\phi_{r,k,m}=\begin{cases}(n-r-k+1)\frac{v_{n+m-r-1}}{v_{n+m-r-2}}\phi_{r-1,k,m} & \text{if }k<r,\\[1ex] 0 & \text{if }k=r;\end{cases}
\end{align*}
If $n=2l$, then for $m\geq 2$,
\begin{align*}
\Lambda\phi_{l,-l,m}=0.
\end{align*}
\end{theorem}

\begin{proposition}
\label{pro:LM00}
For any $r,k,m\in\NN$ with $r\leq n-1$, $k\leq \min\{r,n-r\}$, and $m\geq2$,
\begin{align}
\label{eq:LM01}
d\zeta_{\b 1}\sigma_{r,k}&=0,\\
\label{eq:LM02}
d\zeta_{\b 1}d\theta_{r,k}&=-\zeta_{\b1}d\sigma_{r,k},\\
\label{eq:LM03}
\zeta_{\b 1}^{m-2}d\sigma_{r,k}&=\frac{n-r-k+1}{n+m-r-1} d\omega_{r-1,k,m} \quad\text{if }k<r,\text{ and}\\
\label{eq:LM04}
d\sigma_{k,k}&=0.
\end{align}
\end{proposition}

\begin{proof}
\eqref{eq:LM01} is obvious. \eqref{eq:LM02} follows easily by differentiating \eqref{eq:RuminD4} and using \eqref{eq:LM01}. To show \eqref{eq:LM03}, observe that $k<r$ implies $r\geq 2$. Then on the one hand we have $d\sigma_{r,k}=(n-r-k+1)\delta_{r-1,k}$ by \eqref{eq:RuminD03} and \eqref{eq:RuminD04}; on the other hand, by \eqref{eq:RuminD2} and \eqref{eq:RuminD3}, $d\omega_{r-1,k,m}=(n+m-r-1)\zeta_{\b 1}^{m-2}\delta_{r-1,k}$. Finally, \eqref{eq:LM04} follows at once from \eqref{eq:RuminD04}.
\end{proof}

\begin{lemma}
\label{lem:LieT}
 For any $r,k,m\in\NN$ with $r\leq n-1$, $k\leq \min\{r,n-r\}$, and $m\geq2$,
\begin{align*}
\calL_T(D\omega_{r,k,m})=\begin{cases}\frac{(n+m-r-2)(n-r-k+1)}{n+m-r-1}D\omega_{r-1,k,m}&\text{if }k<r,\\[1ex]
0&\text{if }k=r,\end{cases}
\end{align*}
where $T= \sum_{i\in\calI} \zeta_i \pder{z_i}$ is the Reeb vector field.
\end{lemma}

\begin{proof}
First of all, from Cartan's magic formula and the fact that $d\circ D=0$ we infer $\calL_T(D\omega_{r,k,m})=d(\iota_T D\omega_{r,k,m})$. Further, observe that
\begin{align*}
\iota_T\sigma_{r,k}=0
\end{align*}
and
\begin{align*}
\iota_T\tau_{r,k}=(-1)^{n-r}\b{\zeta_K}(\b{d\zeta_K})^{[k-1]}(d\zeta_J)^{[n-r-k+1]}\zeta_J(dz_J)^{[r-2]}=(-1)^{k+1}\theta_{r-1,k};
\end{align*}
note that in view of \eqref{eq:RuminD01} this identity continues to hold for $k=r$ if we set $\theta_{k-1,k}=0$. Using this together with \eqref{eq:RuminD5} and the second part of Theorem \ref{thm:RuminD}, we obtain
\begin{align*}
&\iota_TD\omega_{r,k,m}\\
&\quad=c_{r,m}\zeta_{\b 1}^{m-2}\big[(n-r-k+1)\theta_{r-1,k}\alpha\\
&\quad\qquad\qquad\qquad+(-1)^{n+1}(m+k-1)\sigma_{r,k}+(-1)^{n+k}(n-r-k+1)\tau_{r,k}\big]\\
&\quad=c_{r,m}\zeta_{\b 1}^{m-2}\left[(n-r-k+1)\theta_{r-1,k}\alpha+(-1)^{n+1}(m-1)\sigma_{r,k}+(-1)^{n+1}d\theta_{r,k}\right],
\end{align*}
where
\begin{align*}
c_{r,m}=(-1)^{n+1}(n+m-r-2).
\end{align*}
Consequently, using Proposition \ref{pro:LM00}, we have
\begin{align*}
d(\iota_TD\omega_{k,k,m})=0
\end{align*}
and, for $k<r$,
\begin{align*}
&d(\iota_TD\omega_{r,k,m})\\
&\quad=(-1)^{n+1}(n+m-r-2)(n-r-k+1)\,d\left(\zeta_{\b 1}^{m-2}\theta_{r-1,k}\alpha\right)\\
&\qquad+(m-1)(n+m-r-2)\zeta_{\b1}^{m-2}d\sigma_{r,k}\\
&\qquad-(m-2)(n+m-r-2)\zeta_{\b1}^{m-2}d\sigma_{r,k}\\
&\quad=\frac{(n+m-r-2)(n-r-k+1)}{n+m-r-1}\, d\left(c_{r-1,m}\zeta_{\b 1}^{m-2}\theta_{r-1,k}\alpha+\omega_{r-1,k,m}\right).
\end{align*}
Now the claim follows from the first part of Theorem \ref{thm:RuminD}.
\end{proof}

If $n=2l$, then the analogous statement for the remaining highest weight vectors can be further deduced from Lemma \ref{lem:LieT}, just like Corollary \ref{cor:RuminD} was deduced from Theorem \ref{thm:RuminD}; namely,

\begin{corollary}
\label{cor:LieT}
For any $m\geq2$ one has
\begin{align*}
\calL_T(D\omega_{l,-l,m})=0.
\end{align*}
\end{corollary}

\begin{proof}
The reflection $R$ considered in the proof of Corollary \ref{cor:RuminD} clearly commutes with $\calL_T$ and hence the claim follows at once from the second part of Lemma \ref{lem:LieT}.
\end{proof}

\begin{proof}[Proof of Theorem \ref{thm:actionLambda}]
The operator $D$ commutes with  pullback along contactomorphisms, in particular with $\calL_T$. Consequently, if $k<r$, the first part of Lemma \ref{lem:LieT} together with \eqref{eq:volumeBall} and Theorem \ref{thm:kernel} implies that the form
\begin{align*}
\frac1{s_{n+m-r-3}}\calL_T\omega_{r,k,m}-\frac{(n-r-k+1)v_{n+m-r-1}}{s_{n+m-r-2}v_{n+m-r-2}}\omega_{r-1,k,m}
\end{align*}
defines the zero valuation and hence the claim follows from Proposition \ref{pro:convLefschetz}. The remaining cases follow similarly from the second part of Lemma \ref{lem:LieT} and Corollary \eqref{cor:LieT}, respectively.
\end{proof}

\section{Hard Lefschetz theorem and Hodge--Riemann relations}

\label{s:HLHR}

Finally, using the description of the operator
\begin{align*}
 \Lambda\phi=2\mu_{n-1}*\phi
\end{align*}
obtained in the previous section, we will give here a new proof of the hard Lefschetz theorem and prove the Hodge--Riemann relations for valuations, for Euclidean balls as reference bodies. Since our proofs rely profusely on the Alesker--Bernig--Schuster decomposition theorem, let us recall that the seeming dependence of this result on the hard Lefschetz theorem can be easily removed, see Remark \ref{re:HLinABS}.

\begin{theorem}
\label{thm:HLHR}
Let $0\leq r\leq \lfloor\frac n2\rfloor$. Then the following properties hold:
\begin{enuma}
\item \textnormal{Hard Lefschetz theorem}. The map $\Val_{n-r}^\infty(\RR^n)\rightarrow\Val_r^\infty(\RR^n)$ given by
\begin{align*}
\phi\mapsto\phi*(\mu_{n-1})^{n-2r}
\end{align*}
is an isomorphism of topological vector spaces.
\item \textnormal{Hodge--Riemann relations}. The sesquilinear form
\begin{align*}
Q(\phi,\psi)=(-1)^r\,\phi*\b\psi*(\mu_{n-1})^{n-2r}
\end{align*}
is positive definite on
\begin{align*}
P_{n-r}=\left\{\phi\in\Val^\infty_{n-r}(\RR^n)\mid \phi*\left(\mu_{n-1}\right)^{n-2r+1}=0\right\}.
\end{align*}
\end{enuma}
\end{theorem}

\begin{proof}
If $r=0$, both items follow at once from \eqref{eq:Lambdaton} so we will assume $r>0$.
\begin{enuma}
\item 
Injectivity follows from Theorem \ref{thm:ABS},  Theorem~\ref{thm:actionLambda}, and equation \eqref{eq:Lambdaton}. To prove surjectivity, it suffices to show that the map
\begin{align*}
S_r=\Lambda^{n-2r}\circ\FF\maps{\Val_r^\infty(\RR^n)}{\Val_r^\infty(\RR^n)}
\end{align*}
is surjective. To this end, observe that $S_r$ commutes with $\SO(n)$ and let $e_{r,k,m}\in\CC$, for $k\leq \min\{r,n-r\}$ and $m\geq2$, denote its eigenvalues, i.e.,
\begin{align*}
S_r\phi_{r,k,m}=e_{r,k,m}\phi_{r,k,m}.
\end{align*}
Since $v_m/v_{m-1} = O(m^{1/2})$  as $m\to\infty$, Theorems \ref{thm:FFhwv} and \ref{thm:actionLambda} imply 	$|e_{r,k,m}|= O(m^{n/2-r})$. Consequently, by Theorem \ref{thm:decay}, given $\phi\in\Val_r^\infty(\RR^n)$ the sum
\begin{align*}
\sum_{\lambda\in \Lambda_r} e_{r,k,m}^{-1} \pi_{\lambda} (\phi)
\end{align*}
converges to some $\phi_0\in \Val^{\infty}_r(\RR^n)$. By linearity and continuity of $S_r$ and by Theorem \ref{thm:H-Ch}, this valuation satisfies $S_r\phi_0=\phi$ which proves $S_r$ is surjective. Using the open mapping theorem, we conclude that $\Lambda^{n-2r}$ is an isomorphism of topological vector spaces.
\item  
Clearly the form $Q$ and the subspace $P_{n-r}\subset\Val_{n-r}^\infty(\RR^n)$ are both $\OO(n)$ invariant for so is the valuation $\mu_{n-1}$ and the convolution commutes with the action of $\OO(n)$. By Theorems \ref{thm:ABS} and \ref{thm:actionLambda}, $P_{n-r}$ decomposes into $\SO(n)$-types given by the following highest weights:
\begin{align*}
\Pi_r=\{(\lambda_1,\dots,\lambda_l)\in\Lambda_r\mid \lambda_r\neq 0\},
\end{align*}
each occurring with multiplicity one. Using Proposition \ref{pro:Schur}, by linearity and continuity it thus suffices to show $
Q(\phi_\lambda,\phi_\lambda)>0$ for any $\lambda\in\Pi_r$ and some vector $\phi_\lambda\in P_{n-r}$ from the irreducible subspace corresponding to $\lambda$. Fix $m\geq2$. First,
\begin{align*}
Q(\phi_{n-r,r,m},\phi_{n-r,r,m})>0
\end{align*}
by Theorem \ref{thm:actionLambda} and  Theorem \ref{thm:PD}. Second, if $n=2l$ and $r=l$, applying the reflection $R\in\OO(n)$ from the proof of Corollary \ref{cor:RuminD} and using invariance of $Q$, we obtain
\begin{align*}
Q(\phi_{l,-l,m},\phi_{l,-l,m})>0.
\end{align*}
\end{enuma}
\end{proof}

As pointed out to the authors by R.~van~Handel, combined with the authors' previous result \cite[Theorem 1.1]{KotrbatyWannerer:MixedHR}, the special case of the Hodge--Riemann relations proven above implies a slightly more general statement. This provides further evidence in favour of Conjecture \ref{con}.

\begin{proof}[Proof of Corollary~\ref{cor:D}] Note that $\mu_{D^{n}}$ is positively proportional to $\mu_{n-1}$. Let $\phi\in\Val_{n-2}^\infty(\RR^n)$ and suppose that 
	$$ \phi* \mu_{C_0} * (\mu_{n-1})^{n-4}=0.$$
 It was shown by the authors in \cite[Theorem 1.1]{KotrbatyWannerer:MixedHR} that Conjecture \ref{con} is true if $r=1$. The hard Lefschetz theorem  for $r=1$ thus  implies that there exists (a unique) $\eta\in\Val_{n-1}^\infty(\RR^n)$ with
\begin{align}
	\label{eq:Lambdan-3}
	(\mu_{n-1})^{n-3}*\phi=\mu_{C_0}*(\mu_{n-1})^{n-3}*\eta.
\end{align}
Clearly,  $(\mu_{C_0})^{2}*(\mu_{n-1})^{n-3}*\eta=0$ and hence the Hodge--Riemann relations for $r=1$ yield
\begin{align*}
	\eta*\b\eta*(\mu_{C_0})^{2}*(\mu_{n-1})^{n-4}\leq 0.
\end{align*}
\eqref{eq:Lambdan-3} further implies  $\phi-\mu_{C_0}*\eta\in P_{n-2}$. Theorem \ref{thm:HLHR} thus yields 
\begin{align*}
	(\phi-\mu_{C_0}*\eta)*\left(\b\phi-\mu_{C_0}*\b\eta\right)*(\mu_{n-1})^{n-4}\geq0.
\end{align*}
Combining the two inequalities readily gives $\phi* \b\phi *(\mu_{n-1})^{n-4}\geq 0$.

	If equality holds in the latter inequality, then equality also holds in the previous two inequalities. The equality conditions there  imply, respectively, $\phi-\mu_{C_0}*\eta=0$ and $\eta=0$ and hence $\phi=0$. 
\end{proof}

\bibliographystyle{abbrv}
\bibliography{ref_papers,ref_books}

\begin{bibdiv}
\begin{biblist}

\bib{AbardiaBernig:AdditiveFlag}{article}{
      author={Abardia-Ev\'{e}quoz, Judit},
      author={Bernig, Andreas},
       title={Additive kinematic formulas for flag area measures},
        date={2021},
        ISSN={0025-5831},
     journal={Math. Ann.},
      volume={381},
      number={3-4},
       pages={1615\ndash 1652},
         url={https://doi.org/10.1007/s00208-021-02264-w},
      review={\MR{4333426}},
}

\bib{Alesker:Irreducibility}{article}{
      author={Alesker, S.},
       title={Description of translation invariant valuations on convex sets
  with solution of {P}. {M}c{M}ullen's conjecture},
        date={2001},
        ISSN={1016-443X},
     journal={Geom. Funct. Anal.},
      volume={11},
      number={2},
       pages={244\ndash 272},
         url={https://doi.org/10.1007/PL00001675},
      review={\MR{1837364}},
}

\bib{Alesker:HLComplex}{article}{
      author={Alesker, S.},
       title={Hard {L}efschetz theorem for valuations, complex integral
  geometry, and unitarily invariant valuations},
        date={2003},
        ISSN={0022-040X},
     journal={J. Differential Geom.},
      volume={63},
      number={1},
       pages={63\ndash 95},
         url={http://projecteuclid.org/euclid.jdg/1080835658},
      review={\MR{2015260}},
}

\bib{Alesker:Product}{article}{
      author={Alesker, S.},
       title={The multiplicative structure on continuous polynomial
  valuations},
        date={2004},
        ISSN={1016-443X},
     journal={Geom. Funct. Anal.},
      volume={14},
      number={1},
       pages={1\ndash 26},
         url={https://doi.org/10.1007/s00039-004-0450-2},
      review={\MR{2053598}},
}

\bib{Alesker:Fourier}{article}{
      author={Alesker, S.},
       title={A {F}ourier-type transform on translation-invariant valuations on
  convex sets},
        date={2011},
        ISSN={0021-2172},
     journal={Israel J. Math.},
      volume={181},
       pages={189\ndash 294},
         url={https://doi.org/10.1007/s11856-011-0008-6},
      review={\MR{2773042}},
}

\bib{Alesker:Kent}{book}{
      author={Alesker, S.},
       title={Introduction to the theory of valuations},
      series={CBMS Regional Conference Series in Mathematics},
   publisher={American Mathematical Society, Providence, RI},
        date={2018},
      volume={126},
        ISBN={978-1-4704-4359-7},
      review={\MR{3820854}},
}

\bib{Alesker:Kotrbaty}{article}{
      author={Alesker, S.},
       title={Kotrbat\'y's theorem on valuations and geometric inequalities for
  convex bodies},
        date={2021},
     journal={Israel J. Math.},
       pages={to appear},
}

\bib{AleskerBernig:Product}{article}{
      author={Alesker, S.},
      author={Bernig, Andreas},
       title={The product on smooth and generalized valuations},
        date={2012},
        ISSN={0002-9327},
     journal={Amer. J. Math.},
      volume={134},
      number={2},
       pages={507\ndash 560},
         url={https://doi.org/10.1353/ajm.2012.0011},
      review={\MR{2905004}},
}

\bib{ABS:Harmonic}{article}{
      author={Alesker, S.},
      author={Bernig, Andreas},
      author={Schuster, Franz~E.},
       title={Harmonic analysis of translation invariant valuations},
        date={2011},
        ISSN={1016-443X},
     journal={Geom. Funct. Anal.},
      volume={21},
      number={4},
       pages={751\ndash 773},
         url={https://doi.org/10.1007/s00039-011-0125-8},
      review={\MR{2827009}},
}

\bib{Alesker:VMfdsIII}{article}{
      author={Alesker, S.},
      author={Fu, J. H.~G.},
       title={Theory of valuations on manifolds. {III}. {M}ultiplicative
  structure in the general case},
        date={2008},
        ISSN={0002-9947},
     journal={Trans. Amer. Math. Soc.},
      volume={360},
      number={4},
       pages={1951\ndash 1981},
         url={https://doi.org/10.1090/S0002-9947-07-04489-3},
      review={\MR{2366970}},
}

\bib{AleskerFu:Barcelona}{book}{
      author={Alesker, S.},
      author={Fu, J. H.~G.},
       title={Integral geometry and valuations},
      series={Advanced Courses in Mathematics. CRM Barcelona},
   publisher={Birkh\"{a}user/Springer, Basel},
        date={2014},
        ISBN={978-3-0348-0873-6; 978-3-0348-0874-3},
      review={\MR{3380549}},
}

\bib{Bernig:Formula}{article}{
      author={Bernig, Andreas},
       title={A product formula for valuations on manifolds with applications
  to the integral geometry of the quaternionic line},
        date={2009},
        ISSN={0010-2571},
     journal={Comment. Math. Helv.},
      volume={84},
      number={1},
       pages={1\ndash 19},
         url={https://doi.org/10.4171/CMH/150},
      review={\MR{2466073}},
}

\bib{BernigBroecker:Rumin}{article}{
      author={Bernig, Andreas},
      author={Br\"{o}cker, Ludwig},
       title={Valuations on manifolds and {R}umin cohomology},
        date={2007},
        ISSN={0022-040X},
     journal={J. Differential Geom.},
      volume={75},
      number={3},
       pages={433\ndash 457},
         url={http://projecteuclid.org/euclid.jdg/1175266280},
      review={\MR{2301452}},
}

\bib{BernigFu:Convolution}{article}{
      author={Bernig, Andreas},
      author={Fu, Joseph H.~G.},
       title={Convolution of convex valuations},
        date={2006},
        ISSN={0046-5755},
     journal={Geom. Dedicata},
      volume={123},
       pages={153\ndash 169},
         url={https://doi.org/10.1007/s10711-006-9115-7},
      review={\MR{2299731}},
}

\bib{BernigFu:Hig}{article}{
      author={Bernig, Andreas},
      author={Fu, Joseph H.~G.},
       title={Hermitian integral geometry},
        date={2011},
        ISSN={0003-486X},
     journal={Ann. of Math. (2)},
      volume={173},
      number={2},
       pages={907\ndash 945},
         url={https://doi.org/10.4007/annals.2011.173.2.7},
      review={\MR{2776365}},
}

\bib{BFS}{article}{
      author={Bernig, Andreas},
      author={Fu, Joseph H.~G.},
      author={Solanes, Gil},
       title={Integral geometry of complex space forms},
        date={2014},
        ISSN={1016-443X},
     journal={Geom. Funct. Anal.},
      volume={24},
      number={2},
       pages={403\ndash 492},
         url={https://doi.org/10.1007/s00039-014-0251-1},
      review={\MR{3192033}},
}

\bib{BernigHug:Tensor}{article}{
      author={Bernig, Andreas},
      author={Hug, Daniel},
       title={Kinematic formulas for tensor valuations},
        date={2018},
        ISSN={0075-4102},
     journal={J. Reine Angew. Math.},
      volume={736},
       pages={141\ndash 191},
         url={https://doi.org/10.1515/crelle-2015-0023},
      review={\MR{3769988}},
}

\bib{BernigSolanes:Kinematic}{article}{
      author={Bernig, Andreas},
      author={Solanes, Gil},
       title={Kinematic formulas on the quaternionic plane},
        date={2017},
        ISSN={0024-6115},
     journal={Proc. Lond. Math. Soc. (3)},
      volume={115},
      number={4},
       pages={725\ndash 762},
         url={https://doi.org/10.1112/plms.12050},
      review={\MR{3716941}},
}

\bib{Blaschke:Integralgeometrie}{book}{
      author={Blaschke, Wilhelm},
       title={Vorlesungen \"{u}ber {I}ntegralgeometrie},
   publisher={Deutscher Verlag der Wissenschaften, Berlin},
        date={1955},
        note={3te Aufl},
      review={\MR{0076373}},
}

\bib{BDS:Dimension}{article}{
      author={B\"{o}r\"{o}czky, K.~J.},
      author={Domokos, M.},
      author={Solanes, G.},
       title={Dimension of the space of unitary equivariant translation
  invariant tensor valuations},
        date={2021},
        ISSN={0022-1236},
     journal={J. Funct. Anal.},
      volume={280},
      number={4},
       pages={Paper No. 108862, 18},
         url={https://doi.org/10.1016/j.jfa.2020.108862},
      review={\MR{4181164}},
}

\bib{Chern:Kinematic}{article}{
      author={Chern, Shiing-Shen},
       title={On the kinematic formula in the {E}uclidean space of {$n$}
  dimensions},
        date={1952},
        ISSN={0002-9327},
     journal={Amer. J. Math.},
      volume={74},
       pages={227\ndash 236},
         url={https://doi.org/10.2307/2372080},
      review={\MR{47353}},
}

\bib{CLM:Hadwiger}{article}{
      author={Colesanti, Andrea},
      author={Ludwig, Monika},
      author={Mussnig, Fabian},
       title={The {H}adwiger theorem on convex functions. {I}},
        date={2020},
     journal={preprint},
      eprint={arXiv:2009.03702},
}

\bib{CLM:Homogeneous}{article}{
      author={Colesanti, Andrea},
      author={Ludwig, Monika},
      author={Mussnig, Fabian},
       title={A homogeneous decomposition theorem for valuations on convex
  functions},
        date={2020},
        ISSN={0022-1236},
     journal={J. Funct. Anal.},
      volume={279},
      number={5},
       pages={108573, 25},
         url={https://doi.org/10.1016/j.jfa.2020.108573},
      review={\MR{4097279}},
}

\bib{deRham}{book}{
      author={de~Rham, Georges},
       title={Differentiable manifolds},
      series={Grundlehren der mathematischen Wissenschaften [Fundamental
  Principles of Mathematical Sciences]},
   publisher={Springer-Verlag, Berlin},
        date={1984},
      volume={266},
        ISBN={3-540-13463-8},
         url={https://doi.org/10.1007/978-3-642-61752-2},
        note={Forms, currents, harmonic forms, Translated from the French by F.
  R. Smith, With an introduction by S. S. Chern},
      review={\MR{760450}},
}

\bib{Faifman:Crofton}{article}{
      author={Faifman, Dmitry},
       title={Crofton formulas and indefinite signature},
        date={2017},
        ISSN={1016-443X},
     journal={Geom. Funct. Anal.},
      volume={27},
      number={3},
       pages={489\ndash 540},
         url={https://doi.org/10.1007/s00039-017-0406-y},
      review={\MR{3655955}},
}

\bib{Faifman:Contact}{article}{
      author={Faifman, Dmitry},
       title={Contact integral geometry and the {H}eisenberg algebra},
        date={2019},
        ISSN={1465-3060},
     journal={Geom. Topol.},
      volume={23},
      number={6},
       pages={3041\ndash 3110},
         url={https://doi.org/10.2140/gt.2019.23.3041},
      review={\MR{4039185}},
}

\bib{Federer:CurvatureMeasures}{article}{
      author={Federer, Herbert},
       title={Curvature measures},
        date={1959},
        ISSN={0002-9947},
     journal={Trans. Amer. Math. Soc.},
      volume={93},
       pages={418\ndash 491},
         url={https://doi.org/10.2307/1993504},
      review={\MR{110078}},
}

\bib{Fu:Unitary}{article}{
      author={Fu, Joseph H.~G.},
       title={Structure of the unitary valuation algebra},
        date={2006},
        ISSN={0022-040X},
     journal={J. Differential Geom.},
      volume={72},
      number={3},
       pages={509\ndash 533},
         url={http://projecteuclid.org/euclid.jdg/1143593748},
      review={\MR{2219942}},
}

\bib{Fu:IGRegularity}{incollection}{
      author={Fu, Joseph H.~G.},
       title={Integral geometric regularity},
        date={2017},
   booktitle={Tensor valuations and their applications in stochastic geometry
  and imaging},
      series={Lecture Notes in Math.},
      volume={2177},
   publisher={Springer, Cham},
       pages={261\ndash 299},
      review={\MR{3702376}},
}

\bib{GoodeyHugWeil:AreaMeas}{article}{
      author={Goodey, Paul},
      author={Hug, Daniel},
      author={Weil, Wolfgang},
       title={Kinematic formulas for area measures},
        date={2017},
        ISSN={0022-2518},
     journal={Indiana Univ. Math. J.},
      volume={66},
      number={3},
       pages={997\ndash 1018},
         url={https://doi.org/10.1512/iumj.2017.66.6047},
      review={\MR{3663334}},
}

\bib{Gray:TubesChernClasses}{article}{
      author={Gray, Alfred},
       title={Volumes of tubes about {K}\"{a}hler submanifolds expressed in
  terms of {C}hern classes},
        date={1984},
        ISSN={0025-5645},
     journal={J. Math. Soc. Japan},
      volume={36},
      number={1},
       pages={23\ndash 35},
         url={https://doi.org/10.2969/jmsj/03610023},
      review={\MR{723590}},
}

\bib{Gray:Tubes}{book}{
      author={Gray, Alfred},
       title={Tubes},
     edition={Second},
      series={Progress in Mathematics},
   publisher={Birkh\"{a}user Verlag, Basel},
        date={2004},
      volume={221},
        ISBN={3-7643-6907-8},
         url={https://doi.org/10.1007/978-3-0348-7966-8},
        note={With a preface by Vicente Miquel},
      review={\MR{2024928}},
}

\bib{Groemer:Harmonics}{book}{
      author={Groemer, H.},
       title={Geometric applications of {F}ourier series and spherical
  harmonics},
      series={Encyclopedia of Mathematics and its Applications},
   publisher={Cambridge University Press, Cambridge},
        date={1996},
      volume={61},
        ISBN={0-521-47318-7},
         url={https://doi.org/10.1017/CBO9780511530005},
      review={\MR{1412143}},
}

\bib{Gruber:CDG}{book}{
      author={Gruber, Peter~M.},
       title={Convex and discrete geometry},
      series={Grundlehren der mathematischen Wissenschaften [Fundamental
  Principles of Mathematical Sciences]},
   publisher={Springer, Berlin},
        date={2007},
      volume={336},
        ISBN={978-3-540-71132-2},
      review={\MR{2335496}},
}

\bib{Hadwiger:Vorlesungen}{book}{
      author={Hadwiger, H.},
       title={Vorlesungen \"{u}ber {I}nhalt, {O}berfl\"{a}che und
  {I}soperimetrie},
   publisher={Springer-Verlag, Berlin-G\"{o}ttingen-Heidelberg},
        date={1957},
      review={\MR{0102775}},
}

\bib{KlainRota}{book}{
      author={Klain, Daniel~A.},
      author={Rota, Gian-Carlo},
       title={Introduction to geometric probability},
      series={Lezioni Lincee. [Lincei Lectures]},
   publisher={Cambridge University Press, Cambridge},
        date={1997},
        ISBN={0-521-59362-X; 0-521-59654-8},
      review={\MR{1608265}},
}

\bib{Knapp:Overview}{book}{
      author={Knapp, Anthony~W.},
       title={Representation theory of semisimple groups},
      series={Princeton Landmarks in Mathematics},
   publisher={Princeton University Press, Princeton, NJ},
        date={2001},
        ISBN={0-691-09089-0},
        note={An overview based on examples, Reprint of the 1986 original},
      review={\MR{1880691}},
}

\bib{Knoerr:Smooth}{article}{
      author={Knoerr, Jonas},
       title={Smooth valuations on convex functions},
        date={2020},
     journal={preprint},
      eprint={arXiv:2006.12933},
}

\bib{Knoerr:Support}{article}{
      author={Knoerr, Jonas},
       title={The support of dually epi-translation invariant valuations on
  convex functions},
        date={2021},
        ISSN={0022-1236},
     journal={J. Funct. Anal.},
      volume={281},
      number={5},
       pages={Paper No. 109059, 52},
         url={https://doi.org/10.1016/j.jfa.2021.109059},
      review={\MR{4252807}},
}

\bib{Kotrbaty:HR}{article}{
      author={Kotrbat\'{y}, Jan},
       title={On {H}odge-{R}iemann relations for translation-invariant
  valuations},
        date={2021},
        ISSN={0001-8708},
     journal={Adv. Math.},
      volume={390},
       pages={Paper No. 107914, 28},
         url={https://doi.org/10.1016/j.aim.2021.107914},
      review={\MR{4295088}},
}

\bib{KotrbatyWannerer:MixedHR}{article}{
      author={Kotrbat\'{y}, Jan},
      author={Wannerer, Thomas},
       title={On mixed {Hodge-Riemann} relations for translation-invariant
  valuations and {Aleksandrov-Fenchel} inequalities},
        date={2021},
     journal={Commun. Contemp. Math},
       pages={to appear},
}

\bib{Lee:SM}{book}{
      author={Lee, John~M.},
       title={Introduction to smooth manifolds},
     edition={Second},
      series={Graduate Texts in Mathematics},
   publisher={Springer, New York},
        date={2013},
      volume={218},
        ISBN={978-1-4419-9981-8},
      review={\MR{2954043}},
}

\bib{McMullen:EulerType}{article}{
      author={McMullen, P.},
       title={Valuations and {E}uler-type relations on certain classes of
  convex polytopes},
        date={1977},
        ISSN={0024-6115},
     journal={Proc. London Math. Soc. (3)},
      volume={35},
      number={1},
       pages={113\ndash 135},
         url={https://doi.org/10.1112/plms/s3-35.1.113},
      review={\MR{448239}},
}

\bib{McMullen:PolytopeAlgebra}{article}{
      author={McMullen, Peter},
       title={The polytope algebra},
        date={1989},
        ISSN={0001-8708},
     journal={Adv. Math.},
      volume={78},
      number={1},
       pages={76\ndash 130},
         url={https://doi.org/10.1016/0001-8708(89)90029-7},
      review={\MR{1021549}},
}

\bib{McMullen:SimplePolytopes}{article}{
      author={McMullen, Peter},
       title={On simple polytopes},
        date={1993},
        ISSN={0020-9910},
     journal={Invent. Math.},
      volume={113},
      number={2},
       pages={419\ndash 444},
         url={https://doi.org/10.1007/BF01244313},
      review={\MR{1228132}},
}

\bib{Rumin:Contact}{article}{
      author={Rumin, Michel},
       title={Formes diff\'{e}rentielles sur les vari\'{e}t\'{e}s de contact},
        date={1994},
        ISSN={0022-040X},
     journal={J. Differential Geom.},
      volume={39},
      number={2},
       pages={281\ndash 330},
         url={http://projecteuclid.org/euclid.jdg/1214454873},
      review={\MR{1267892}},
}

\bib{Saienko}{article}{
      author={Saienko, Mykhailo},
       title={Characterisation of valuations and curvature measures in
  {E}uclidean spaces},
     journal={Israel J. Math.},
       pages={to appear},
}

\bib{Schneider:BM}{book}{
      author={Schneider, Rolf},
       title={Convex bodies: the {B}runn-{M}inkowski theory},
     edition={expanded},
      series={Encyclopedia of Mathematics and its Applications},
   publisher={Cambridge University Press, Cambridge},
        date={2014},
      volume={151},
        ISBN={978-1-107-60101-7},
      review={\MR{3155183}},
}

\bib{Sugiura:Fourier}{article}{
      author={Sugiura, Mitsuo},
       title={Fourier series of smooth functions on compact {L}ie groups},
        date={1971},
        ISSN={0388-0699},
     journal={Osaka Math. J.},
      volume={8},
       pages={33\ndash 47},
         url={http://projecteuclid.org/euclid.ojm/1200693127},
      review={\MR{294571}},
}

\bib{Timorin:Analogue}{article}{
      author={Timorin, V.~A.},
       title={An analogue of the {H}odge-{R}iemann relations for simple convex
  polyhedra},
        date={1999},
        ISSN={0042-1316},
     journal={Uspekhi Mat. Nauk},
      volume={54},
      number={2(326)},
       pages={113\ndash 162},
         url={https://doi.org/10.1070/rm1999v054n02ABEH000134},
      review={\MR{1711255}},
}

\bib{vanHandel:Shepard}{article}{
      author={van Handel, Ramon},
       title={{Shephard's inequalities, Hodge-Riemann relations, and a
  conjecture of Fedotov}},
        date={2021},
     journal={preprint},
      eprint={arXiv:2109.05169},
}

\bib{Wannerer:UnitaryAreaMeasures}{article}{
      author={Wannerer, Thomas},
       title={Integral geometry of unitary area measures},
        date={2014},
        ISSN={0001-8708},
     journal={Adv. Math.},
      volume={263},
       pages={1\ndash 44},
         url={https://doi.org/10.1016/j.aim.2014.06.005},
      review={\MR{3239133}},
}

\bib{Wannerer:Extendability}{article}{
      author={Wannerer, Thomas},
       title={On the extendability by continuity of angular valuations on
  polytopes},
        date={2020},
        ISSN={0022-1236},
     journal={J. Funct. Anal.},
      volume={279},
      number={8},
       pages={108665, 25},
         url={https://doi.org/10.1016/j.jfa.2020.108665},
      review={\MR{4109090}},
}

\bib{Warner:HarmonicI}{book}{
      author={Warner, Garth},
       title={Harmonic analysis on semi-simple {L}ie groups. {I}},
      series={Die Grundlehren der mathematischen Wissenschaften, Band 188},
   publisher={Springer-Verlag, New York-Heidelberg},
        date={1972},
      review={\MR{0498999}},
}

\end{biblist}
\end{bibdiv}

\end{document}